\documentclass[11pt,a4paper]{amsart}
\usepackage{amsmath,amsthm,amsfonts,amssymb,graphicx,psfrag,dsfont,a4wide}
\usepackage{color}

\usepackage{hyperref}
\usepackage[alphabetic,initials]{amsrefs}

\psfrag{(M,xi)}{$(M,\xi)$}
\psfrag{(M',xi')}{$(M',\xi')$}
\psfrag{N(T), T in I0}{$\nN(T)$, $T \subset \iI_0$}
\psfrag{N(pM0)}{$\nN(\p M_0)$}
\psfrag{N(B)}{$\nN(B)$}
\psfrag{one}{$1$}
\psfrag{-one}{$-1$}
\psfrag{r}{$r$}
\psfrag{r'}{$r'$}
\psfrag{-r}{$-r$}
\psfrag{-r'}{$-r'$}
\psfrag{tau=1}{$\tau=1$}
\psfrag{tau=-1}{$\tau=-1$}
\psfrag{tau=0}{$\tau=0$}
\psfrag{rho}{$\rho$}
\psfrag{t}{$t$}
\psfrag{Wj}{$\wW_j$}
\psfrag{UTj}{$\uU_{T_j}$}
\psfrag{1 x Tj}{$\{1\} \times T_j$}
\psfrag{[0,1] x N(Tj)}{$[0,1] \times \nN(T_j)$}
\psfrag{HTj}{$H_{T_j} = \p \overline{\uU}_{T_j}$}
\psfrag{M'flat}{$M'_{\text{flat}}$}
\psfrag{(M'con,xi')}{$(M'_{\text{convex}},\xi')$}
\psfrag{M0P}{$M_0^P$}
\psfrag{...}{$\cdots$}
\psfrag{T}{$T$}
\psfrag{(W',om)}{$(W',\omega)$}
\psfrag{D x S2}{$\DD \times S^2$}
\psfrag{S1 x S2}{$S^1 \times S^2$}
\psfrag{(W,om)}{$(W,\omega)$}
\psfrag{T1}{$T_1$}
\psfrag{T2}{$T_2$}
\psfrag{T3}{$T_3$}
\psfrag{T0}{$T_0$}
\psfrag{M0}{$M_0$}
\psfrag{M1}{$M_1$}
\psfrag{M2}{$M_2$}
\psfrag{M3}{$M_3$}
\psfrag{(T3,xi2)}{$(T_3,\xi_2)$}
\psfrag{(S1S2,xiOT)}{$(S^1 \times S^2,\xi_\OT)$}
\psfrag{(S1S2,xi0) c (S1S2,xi0)}{$(S^1 \times S^2,\xi_0) \sqcup (S^1 \times S^2,\xi_0)$}
\psfrag{T-}{$T_-$}
\psfrag{T+}{$T_+$}
\psfrag{M-}{$M_-$}
\psfrag{M+}{$M_+$}
\psfrag{M-'}{$M_-'$}
\psfrag{Z}{$Z$}
\psfrag{Z'}{$Z'$}
\psfrag{GT(M,xi)ge1}{$\GT(M,\xi) \ge 1$}
\psfrag{Sig-}{$\Sigma_-$}
\psfrag{Sig+}{$\Sigma_+$}
\psfrag{Gam}{$\Gamma$}
\psfrag{overtwisted}{overtwisted}

\theoremstyle{plain}
\newtheorem*{thmu}{Theorem}

\newtheorem{thm}{Theorem}
\newtheorem{thmp}{Theorem}
\newtheorem{prop}{Proposition}[section]
\newtheorem{conj}{Conjecture}
\newtheorem{cor}{Corollary}
\newtheorem*{coru}{Corollary}
\newtheorem{lemma}[prop]{Lemma}

\newtheorem{question}{Question}

\theoremstyle{definition}

\newtheorem{defn}[prop]{Definition}

\theoremstyle{remark}
\newtheorem{remark}[prop]{Remark}

\newcommand{\core}{^{\operatorname{K}}}

\newcommand{\dR}{{\operatorname{dR}}}
\newcommand{\ECH}{{\operatorname{ECH}}}

\newcommand{\im}{\operatorname{im}}

\renewcommand{\AA}{{\mathbb A}}
\newcommand{\CC}{{\mathbb C}}
\newcommand{\DD}{{\mathbb D}}

\newcommand{\NN}{{\mathbb N}}
\newcommand{\QQ}{{\mathbb Q}}
\newcommand{\RR}{{\mathbb R}}

\newcommand{\ZZ}{{\mathbb Z}}
\newcommand{\fF}{{\mathcal F}}

\newcommand{\iI}{{\mathcal I}}

\newcommand{\nN}{{\mathcal N}}

\newcommand{\uU}{{\mathcal U}}
\newcommand{\vV}{{\mathcal V}}
\newcommand{\wW}{{\mathcal W}}
\newcommand{\p}{\partial}
\newcommand{\Abstract}{{\operatorname{abs}}}

\newcommand{\defin}[1]{\textbf{#1}}

\newcommand{\GT}{\operatorname{GT}}

\newcommand{\PD}{\operatorname{PD}}

\newcommand{\Lie}{{\mathcal{L}}}
\newcommand{\OT}{{\operatorname{OT}}}

\newcommand{\handle}{\mathcal{H}}
\newcommand{\roundhandle}{\widehat{\mathcal{H}}}
\renewcommand{\core}{\mathcal{K}}
\newcommand{\cocore}{\mathcal{K}'}
\newcommand{\roundcore}{\widehat{\mathcal{K}}}
\newcommand{\roundcocore}{\widehat{\mathcal{K}}'}
\newcommand{\cob}{\curlyeqprec}
\newcommand{\ecob}{\prec}
\newcommand{\necob}{\nprec}

\numberwithin{equation}{section}

\title[Non-Exact Symplectic Cobordisms]{Non-Exact Symplectic Cobordisms Between Contact $3$-Manifolds}
\author{Chris Wendl}
\address{Department of Mathematics \\ 
University College London \\
Gower Street \\
London WC1E 6BT \\ 
United Kingdom}
\email{c.wendl@ucl.ac.uk}
\thanks{Research supported by an Alexander von Humboldt
Foundation fellowship.}
 

\begin{document}

\begin{abstract}
We show that the pre-order defined on the category of contact manifolds
by arbitrary symplectic cobordisms is considerably less rigid than
its counterparts for exact or Stein cobordisms: in particular,
we exhibit large new classes of contact $3$-manifolds which are symplectically
cobordant to something overtwisted, or to the tight $3$-sphere, or which
admit symplectic caps containing symplectically embedded spheres with
vanishing self-intersection.  These constructions imply new and simplified
proofs of several recent results involving fillability, planarity and
non-separating contact type embeddings.  The cobordisms are built from
symplectic handles of the form $\Sigma \times \DD$ and
$\Sigma \times [-1,1] \times S^1$, which have symplectic cores and can
be attached to contact
$3$-manifolds along sufficiently large neighborhoods of transverse links
and pre-Lagrangian tori.  
We also sketch a construction
of $J$-holomorphic foliations in these cobordisms and formulate a 
conjecture regarding maps induced on Embedded Contact Homology with twisted
coefficients.
\end{abstract}

\maketitle

\tableofcontents

\section{Introduction}
\label{sec:intro}

Many important notions and results in contact topology can be expressed in 
terms of symplectic cobordisms. For example, the existence of a symplectic 
cobordism between particular contact manifolds can be used to determine 
whether one of them is symplectically fillable,
assuming the filling properties of the other are already understood. 
Moreover, cobordisms
are intimately associated with various notions of surgery, 
e.g.~Weinstein \cite{Weinstein:surgery} defined a notion of symplectic
handle attachment in which handles with Lagrangian cores can be attached
along Legendrian spheres in a contact manifold $(M,\xi)$, giving a Stein
cobordism (see \cite{CieliebakEliashberg}) to the contact manifold obtained 
from $(M,\xi)$ by Legendrian surgery.  In dimension
three, Eliashberg \cite{Eliashberg:cap}, Gay \cite{Gay:GirouxTorsion} and 
Gay-Stipsicz \cite{GayStipsicz:cap} have defined various other
types of surgeries on contact manifolds that yield symplectic cobordisms 
which are not Stein.
In some cases, these cobordisms have also been shown to induce surprisingly well-behaved
morphisms for certain contact invariants, e.g.~in Heegaard Floer homology \cite{Baldwin:cap}. 
The cobordism of \cite{Gay:GirouxTorsion} and its precursor in \cite{Eliashberg:fillableTorus}
yielded breakthroughs in the study of symplectic
fillings, resulting in the proof that contact manifolds with Giroux torsion are not strongly
fillable. The latter result was recently reinterpreted by the author \cite{Wendl:openbook2}
in the wider context
of blown up summed open book decompositions, leading to an infinite hierarchy of more
general filling obstructions that can be detected via holomorphic curves.

The main purpose of the present article is to show that the various
seemingly unrelated constructions
of non-Stein cobordisms mentioned above are all special cases of a much more general phenomenon, which arises
naturally in the setting of blown up summed open books and produces non-exact symplectic
cobordisms in many situations where exact or Stein cobordisms are forbidden. The overall
pattern seems to be that while exact and Stein cobordisms are rigidly constrained by a variety
of symplectic obstructions (for instance in Symplectic Field Theory \cite{LatschevWendl} 
or Heegaard Floer homology \cite{Karakurt:Stein}), 
non-exact symplectic cobordisms are quite flexible: they tend to exist
whenever there is no obvious reason why they should not.

Our constructions are based on a new notion of generalized symplectic handles, which have
symplectic cores and co-cores and can be attached to contact $3$-manifolds along ``sufficiently
wide'' neighborhoods of transverse links and pre-Lagrangian tori.  The notion of a
``sufficiently wide'' neighborhood here is somewhat subtle, and indicates that
in contrast to Weinstein handle attachment and
many other forms of surgery, our surgery is not a truly \emph{local} operation---e.g.~the surgery
along a transverse knot requires a neighborhood of a size that is only guaranteed to exist
in certain situations, notably when the knot is one of several binding components in
a supporting open book.

We begin by stating in
\S\ref{sec:results} the essential definitions and explaining some existence results for non-exact symplectic
cobordisms that follow from the handle construction. As easy applications, these results
imply new and substantially simplified proofs of several recent theorems of the author and
collaborators on obstructions to symplectic fillings, contact type embeddings and embeddings
of partially planar domains. The main results concerning the handle construction itself will
be explained in \S\ref{sec:handles}, together with a simple application to Embedded Contact Homology and
a conjectured generalization. In \S\ref{sec:consequences} we discuss further applications and examples, providing a
unified framework for reproving several important previous results of Eliashberg, Gay, Etnyre
and others involving the existence of symplectic cobordisms. The hard work is then undertaken
in \S\ref{sec:details}, of which the first several sections construct the symplectic handles described in
\S\ref{sec:handles}, \S\ref{sec:proofs} completes the proofs of the results stated in the introduction, 
and \S\ref{sec:holomorphic} discusses
the construction of a holomorphic foliation in the cobordisms, which we expect should have
interesting applications in Embedded Contact Homology and/or Symplectic Field Theory.

\subsubsection*{Acknowledgments}

The writing of this article benefited substantially from discussions
with John Etnyre, David Gay, Michael Hutchings,
Klaus Niederkr\"uger, Andr\'as Stipsicz, Cliff Taubes and
Jeremy Van Horn-Morris.  Several of these conversations took place at the
March 2010 MSRI Workshop \textsl{Symplectic and Contact Topology and 
Dynamics: Puzzles and Horizons}.
I would also like to thank Patrick Massot for suggesting the use of
the term ``co-fillable,'' as well as Paolo Ghiggini and an anonymous 
referee, both of whom read the original
version quite carefully and made several suggestions for improving the
exposition.

\subsection{Some background and sample results}
\label{sec:results}

In topology, an oriented cobordism from one closed oriented
manifold~$M_-$ to another~$M_+$ is a compact oriented manifold~$W$
such that $\p W = M_+ \sqcup (-M_-)$.  If~$W$ has dimension~$2n$ and 
also carries a symplectic
structure~$\omega$, then it is natural to consider the case where
$(W,\omega)$ is \defin{symplectically convex} at~$M_+$ and \defin{concave} 
at~$M_-$: this means there exists a vector field
$Y$ near~$\p W$ which points transversely outward at~$M_+$ and 
inward at~$M_-$, and is a \defin{Liouville} vector field, 
i.e.~$\Lie_Y\omega = \omega$.  In this case the $1$-form
$\lambda := \iota_Y\omega$ is a primitive of~$\omega$ and its restriction
to each boundary component~$M_\pm$ is a (positive) \defin{contact form},
meaning it satisfies $\lambda \wedge (d\lambda)^{n-1} > 0$.  The
induced \defin{contact structure} on~$M_\pm$ is the 
co-oriented\footnote{Though contact structures need not be co-orientable
in general, all contact structures considered in this paper will be, and
we shall regard the co-orientation always as an essential part of the data,
though it will usually be suppressed in the notation.}
hyperplane field
$\xi_\pm := \ker\left( \lambda|_{TM\pm}\right)$, and up to isotopy it
depends only on the symplectic structure near the boundary.  We thus
call $(W,\omega)$ a \defin{(strong) symplectic cobordism} from
$(M_-,\xi_-)$ to $(M_+,\xi_+)$, and when such a cobordism exists, we say
that $(M_-,\xi_-)$ is \defin{(strongly) symplectically cobordant} 
to $(M_+,\xi_+)$.
If $\lambda$ also extends to a global primitive of $\omega$,
or equivalently, $Y$ extends to a global Liouville field on~$W$, then
we call $(W,\omega)$ an \defin{exact symplectic cobordism} from
$(M_-,\xi_-)$ to $(M_+,\xi_+)$.  Whenever $(M_-,\xi_-)$ and $(M_+,\xi_+)$
are both connected, we shall abbreviate the existence of a
\emph{connected} symplectic cobordism from $(M_-,\xi_-)$ to $(M_+,\xi_+)$
by writing
$$
(M_-,\xi_-) \cob (M_+,\xi_+)
$$
for the general case, and
$$
(M_-,\xi_-) \ecob (M_+,\xi_+)
$$
for the exact case.\footnote{The reason to single out connected
cobordisms is that
technically, every closed contact $3$-manifold $(M,\xi)$ is 
symplectically cobordant to the standard 
contact $3$-sphere $(S^3,\xi_0)$, namely via the disjoint union of
a symplectic cap for $(M,\xi)$ with a symplectic filling of $(S^3,\xi_0)$.
We will see however that if the cobordism is required to be connected,
then the question of when $(M,\xi) \cob (S^3,\xi_0)$ becomes an interesting
one, cf.~Theorems~\ref{thm:toS3}, \ref{thm:planar} and~\ref{thm:nonpositive}.}

When $\dim W = 4$, it is also interesting to consider a much weaker notion: 
without assuming
that $\omega$ is exact near $\p W$, we call $(W,\omega)$ a
\defin{weak symplectic cobordism} from $(M_-,\xi_-)$ to $(M_+,\xi_+)$
if $\xi_\pm$ are any two positive co-oriented (and hence also oriented) 
contact structures such that $\omega|_{\xi_\pm} > 0$.  We then say
that $\omega$ \defin{dominates} the contact structures on both boundary
components.  In order to distinguish strong symplectic cobordisms
from this weaker notion, we will sometimes refer to convex/concave boundary
components of strong cobordisms as \emph{strongly} convex/concave.

It is a standard fact that strong or exact symplectic cobordisms can always be
glued together along contactomorphic boundary components, thus
the relations $\cob$ and $\ecob$ define preorders on 
the category of closed and connected contact manifolds, 
i.e.~they are reflexive and transitive.  They are
neither symmetric nor antisymmetric, as is clear from some simple examples 
that we shall recall in a moment.
Regarding the 
empty set as a trivial example of a contact manifold, we say that
$(M,\xi)$ is \defin{weakly}/\defin{strongly}/\defin{exactly} \defin{fillable} 
if there exists a weak/strong/exact symplectic cobordism from $\emptyset$
to $(M,\xi)$.  For example, the tight $3$-sphere
$(S^3,\xi_0)$ is exactly fillable, as
it is the convex boundary of the unit $4$-ball with its standard 
symplectic structure.
There are many known examples of contact $3$-manifolds that are not
fillable by these various definitions: the original such result, that the
so-called \emph{overtwisted} contact manifolds are not weakly fillable,
was proved by Gromov \cite{Gromov} and Eliashberg 
\cite{Eliashberg:overtwisted}.  In contrast, Etnyre and Honda
\cite{EtnyreHonda:cobordisms} showed that every contact $3$-manifold
admits a \defin{symplectic cap}, meaning it is strongly cobordant
to~$\emptyset$ (though never exactly, due to Stokes' theorem).

There is an obvious obstruction to the relation
$(M_-,\xi_-) \cob (M_+,\xi_+)$ whenever $(M_-,\xi_-)$ is strongly fillable
but $(M_+,\xi_+)$ is not, e.g.~$(M_-,\xi_-)$ cannot be the tight
$3$-sphere if $(M_+,\xi_+)$ is overtwisted.  Put another way, symplectic
cobordisms imply filling obstructions, as $(M,\xi)$ cannot be fillable if
it is cobordant to anything overtwisted.  The following question
may be viewed as a test case for the existence of subtler obstructions
to symplectic cobordisms.

\begin{question}
\label{question:cob}
Is every contact $3$-manifold $(M,\xi)$ that is not strongly fillable
also symplectically cobordant to some overtwisted contact manifold 
$(M_\OT,\xi_\OT)$?
\end{question}

This question was open when the first version of the present article appeared,
but it has since been answered in the negative: by an argument of
Hutchings \cite{Hutchings:SFT}, 
$(M,\xi) \cob (M_\OT,\xi_\OT)$ implies that the contact invariant
in the Embedded Contact Homology of $(M,\xi)$ must vanish, and therefore
(using \cite{ColinGhigginiHonda:HF=ECH3}), so does the Ozsv\'ath-Szab\'o
contact invariant.  A negative example for Question~\ref{question:cob}
is therefore furnished by any nonfillable contact manifold with nonvanishing
Ozsv\'ath-Szab\'o invariant; the first such examples were found by
Lisca and Stipsicz in \cite{LiscaStipsicz:tight1}. 

The answer to the corresponding question for exact cobordisms is much less subtle:
by an argument originally due to Hofer \cite{Hofer:weinstein}, $(M,\xi) \ecob
(M_\OT,\xi_\OT)$ implies that every Reeb vector field on $(M,\xi)$ admits
a contractible periodic orbit, yet there are simple examples of
contact manifolds without contractible orbits that are known to be
non-fillable, e.g.~all of the tight $3$-tori other than the standard one.
More generally, it has recently become clear that overtwistedness is only
the first level in an infinite hierarchy of filling obstructions called
\emph{planar $k$-torsion} for integers $k \ge 0$, cf.~\cite{Wendl:openbook2}.
A contact manifold is overtwisted if and only if it has planar $0$-torsion,
and there are many examples which are tight or have no Giroux
torsion but have planar $k$-torsion for some $k \in \NN$, and are thus
not strongly fillable.  The aforementioned argument of Hofer then generalizes
to define an algebraic filling obstruction \cite{LatschevWendl} that lives in
Symplectic Field Theory and sometimes also gives obstructions to exact 
cobordisms from $k$-torsion to
$(k-1)$-torsion.  Our first main result says that no such
obstructions exist for non-exact cobordisms, thus giving a large class of
contact $3$-manifolds for which the answer to Question~\ref{question:cob}
is yes.

\begin{thm}
\label{thm:overtwisted}
Every closed contact $3$-manifold with planar torsion admits a 
(non-exact) symplectic cobordism to an overtwisted contact manifold.
\end{thm}

This of course yields a new and comparatively low-tech proof of the fact,
proved first in \cite{Wendl:openbook2},
that planar torsion obstructs strong fillings.  It also generalizes a result 
proved by David Gay in \cite{Gay:GirouxTorsion}, that any contact manifold 
with Giroux torsion at least~$2$ is cobordant to something overtwisted;
as shown in \cite{Wendl:openbook2}, positive Giroux torsion implies
planar $1$-torsion (cf.~\S\ref{sec:Gay}).  
By a result of Etnyre and Honda \cite{EtnyreHonda:cobordisms}, every
connected overtwisted contact manifold admits a connected
Stein cobordism to any other
connected contact $3$-manifold, and Gay \cite{Gay:GirouxTorsion} showed
that the word ``connected'' can be removed from this statement at the
cost of dropping the Stein condition.  We thus have the following consequence:

\begin{cor}
\label{cor:equivalence}
Every closed connected contact $3$-manifold with planar torsion admits a
connected strong symplectic cobordism to every other 
closed contact $3$-manifold.
\end{cor}
It should be emphasized that due to the obstructions mentioned above,
Corollary~\ref{cor:equivalence} is not true
for \emph{exact} cobordisms, not even if the positive boundary is required
to be connected.  In fact, there is no known example of an exact
cobordism from anything tight to anything overtwisted, and many examples
that are tight but non-fillable (e.g.~the $3$-tori with positive Giroux 
torsion) certainly do not admit such cobordisms.

There is also a version of Theorem~\ref{thm:overtwisted} that implies the
more general obstruction to weak fillings proved in \cite{NiederkruegerWendl}.
Recall (see Definitions~\ref{defn:planar} and~\ref{defn:planarTorsion})
that for a given closed $2$-form~$\Omega$ on a contact $3$-manifold
$(M,\xi)$, we say that $(M,\xi)$ has \emph{$\Omega$-separating} planar
torsion if it contains a planar torsion domain in which a certain set
of embedded $2$-tori~$T$ all satisfy
$$
\int_T \Omega = 0.
$$
If this is true for all closed $2$-forms~$\Omega$, then $(M,\xi)$ is said
to have \emph{fully separating} planar torsion.

\begin{thm}
\label{thm:weak}
Suppose $(M,\xi)$ is a closed contact $3$-manifold with 
$\Omega$-separating planar torsion for some closed $2$-form $\Omega$ on~$M$
with $\Omega|_\xi > 0$.  Then there exists a weak symplectic cobordism
$(W,\omega)$ from $(M,\xi)$ to an overtwisted contact manifold,
with $\omega|_{TM} = \Omega$.
\end{thm}

Using a Darboux-type normal form near the boundary, weak
symplectic cobordisms can be glued together along contactomorphic
boundary components of opposite sign whenever the restrictions of the symplectic
forms on the boundaries match (see Lemma~\ref{lemma:gluing}).
Thus if $(M,\xi)$ has $\Omega$-separating planar $k$-torsion and
admits a weak filling $(W,\omega)$ with $[\omega|_{TM}] = [\Omega] \in
H^2_\dR(M)$, then Theorem~\ref{thm:weak} yields a weak filling of an 
overtwisted contact manifold, and hence a
contradiction due to the well known theorem of Gromov \cite{Gromov} and
Eliashberg \cite{Eliashberg:diskFilling}.  We thus obtain a much
simplified proof of the following result, which was proved 
in \cite{NiederkruegerWendl} by a direct holomorphic curve argument and
also follows from a computation of the twisted ECH contact invariant
in \cite{Wendl:openbook2}.

\begin{coru}[\cite{NiederkruegerWendl}]
If $(M,\xi)$ has $\Omega$-separating planar torsion for some closed
$2$-form~$\Omega$ on~$M$, then it does not admit any weak filling 
$(W,\omega)$ with $[\omega|_{TM}] = [\Omega] \in H^2_\dR(M)$.  In particular,
if $(M,\xi)$ has fully separating planar torsion then it is not weakly
fillable.
\end{coru}

We now state some related results that also apply to fillable contact
manifolds.  The aforementioned existence result of 
\cite{EtnyreHonda:cobordisms} for symplectic caps was generalized 
independently by Eliashberg \cite{Eliashberg:cap} and Etnyre 
\cite{Etnyre:fillings} to weak cobordisms: they showed namely that
for any $(M,\xi)$ with a closed $2$-form
$\Omega$ that dominates~$\xi$, there is a symplectic cap $(W,\omega)$ with
$\p W = -M$ and $\omega|_{TM} = \Omega$.  Our next result concerns a
large class of contact manifolds for which this cap may be assumed to
have a certain very restrictive property.

\begin{thm}
\label{thm:caps}
Suppose $(M,\xi)$ is a contact $3$-manifold containing an
$\Omega$-sepa\-ra\-ting partially planar
domain $M_0 \subset M$ (see Definition~\ref{defn:planar}) for
some closed $2$-form $\Omega$ on~$M$ with $\Omega|_\xi > 0$.
Then $(M,\xi)$ admits a symplectic cap $(W,\omega)$
such that $\omega|_{TM} = \Omega$ and there exists a symplectically
embedded $2$-sphere $S \subset W$ with vanishing self-intersection number.
\end{thm}

As the work of McDuff \cite{McDuff:rationalRuled} makes clear, symplectic
manifolds that contain symplectic spheres of square~$0$ are quite special, and
for instance any closed symplectic manifold
obtained by gluing the cap from Theorem~\ref{thm:caps} to a filling
of $(M,\xi)$ must be rational or ruled.  An easy adaptation of the
main result in \cite{AlbersBramhamWendl} also provides the following
consequence, which was proved using much harder punctured holomorphic
curve arguments in \cites{Wendl:openbook2,NiederkruegerWendl}:

\begin{cor}
\label{cor:separating}
Suppose $(M,\xi)$ contains an $\Omega$-separating partially planar domain
for some closed $2$-form $\Omega$ on~$M$.  If $(W,\omega)$ is a closed
symplectic $4$-manifold and $M$ admits an embedding 
$\iota : M \hookrightarrow W$ such that $\iota^*\omega|_\xi > 0$ and
$[\iota^*\omega] = [\Omega] \in H^2_\dR(M)$, then~$\iota(M)$ 
separates~$W$.
\end{cor}
Since planar torsion domains are also partially planar domains, this implies
that planar torsion is actually an obstruction to contact type embeddings
into closed symplectic manifolds, not just symplectic fillings.

Some examples of contact manifolds admitting non-separating embeddings
arise from special types of symplectic fillings: we shall say that
$(M,\xi)$ is (strongly or weakly) \defin{co-fillable} if there is a
connected (strong or weak) filling $(W,\omega)$ whose boundary is the
disjoint union of $(M,\xi)$ with an arbitrary non-empty contact manifold.
Put another way, $(M,\xi)$ admits a connected \defin{semi-filling} 
with disconnected boundary.
Given such a filling, one can always attach a symplectic $1$-handle to
connect distinct boundary components and then cap off the boundary 
to realize $(M,\xi)$
as a non-separating contact hypersurface.  Various examples of
contact manifolds that are or are not co-fillable have been
known for many years:
\begin{itemize}
\setlength{\itemsep}{0cm}
\item The tight $3$-sphere $(S^3,\xi_0)$ is not weakly co-fillable,
by arguments due to Gromov \cite{Gromov}, 
Eliashberg \cite{Eliashberg:diskFilling} and McDuff \cite{McDuff:boundaries}.
Etnyre \cite{Etnyre:planar} extended this result to all planar contact
manifolds.
\item McDuff \cite{McDuff:boundaries} showed that for any Riemann surface
$\Sigma$ of genus at least~$2$, the unit cotangent bundle $S T^*\Sigma$
with its canonical contact structure is strongly co-fillable.  Further
examples were found by Geiges \cite{Geiges:disconnected} and
Mitsumatsu \cite{Mitsumatsu:Anosov}.
\item Giroux \cite{Giroux:plusOuMoins} showed that every tight contact
structure on~$T^3$ is weakly co-fillable.  However, none of them are
strongly co-fillable, due to a result of the author \cite{Wendl:fillable}.
\end{itemize}
All of the negative results just mentioned can be viewed as special cases
of Corollary~\ref{cor:separating}, and so can the closely related
result in \cite{AlbersBramhamWendl}, that partially planar contact manifolds
never admit non-separating contact type embeddings.  
Observe that any contact manifold cobordant to one for which
Corollary~\ref{cor:separating} holds also cannot be co-fillable:
in particular this shows that not every contact $3$-manifold is cobordant
to $(S^3,\xi_0)$.  We are
thus led to an analogue of Question~\ref{question:cob} that also applies to
fillable contact manifolds:

\begin{question}
\label{question:S3}
Does every closed and connected contact $3$-manifold $(M,\xi)$ that is not 
strongly co-fillable satisfy $(M,\xi) \cob (S^3,\xi_0)$?
\end{question}

To the author's knowledge, this question is open.
The answer is again clearly no for exact cobordisms, as a
variation on Hofer's argument from \cite{Hofer:weinstein} also shows that
$(M,\xi)$ must always admit contractible Reeb orbits if
$(M,\xi) \ecob (S^3,\xi_0)$.
The following result provides some evidence for a positive answer in
the non-exact case, though
it is not quite as general as one might have hoped.  (See also
Remark~\ref{remark:counterexamples?} below for a candidate counterexample.)

\begin{thm}
\label{thm:toS3}
Suppose $(M,\xi)$ is a connected contact $3$-manifold containing a partially
planar domain which either has more than one irreducible subdomain or
has nonempty binding.  Then $(M,\xi) \cob (S^3,\xi_0)$, hence $(M,\xi)$ 
admits a connected strong symplectic cobordism to every connected strongly
fillable contact $3$-manifold.
\end{thm}

The conditions of Theorem~\ref{thm:toS3} hold in particular for all
planar contact manifolds, and in fact a stronger version can be stated
since the fully separating condition is always satisfied.
We will show in \S\ref{sec:Etnyre} that this implies Etnyre's
planarity obstruction from \cite{Etnyre:planar}.

\setcounter{thmp}{\value{thm}}
\begin{thmp}
\label{thm:planar}
Suppose $(M,\xi)$ is a connected planar contact $3$-manifold.  Then 
there exists a compact connected $4$-manifold $W$ with 
$\p W = S^3 \sqcup (-M)$, with the property that
for every closed $2$-form $\Omega$ on~$M$ with $\Omega|_\xi > 0$,
$W$ admits a symplectic structure~$\omega$ such that $\omega|_{TM} = \Omega$
and $(W,\omega)$ is a weak symplectic cobordism from $(M,\xi)$ to
$(S^3,\xi_0)$.
\end{thmp}

\begin{remark}
\label{remark:counterexamples?}
Let us describe a contact manifold that could conceivably furnish a
negative answer to Question~\ref{question:S3}.
Consider the standard contact $3$-torus $(T^3,\xi_1)$ (the definition
of $\xi_1$ is recalled in \eqref{eqn:xin} below), and divide it by the
$\ZZ_2$-action induced by the contact involution
$$
T^3 \to T^3 : (\eta,\phi,\theta) \mapsto (\eta + 1/2, -\phi, -\theta).
$$
The quotient $T^3 / \ZZ_2$ then inherits a contact structure~$\xi$,
which is supported by a summed open book with empty binding,
one interface torus $T := \{ 2\eta \in \ZZ \}$ and fibration
$$
\pi([\eta,\phi,\theta]) = \begin{cases}
\phi & \text{ for $0 < \eta < 1/2$,}\\
-\phi & \text{ for $1/2 < \eta < 1$.}
\end{cases}
$$
Since the pages are cylinders, $(T^3 / \ZZ_2,\xi)$ is a partially planar domain,
so Corollary~\ref{cor:separating} implies that it is not strongly
co-fillable.  (Note that $(T^3 / \ZZ_2,\xi)$ \emph{is} Stein fillable, as it
can be constructed from the Stein fillable torus $(T^3,\xi_1)$ by a
sequence of contact $(-1)$-surgeries along Legendrian curves in the
pre-Lagrangian fibers $\{\eta = \text{const}\}$.)  
Theorem~\ref{thm:toS3} however
does not apply, as there is only one irreducible subdomain and no binding.
It is not clear whether $(T^3 / \ZZ_2,\xi) \cob (S^3,\xi_0)$.
\end{remark}

\subsection{The main theorems on handle attaching}
\label{sec:handles}

The cobordisms of the previous section
are constructed by
repeated application of two handle attaching constructions that we 
shall now describe.  The handles we will work with take the form
$$
-\Sigma \times \DD
\quad\text{ and }\quad
-\Sigma \times [-1,1] \times S^1,
$$
where in each case $\Sigma$ is a compact oriented surface with boundary,
appearing with reversed orientation because we think of it as a
``symplectic cap'' for the page of an open book decomposition.
In the first case, we shall attach $\p\Sigma \times \DD$ to the neighborhood
of a transverse link, and in the second case, 
$\p\Sigma \times [-1,1] \times S^1$ is attached to the neighborhood of a
disjoint union of pre-Lagrangian tori.  It is important however to understand
that these constructions are not truly \emph{local}, as the attaching
requires neighborhoods that are in some sense sufficiently large.  This
condition on the neighborhoods is most easily stated in the language of
(possibly blown up and summed) open books---that is not
necessarily the only natural setting in which these operations make sense,
but it is the first that comes to mind.

We have derived considerable inspiration
from the symplectic capping technique introduced
by Eliashberg in \cite{Eliashberg:cap}.  The goal of Eliashberg's
construction was somewhat different, namely to embed any weak symplectic
filling into a closed symplectic manifold, but it can also be used to
construct symplectic cobordisms between contact manifolds with supporting
open books that are related to each other by capping off binding components.
Indeed, Eliashberg's capping construction works as follows:
\begin{enumerate}
\item Given $(M,\xi)$ with a supporting open book $\pi : M \setminus B \to S^1$,
attach $2$-handles to $[0,1]\times M$ at $\{1\} \times M$
along each component of~$B$ via the page framing.  This transforms $M$ 
by a $0$-surgery along each binding component, producing a new $3$-manifold 
$M'$ with a fibration $M' \to S^1$ whose fibers are the closed surfaces 
obtained by capping the original pages with disks.  Any symplectic structure
on $[0,1]\times M$ dominating~$\xi$ can then be extended over the
handles so that the fibers of $M' \to S^1$ are symplectic.
\item Cap off the boundary of the cobordism above by presenting $M'$ as
the boundary of a Lefschetz fibration over the disk with closed fibers.
\end{enumerate}
The first step can be generalized by observing that
if we choose to attach $2$-handles along some but
\emph{not all} components of the binding, then the new manifold $M'$ 
inherits an open book decomposition
$$
\pi' : M' \setminus B' \to S^1
$$
obtained from~$\pi$ by \emph{capping off} the corresponding 
boundary components of the pages (cf.~\cite{Baldwin:cap}), and we will show
that the symplectic structure can always be arranged so as to produce a weak 
symplectic
cobordism from $(M,\xi)$ to $(M',\xi')$, where $\xi'$ is supported 
by~$\pi'$.  Under some additional topological assumptions one can actually
arrange the weak cobordism to be strong; 
this variation on Eliashberg's construction has already been 
worked out in detail by Gay and Stipsicz \cite{GayStipsicz:cap}.
To generalize further, one can also imagine replacing the
usual $2$-handle $\DD \times \DD$ by $\Sigma \times \DD$ for any
compact orientable surface~$\Sigma$.  We shall carry out this generalization
below, though the reader may prefer to pretend $\Sigma=\DD$ on first reading,
and this suffices for most of the applications we will discuss.

The key to our construction will be to combine the above brand of handle
attachment with a ``blown up'' version,
in which a round handle is attached to $(M,\xi)$ along a $2$-torus
that can be thought of as a blown up binding circle.  This is most
naturally described in the language of \emph{blown up summed open books},
a generalization of open book decompositions that was introduced in
\cite{Wendl:openbook2} and will be
reviewed in more detail in \S\ref{sec:BUSOBD}.  Rougly speaking, a blown
up summed open book on a compact $3$-manifold~$M$, possibly with boundary,
defines a fibration
$$
\pi : M \setminus (B \cup \iI) \to S^1,
$$
where the \emph{binding} $B \subset M \setminus \p M$ is an oriented link
and the \emph{interface} $\iI \subset M \setminus \p M$ is a set of disjoint
$2$-tori, and the connected components of the fibers, which intersect
$\p M$ transversely, are called \emph{pages}.
An ordinary open book is the special case where $\iI = \p M = \emptyset$,
and in general we allow any or all of $B$, $\iI$ and~$\p M$ to be empty,
so there may be closed pages.
As with ordinary open books there is a natural notion of contact
structures being \emph{supported} by a blown up summed open book, in which
case binding components become positively transverse links and interface
and boundary components become pre-Lagrangian tori.
Such a contact structure exists and is unique up to deformation 
unless the pages are closed.  

Suppose~$(M,\xi)$ is a closed contact $3$-manifold containing a 
compact $3$-dimensional submanifold $M_0 \subset M$,
possibly with boundary,
which carries a blown up summed open book $\boldsymbol{\pi}$ that
supports~$\xi|_{M_0}$ and has nonempty binding.  Pick a set of
binding components
$$
B_0 = \gamma_1 \cup \ldots \cup \gamma_N \subset B,
$$
each of which comes with a natural framing determined by the pages 
adjacent to~$\gamma$, called the \defin{page framing}.  For each $\gamma_j \subset B_0$,
we identify a
tubular neighborhood $\nN(\gamma_j) \subset M$ of~$\gamma_j$ with the oriented
solid torus $S^1 \times \DD$ via this framing so that 
$\gamma_j = S^1 \times \{0\}$
with the correct orientation and the fibration~$\pi$ takes the form
$$
\pi(\theta,\rho,\phi) = \phi
$$
on $\nN(\gamma_j) \setminus \gamma_j$, where
$(\rho,\phi)$ denote polar coordinates on the disk, normalized so
that $\phi \in S^1 = \RR / \ZZ$.  Assign to $\p\nN(\gamma_j)$
its natural orientation as the boundary of $\nN(\gamma_j)$ and denote by
$$
\{\mu_j,\lambda_j\} \subset H_1(\p\nN(\gamma_j))
$$
the distinguished positively oriented homology basis for which
$\mu_j$ is a meridian and $\lambda_j$ is the longitude determined by
the page framing.
Denote $\nN(B_0) = \nN(\gamma_1) \cup \ldots \cup \nN(\gamma_N)$.

Now pick a compact, connected and oriented 
surface~$\Sigma$ with~$N$ boundary components
$$
\p \Sigma = \p_1 \Sigma \cup \ldots \cup \p_N \Sigma
$$
and choose an orientation preserving diffeomorphism of each
$\p_j\Sigma$ to $S^1$, thus defining a coordinate $s \in S^1$ for $\p_j\Sigma$.
Using this, we define new compact oriented manifolds
\begin{equation*}
\begin{split}
M' &= (M \setminus \nN(B_0)) \cup (-\Sigma \times S^1), \\
M'_0 &= (M_0 \setminus \nN(B_0)) \cup (-\Sigma \times S^1)
\end{split}
\end{equation*}
by gluing in $\Sigma \times S^1$ via orientation reversing diffeomorphisms
$\p\Sigma_j \times S^1 \to \p\nN(\gamma_j)$ that take the form
$$
(s,t) \mapsto (s,1,t)
$$
in the chosen coordinates.  On the level of homology, the map
$\p\Sigma \times S^1 \to \p\nN(B_0)$ identifies
$[\p_j\Sigma \times \{*\}]$ with $\lambda_j$ and
$[\{z\} \times S^1]$ for $z \in \p_j\Sigma$ with $\mu_j$.

\begin{remark}
In the special case $\Sigma = \DD$, the operation just defined
is simply a Dehn surgery along a binding component~$\gamma \subset B$
with framing~$0$ relative to the page framing.
\end{remark}

The fibration $\pi : M \setminus (B \cup \iI) \to S^1$ extends smoothly
over $\Sigma \times S^1$ as the projection to the second factor,
thus $M_0'$ inherits from
$\boldsymbol{\pi}$ a natural blown up summed open book
$\boldsymbol{\pi}'$, with binding $B \setminus B_0$,
interface~$\iI$ and pages that are obtained from the pages 
of~$\boldsymbol{\pi}$ by attaching~$-\Sigma$, gluing $\p_j \Sigma$ to 
the boundary component adjacent to~$\gamma_j$.
We say that $\boldsymbol{\pi}'$ is obtained from $\boldsymbol{\pi}$ by 
\defin{$\Sigma$-capping surgery along~$B_0$}.
If $\boldsymbol{\pi}'$ does not have closed pages, then it
supports a contact
structure~$\xi'$ on $M_0'$ which can be assumed to match~$\xi$ outside the
region of surgery, and thus extends to~$M'$.  

The $\Sigma$-capping surgery can also be defined by
attaching a generalized version of a $4$-dimensional $2$-handle:
define
$$
\handle_\Sigma = -\Sigma \times \DD,
$$
with boundary
$$
\p\handle_\Sigma = -\p^-\handle_\Sigma \cup \p^+\handle_\Sigma 
:= -(\p \Sigma \times \DD) \cup ( -\Sigma \times S^1).
$$
The above identifications of the neighborhoods $\nN(\gamma_j)$ with
$S^1 \times \DD$ yield an identification of $\nN(B_0)$ with
$\p^-\handle_\Sigma = \p \Sigma \times \DD$, which we use to attach
$\handle_\Sigma$ to the trivial cobordism $[0,1] \times M$ by
gluing $\p^-\handle_\Sigma$ to $\nN(B_0) \subset \{1\} \times M$, 
defining
\begin{equation}
\label{eqn:Wgamma}
W = ([0,1] \times M) \cup_{\nN(B_0)} \handle_\Sigma,
\end{equation}
which after smoothing the corners has boundary
$$
\p W = M' \sqcup (-M).
$$
We will refer to the oriented submanifolds
$$
\core_\Sigma := ([0,1] \times B_0) \cup_{B_0} (- \Sigma \times \{0\})
\subset W
$$
and
$$
\cocore_\Sigma := \{p\} \times \DD \subset W
$$
for an arbitrary interior point $p \in \Sigma$ as the
\defin{core} and \defin{co-core} respectively.  Note that
$\p \core_\Sigma = -B_0 \subset M$, $\p \cocore_\Sigma \subset M'$
and $\core_\Sigma \bullet \cocore_\Sigma = 1$, where $\bullet$ denotes the
algebraic count of intersections.

The following generalizes results in \cite{Eliashberg:cap}
and \cite{GayStipsicz:cap}.

\begin{thm}
\label{thm:2handles}
Suppose $\omega$ is a symplectic form on $[0,1] \times M$ with 
$\omega|_{\xi} > 0$, and let $W$ denote the handle cobordism
defined in \eqref{eqn:Wgamma}, after smoothing corners.
Then after a symplectic deformation of~$\omega$ away from $\{0\} \times M$,
$\omega$~can be extended symplectically over~$W$ so that it is positive on
$\core_\Sigma$, $\cocore_\Sigma$ and the pages of
$\boldsymbol{\pi}'$.  Moreover, if the latter pages are not
closed, then $\omega$ also dominates a supported contact
structure~$\xi'$ on $M'$, thus defining a weak
symplectic cobordism from $(M,\xi)$ to $(M',\xi')$.
\end{thm}

We will refer to the cobordism $(W,\omega)$ of Theorem~\ref{thm:2handles}
henceforward as a \defin{$\Sigma$-capping cobordism}.  In general it is
a weak cobordism, but under certain conditions that depend only on the
topology of the setup, it can also be made strong.  Recall the standard
fact, observed originally by Eliashberg
\cite{Eliashberg:contactProperties}*{Proposition~3.1} (see also
\cite{Eliashberg:cap}*{Prop.~4.1}),
that whenever $(W,\omega)$ has a boundary
component $M$ on which $\omega$ dominates a positive contact structure~$\xi$
and is exact, $\omega$ can be deformed in a collar neighborhood to 
make~$M$ strongly convex, with~$\xi$ as the induced contact structure.  
In \S\ref{sec:cohomology} we will use routine Mayer-Vietoris
arguments to characterize the situations in which this trick can be
applied to the above construction.

\setcounter{thmp}{\value{thm}}
\begin{thmp}
\label{thm:2handleCohomology}
The symplectic cobordism $(W,\omega)$ constructed by
Theorem~\ref{thm:2handles} can be arranged so that the following holds.
Choose a real $1$-cycle~$h$
in $M \setminus \nN(B_0)$ such that $[h] \in H_1(M ; \RR)$ is Poincar\'e dual to
the restriction of~$\omega$ to $\{0\} \times M$.  Then there is a number $c > 0$
such that 
$$
\PD([\omega]) = [0,1] \times [h] + c\, [\cocore_\Sigma] \in
H_2(W,\p W ; \RR),
$$
where $\PD : H^2_\dR(W) \to H_2(W,\p W ; \RR)$ denotes the Poincar\'e-Lefschetz
duality isomorphism.
In particular, if $\{0\} \times M \subset (W,\omega)$ is 
(strongly) concave then the following conditions are equivalent:
\begin{enumerate}
\renewcommand{\labelenumi}{(\roman{enumi})}
\item $\omega$ is exact.
\item $[\cocore_\Sigma] = 0 \in H_2(W,\p W ; \RR)$.
\item $[\gamma_1] + \ldots + [\gamma_N]$ is not torsion in
$H_1(M)$.
\end{enumerate}
Further, assuming that $\{0\} \times M$ is concave, the following
conditions are also equivalent:
\begin{enumerate}
\renewcommand{\labelenumi}{(\roman{enumi})}
\item $(W,\omega)$ can be arranged to be a strong symplectic cobordism
from $(M,\xi)$ to $(M',\xi')$.
\item $[\p \cocore_\Sigma] = 0 \in H_1(M' ; \RR)$.
\item $\lambda_1 + \ldots + \lambda_N$ is not torsion in
$H_1(M \setminus B_0)$, where $\lambda_j$ denote the longitudes on
$\p\nN(\gamma_j)$ determined by the page framing.
\end{enumerate}
\end{thmp}

It should be emphasized that the above theorem assumes $\Sigma$ is
\emph{connected}.  The case where $\Sigma$ is disconnected is equivalent
to performing multiple surgery operations in succession, but the
statement of Theorem~\ref{thm:2handleCohomology} would then become
more complicated.

\begin{remark}
For the case $\Sigma=\DD$, if $\gamma \subset B$ denotes the binding component
where $0$-surgery is performed, then
Theorem~\ref{thm:2handleCohomology} means that $\omega$ will be exact on~$W$
if and only if $\gamma$ is not torsion in $H_1(M)$, and $(W,\omega)$ can
be made into a strong cobordism
if and only if~$\gamma$ has no nullhomologous cover whose
page framing matches its Seifert framing.  An equivalent condition is assumed
in \cite{GayStipsicz:cap}, which only constructs strong 
cobordisms.
\end{remark}

\begin{remark}
\label{remark:notExact}
Though~$\omega$ in the above construction is sometimes an exact 
symplectic form, $(W,\omega)$ is \emph{never} an exact
cobordism, i.e.~it does not admit a global primitive that 
restricts to suitable contact forms on both boundary components.  This follows
immediately from the observation that the core $\core_\Sigma \subset W$
is a symplectic submanifold whose oriented boundary
is a \emph{negatively transverse} link in $(M,\xi)$, hence if
$\omega = d\lambda$ and $\lambda|_{TM}$ defines a contact form on $(M,\xi)$
with the proper co-orientation, then
$$
0 < \int_{\core_\Sigma} \omega = \int_{\p \core_\Sigma} \lambda < 0,
$$
a contradiction.  A similar remark applies to the round handle cobordism
considered in Theorems~\ref{thm:roundHandles} 
and~\ref{thm:roundHandleCohomology} below.
The non-exactness of $(W,\omega)$ is important because
there are examples in which it is known that no exact cobordism 
from $(M,\xi)$ to $(M',\xi')$ exists (see \S\ref{sec:nonexact}).
\end{remark}

To describe the blown up version of these results, we continue with the same
setup as above and choose a set of interface tori,
$$
\iI_0 = T_1 \cup \ldots \cup T_N \subset \iI,
$$
together with an orientation for each $T_j \subset \iI_0$.  There is then a
distinguished positively oriented homology basis
$$
\{\mu_j,\lambda_j\} \subset H_1(T_j),
$$
where $\lambda_j$ is represented by some oriented boundary component of a
page adjacent to~$T_j$, and $\mu_j$ is represented by a closed leaf of the
characteristic foliation defined on~$T_j$ by~$\xi$.
Choose tubular neighborhoods $\nN(T_j) \subset M$ of~$T_j$ and identify
them with $S^1 \times [-1,1] \times S^1$ to define positively oriented
coordinates $(\theta,\rho,\phi)$ in which $\lambda_j = [S^1 \times \{*\}]$
and $\mu_j = [\{*\} \times S^1]$.  We may then assume that for every
$\theta_0 \in S^1$ the loop $\{(\theta_0,0)\} \times S^1$ is Legendrian,
and the fibration~$\pi$ takes the form
$$
\pi(\theta,\rho,\phi) = \begin{cases}
\phi & \text{ for $\rho > 0$},\\
-\phi & \text{ for $\rho < 0$}.
\end{cases}
$$
Denote the two oriented boundary components of $\nN(T_j)$ by
$$
\p\nN(T_j) = \p_+\nN(T_j) \sqcup \p_-\nN(T_j),
$$
where we define the oriented tori
$\p_\pm \nN(T_j) = \pm(S^1 \times \{\pm 1\} \times S^1)$ with corresponding
homology bases $\{\mu_j^\pm,\lambda_j^\pm\} \subset H_1(\p_\pm\nN(T_j))$
such that
$$
\lambda_j^\pm := \lambda_j \in H_1(\nN(T_j)) \quad\text{ and }\quad
\mu_j^\pm := \pm \mu_j \in H_1(\nN(T_j)).
$$
Denote the union of all the neighborhoods
$\nN(T_j)$ by~$\nN(\iI_0)$.  Then writing two identical 
copies of~$\Sigma$ as $\Sigma_\pm$ and choosing a positively oriented
coordinate $s \in S^1$ for each boundary component $\p_j\Sigma_\pm$, 
we construct new compact oriented manifolds
\begin{equation*}
\begin{split}
M' &= (M \setminus \nN(\iI_0)) \cup 
(- \Sigma_+ \times S^1) \cup (-\Sigma_- \times S^1), \\
M'_0 &= (M_0 \setminus \nN(\iI_0)) \cup 
(- \Sigma_+ \times S^1) \cup (-\Sigma_- \times S^1)
\end{split}
\end{equation*}
from~$M$ and $M_0$ respectively by gluing along orientation reversing
diffeomorphisms $\p_j\Sigma_\pm \times S^1 \to \p_\pm\nN(T_j)$ 
that take the form
$$
(s,t) \mapsto (s,\pm 1,\pm t)
$$
in the chosen coordinates.  Thus in homology,
$[\p_j\Sigma_\pm \times \{*\}] \in H_1(\p_j \Sigma_\pm \times S^1)$ 
is identified with $\lambda_j^\pm$ and
$[\{*\} \times S^1] \in H_1(\p_j\Sigma_\pm \times S^1)$ with~$\mu_j^\pm$.

Once again the fibration $\pi : M \setminus (B \cup \iI) \to S^1$
extends smoothly over the glued in region
$(\Sigma_+ \sqcup \Sigma_-) \times S^1$ as the projection to~$S^1$, so
$M'_0$ inherits from $\boldsymbol{\pi}$ a natural blown up summed open book
$\boldsymbol{\pi}'$, with interface $\iI \setminus \iI_0$, binding~$B$
and fibers that are obtained from the fibers
of~$\boldsymbol{\pi}$ by attaching $-(\Sigma_+ \sqcup \Sigma_-)$ along the
boundary components adjacent to~$\iI_0$.
We say that
$\boldsymbol{\pi}'$ is obtained by \defin{$\Sigma$-decoupling surgery
along~$\iI_0$}.

\begin{remark}
The choice of the term \emph{decoupling} is easiest to justify in the
special case $\Sigma=\DD$: then the surgery cuts open~$M$ along~$T$ and
glues in two solid tori that cap off the corresponding boundary components
of the pages.
\end{remark}

Even if~$M$ is connected, $M'$ may in general be disconnected,
and there is a (possibly empty) component
$$
M'_{\text{flat}} \subset M',
$$
defined as the union of all the closed pages 
of~$\boldsymbol{\pi}'$.  Denote
$M'_{\text{convex}} := M' \setminus M'_{\text{flat}}$, so that
$$
M' = M'_{\text{convex}} \sqcup M'_{\text{flat}}.
$$
On $M'_{\text{convex}}$ there is a contact structure~$\xi'$
which matches~$\xi$ away from the region of surgery and is supported
by $\boldsymbol{\pi}'$ in 
$M'_{\text{convex}} \cap M'_0$.

The above surgery corresponds topologically to the attachment of a
\defin{round handle}: denote the annulus by
$$
\AA = [-1,1] \times S^1
$$
and define
$$
\roundhandle_\Sigma = -\Sigma \times \AA,
$$
with boundary
$$
\p\roundhandle_\Sigma = -\p^-\roundhandle_\Sigma \cup \p^+\roundhandle_\Sigma :=
-\left( \p\Sigma \times \AA \right) \cup 
\left( -\Sigma \times \p \AA \right),
$$
where we identify the two connected components of
$\p^+\roundhandle_\Sigma = -\Sigma \times \{-1,1\} \times S^1$ with
$-\Sigma_\pm \times S^1$ via the orientation preserving maps
\begin{equation}
\label{eqn:Sigmapm}
-\Sigma_\pm \times S^1 \to -\Sigma \times \{\pm 1\} \times S^1 :
(p,\phi) \mapsto (p,\pm 1,\pm \phi).
\end{equation}
Using the identifications of the neighborhoods $\nN(T_j)$ with $S^1 \times
[-1,1] \times S^1$ chosen above, we can identify 
$$
\p^-\roundhandle_\Sigma = 
\p\Sigma \times \AA = \bigsqcup_{j=1}^N \p_j\Sigma \times [-1,1] \times S^1
$$
with $\nN(\iI_0)$ and use this to attach $\roundhandle_\Sigma$
to~$[0,1] \times M$ by gluing $\p^-\roundhandle_\Sigma$ to $\nN(\iI_0) \subset
\{1\} \times M$, defining an oriented cobordism
\begin{equation}
\label{eqn:WT}
W = ([0,1] \times M) \cup_{\nN(\iI_0)} \roundhandle_\Sigma
\end{equation}
with boundary $\p W = M' \sqcup (-M)$.  Use the coordinates
$(\theta,\rho,\phi) \in S^1 \times [-1,1] \times S^1$ on each
$\nN(T_j) \subset \nN(\iI_0)$ to define an oriented link
$\widehat{B}_0$ as the union of all the loops
$$
S^1 \times \{(0,0)\} \subset T_j \subset \iI_0.
$$
Then the core and co-core respectively can be defined as oriented
submanifolds by
$$
\roundcore_\Sigma := ([0,1] \times \widehat{B}_0) \cup_{\widehat{B}_0}
(-\Sigma \times \{(0,0)\}) \subset W
$$
and
$$
\roundcocore_\Sigma := \{p\} \times \AA \subset W
$$
for an arbitrary interior point $p \in \Sigma$.
We have $\p \roundcore_\Sigma = -\widehat{B}_0 \subset M$,
$\p \roundcocore_\Sigma \subset M'$ and $\roundcore_\Sigma \bullet
\roundcocore_\Sigma = 1$.

\begin{thm}
\label{thm:roundHandles}
Suppose $\omega$ is a symplectic form on $[0,1] \times M$ which
satisfies $\omega|_{\xi} > 0$ and
\begin{equation}
\label{eqn:homological}
\sum_{j=1}^N \int_{T_j} \omega = 0,
\end{equation}
and $W$ denotes the round handle cobordism of \eqref{eqn:WT}.
Then after a symplectic deformation away from $\{0\} \times M$,
$\omega$ can be extended symplectically over~$W$ so that it is positive
on $\roundcore_\Sigma$, $\roundcocore_\Sigma$ and the pages 
of~$\boldsymbol{\pi}'$, and $\omega$ dominates a supported contact
structure~$\xi'$ on $M'_{\text{convex}}$.  In particular,
after capping $M'_{\text{flat}}$ by attaching a Lefschetz 
fibration over the disk as in \cite{Eliashberg:cap}, 
this defines a weak symplectic
cobordism from $(M,\xi)$ to $(M'_{\text{convex}},\xi')$.
\end{thm}

We will refer to $(W,\omega)$ in this construction from now on as a
\defin{$\Sigma$-decoupling cobordism}.

\begin{remark}
The homological condition \eqref{eqn:homological} is clearly
not removable since the
$2$-cycles $\sum_{j=1}^N [\p_\pm\nN(T_j)]$ both become nullhomologous in~$M'$.
Note that here the chosen orientations of the tori~$T_j$ play a role,
i.e.~they cannot in general be chosen arbitrarily unless~$\omega$ is exact.
No such issue arose in Theorem~\ref{thm:2handles} because~$\omega$ is
always exact on a neighborhood of a binding circle.  This is the reason
why the ``$\Omega$-separating'' condition is needed for many of the
results in \S\ref{sec:results}, and there are easy examples to show that
those theorems are not true without it 
(cf.~Remark~\ref{remark:needOmegaSeparating}).
\end{remark}

For the analogue of Theorem~\ref{thm:2handleCohomology} in this
context, we shall
restrict for simplicity to the case where $\int_{T_j}\omega$ vanishes for
every~$T_j \subset \iI_0$.  Note that in this case, the
Poincar\'e dual of $\omega|_{TM}$ can be represented by a real $1$-cycle in
$M \setminus \nN(\iI_0)$.

\setcounter{thmp}{\value{thm}}
\begin{thmp}
\label{thm:roundHandleCohomology}
If $\int_{T_j} \omega = 0$ for each $T_1,\ldots,T_N \subset \iI_0$,
then the symplectic cobordism $(W,\omega)$ constructed by 
Theorem~\ref{thm:roundHandles} can be arranged so that the following holds.
Choose a real $1$-cycle $h$
in $M \setminus \nN(\iI_0)$ with $[h] \in H_1(M ; \RR)$
Poincar\'e dual to the
restriction of~$\omega$ to $\{0\} \times M$.  
Then there is a number $c > 0$ with
$$
\PD([\omega]) = [0,1] \times [h] + c\, [\roundcocore_\Sigma] \in
H_2(W,\p W ; \RR).
$$
In particular, if $\{0\} \times M \subset (W,\omega)$ is 
(strongly) concave then the following conditions are equivalent:
\begin{enumerate}
\renewcommand{\labelenumi}{(\roman{enumi})}
\item $\omega$ is exact.
\item $[\roundcocore_\Sigma] = 0 \in H_2(W,\p W ; \RR)$.
\item There are no integers $k,m_1,\ldots,m_N \in \ZZ$ with $k > 0$ and
$\sum_{j=1}^N m_j = 0$ such that the homology class
$$
k (\lambda_1 + \ldots + \lambda_N) + \sum_{j=1}^N m_j \mu_j \in H_1(\iI_0)
$$
is trivial in $H_1(M)$.
\end{enumerate}
Further, if $\{0\} \times M$ is concave and $M'_{\text{flat}} = \emptyset$,
the following conditions are also equivalent:
\begin{enumerate}
\renewcommand{\labelenumi}{(\roman{enumi})}
\item $(W,\omega)$ can be arranged to be a strong symplectic cobordism
from $(M,\xi)$ to $(M',\xi')$.
\item $[\p \roundcocore_\Sigma] = 0 \in H_1(M' ; \RR)$.
\item There are no integers $k_\pm,m_1^\pm,\ldots,m_N^\pm \in \ZZ$ 
with $k_- + k_+ > 0$ and $\sum_j m_j^+ = \sum_j m_j^- = 0$ such that
$$
k_+ \sum_{j=1}^N \lambda_j^+ + k_- \sum_{j=1}^N \lambda_j^- +
\sum_{j=1}^N m_j^+ \mu_j^+ + \sum_{j=1}^N m_j^- \mu_j^- =
0 \in H_1(M \setminus \iI_0).
$$
\end{enumerate}
Finally, if $M'_{\text{flat}}$ and $M'_{\text{convex}}$ are both nonempty,
assume the labels are chosen so that $\Sigma_+ \times S^1 \subset
M'_{\text{convex}}$ and $\Sigma_- \times S^1 \subset M'_{\text{flat}}$,
and consider the weak cobordism
$$
(\overline{W},\omega) = (W,\omega) \cup_{M'_{\text{flat}}} (X,\omega_X)
$$
from $(M,\xi)$ to $(M'_{\text{convex}},\xi')$
obtained by capping off $M'_{\text{flat}}$ with a symplectic
Lefschetz fibration $X \to \DD$ as in \cite{Eliashberg:cap}.
The following conditions are then equivalent:
\begin{enumerate}
\renewcommand{\labelenumi}{(\roman{enumi})}
\item $(W,\omega)$ can be arranged to be a strong symplectic cobordism
from $(M,\xi)$ to $(M'_{\text{convex}},\xi')$.
\item $[\p \roundcocore_\Sigma \cap M'_{\text{convex}}] = 0 
\in H_1(M'_{\text{convex}} ; \RR)$.
\item There are no integers $k,m_1,\ldots,m_N \in \ZZ$ with $k > 0$ and
$\sum_j m_j = 0$ such that the homology class
$$
k \sum_{j=1}^N \lambda_j^+ + \sum_{j=1}^N m_j \mu_j^+
$$
is trivial in $H_1(M \setminus \iI_0)$.
\end{enumerate}
\end{thmp}

We now discuss some applications of the capping and decoupling
cobordisms to Embedded Contact Homology (cf.~\cite{Hutchings:ICM}).
Recall that for a closed contact $3$-manifold $(M,\xi)$ and homology
class $h \in H_1(M)$, $\ECH_*(M,\xi ; h)$ is the homology of a chain
complex generated by sets of Reeb orbits with multiplicities whose
homology classes add up to~$h$, with a differential counting embedded
index~$1$ holomorphic curves with positive and negative cylindrical ends in
the symplectization $\RR\times M$.  Similarly, counting embedded
index~$2$ holomorphic curves through a generic point in~$M$ yields 
the so-called $U$-map,
$$
U : \ECH_*(M,\xi ; h) \to \ECH_{*-2}(M,\xi ; h).
$$
The \defin{ECH contact invariant}
$$
c(M,\xi) \in \ECH_*(M,\xi ; 0)
$$
is the homology class represented by the ``empty orbit set''. 
It is equivalent via an isomorphism of Taubes \cite{Taubes:ECH=SWF5} to
a corresponding invariant in Seiberg-Witten theory, and 
also to the Ozsv\'ath-Szab\'o contact invariant \cite{OzsvathSzabo:contact}
by recent work of Colin-Ghiggini-Honda
\cite{ColinGhigginiHonda:HF=ECH3} and independently
Kutluhan-Lee-Taubes \cite{KutluhanLeeTaubes5}.
Like those invariants, its vanishing gives an obstruction to strong 
symplectic fillings, and a version with twisted coefficients also obstructs
weak fillings.

\begin{remark}
Technically the definitions of $\ECH_*(M,\xi ; h)$ and
$c(M,\xi)$ depend not just on~$\xi$ but also on a choice of contact form
and almost complex structure.  However, Taubes' isomorphism to
Seiberg-Witten Floer homology implies that they are actually independent
of these choices, thus we are safe in writing $\ECH_*(M,\xi ; h)$
without explicitly mentioning the extra data.
\end{remark}

An argument due to Eliashberg\footnote{In the appendix of 
\cite{Yau:overtwisted}, Eliashberg sketches an argument to show that every
overtwisted contact manifold has trivial contact homology, and this argument
also implies the vanishing of the ECH contact invariant.} 
shows that $c(M,\xi) = 0$
whenever $(M,\xi)$ is overtwisted, and a much more general computation
in \cite{Wendl:openbook2} established the same result whenever
$(M,\xi)$ has planar $k$-torsion for any $k \ge 0$.  The latter result can 
now be recovered as a consequence of Theorem~\ref{thm:overtwisted},
using a result recently announced by Hutchings \cite{Hutchings:SFT} that
$(M_-,\xi_-) \cob (M_+,\xi_+)$ and $c(M_+,\xi_+) = 0$ imply
$c(M_-,\xi_-) = 0$.  This is highly non-obvious since
the cobordisms we construct are never exact (see Remark~\ref{remark:notExact}),
and non-exact cobordisms do not in general give rise to well-behaved maps
on~ECH in its standard form.  
The situation becomes slightly simpler however under stricter
assumptions, e.g.~Hutchings and Taubes have explained in
\cite{HutchingsTaubes:Arnold2} how to construct such maps for the case
$h=0$ whenever $(W,\omega)$ is a
strong cobordism with an exact symplectic form, sometimes called a
``weakly exact'' cobordism:

\begin{prop}[\cite{HutchingsTaubes:Arnold2}]
\label{prop:ECH}
Suppose $(W,\omega)$ is a strong symplectic cobordism from
$(M_-,\xi_-)$ to $(M_+,\xi_+)$ such that~$\omega$ is exact.
Then there is a $U$-equivariant map
$$
\ECH_*(M_+,\xi_+ ; 0) \to \ECH_*(M_-,\xi_- ; 0)
$$
that takes $c(M_+,\xi_+)$ to $c(M_-,\xi_-)$.
\end{prop}

\begin{remark}
For the example of a $2$-handle cobordism constructed from an ordinary
open book decomposition, the analogue of Proposition~\ref{prop:ECH} in
Heegaard Floer homology has been established by John Baldwin
\cite{Baldwin:cap}.
\end{remark}

Let us now discuss a conjectural generalization
of Proposition~\ref{prop:ECH} which could remove all conditions on~$\omega$.
Recall that for any closed $2$-form $\Omega$ on~$M$, one can define ECH
with \defin{twisted coefficients} in the group ring
$\ZZ[H_2(M) / \ker\Omega]$, which we shall abbreviate by
$$
\ECH(M,\xi ; h,\Omega) := \ECH\big(M,\xi ; h, \ZZ[H_2(M) / \ker\Omega]\big).
$$
Here the differential keeps track of the homology classes in
$H_2(M) / \ker\Omega$ of the holomorphic curves being counted, see 
\cite{HutchingsSullivan:T3}.  The $U$-map can again be defined as a
degree~$-2$ map on $\ECH(M,\xi ; h,\Omega)$, and
the twisted contact invariant
$c(M,\xi ; \Omega)$ is again the homology class in $\ECH(M,\xi ; 0,\Omega)$ 
generated by the empty orbit set.  The vanishing results in \cite{Wendl:openbook2}
give convincing evidence that a more general version of the map in
Proposition~\ref{prop:ECH} should exist,
in particular with the following consequence:

\begin{conj}
\label{conj:c}
Suppose $(W,\omega)$ is a $\Sigma$-capping or $\Sigma$-decoupling cobordism
from $(M_-,\xi_-)$ to $(M_+,\xi_+)$, and write $\Omega_\pm = 
\omega|_{TM\pm}$.  Then:
\begin{enumerate}
\item If $c(M_+,\xi_+ ; \Omega_+)$ vanishes, then so
does $c(M_-,\xi_- ; \Omega_-)$.
\item If $c(M_+,\xi_+ ; \Omega_+)$ is in the image of
the map $U^k$ on $\ECH(M_+,\xi_+ ; 0, \Omega_+)$ for
some $k \in \NN$, then $c(M_-,\xi_- ; \Omega_-)$
is in the image of $U^k$ on
$\ECH(M_-,\xi_- ; 0, \Omega_-)$.
\end{enumerate}
\end{conj}

The first part of the conjecture would reduce
both the untwisted and twisted vanishing results in \cite{Wendl:openbook2}
to the fact, proved essentially by Eliashberg in the appendix of 
\cite{Yau:overtwisted}, that the fully twisted contact invariant vanishes
for every overtwisted contact manifold.  The second part is related to
another result proved in \cite{Wendl:openbook2}, namely the twisted ECH 
version of the planarity obstruction of
Oszv\'ath-Stipsicz-Szab\'o \cite{OzsvathSzaboStipsicz:planar}
in Heegaard Floer homology: if $(M,\xi)$ is planar, then 
$c(M,\xi ; \Omega)$ is in the image of $U^k$ for all~$k$ and all~$\Omega$.
If the conjecture holds, then this fact follows from 
Theorem~\ref{thm:planar} and the computation of
$\ECH(S^3,\xi_0)$.

The obvious way to try to prove Conjecture~\ref{conj:c} would be by
constructing a $U$-equivariant map
$$
\ECH(M_+,\xi_+ ; 0 , \Omega_+) \to
\ECH(M_-,\xi_- ; 0 , \Omega_-)
$$
which takes $c(M_+,\xi_+ ; \Omega_+)$ to
$c(M_-,\xi_- ; \Omega_-)$.  Due to the non-exactness
of~$\omega$ and a resulting lack of energy bounds, it seems unlikely 
that such a map would exist in general, but a more probable scenario is
to obtain a map
$$
\ECH(M_+,\xi_+ ; 0 , \Lambda_\omega) \to
\ECH(M_-,\xi_- ; 0 , \Lambda_\omega),
$$
where $\Lambda_\omega$ is a Novikov completion of $\ZZ[H_2(W) / \ker\omega]$,
and we take advantage of the natural inclusions
$$
H_2(M_\pm) / \ker\Omega_\pm \hookrightarrow H_2(W) / \ker\omega
$$
to define the ECH of $(M_\pm,\xi_\pm)$ with coefficients in~$\Lambda_\omega$.
In cases where $M_+$ has connected components with closed leaves, one
would expect this map to involve also the Periodic Floer Homology
(cf.~\cite{HutchingsSullivan:DehnTwist}) of the resulting mapping tori.
Defining such a map would require a slightly more careful construction
of the weak cobordism $(W,\omega)$, such that both boundary components
inherit stable Hamiltonian structures which can be used to attach
cylindrical ends and define reasonable moduli spaces of finite energy
punctured holomorphic curves.  This can always be done due to a
construction in \cite{NiederkruegerWendl}, which shows that suitable
stable Hamiltonian structures exist for any desired cohomology class on
the boundary.  It is probably also useful to observe that for an
intelligent choice of data, the holomorphic curves in $(W,\omega)$ with
no positive ends can be enumerated precisely: we will show in
Proposition~\ref{prop:noPositive} that all of them arise from the
symplectic core of the handle.

\section{Further applications, examples and discussion}
\label{sec:consequences}

We shall now give some concrete examples of capping and decoupling
cobordisms and survey a few more applications, including new proofs
of several known results and one or two new ones.

\subsection{The Gromov-Eliashberg theorem using holomorphic spheres}
\label{sec:GromovEliashberg}

In \cites{Wendl:openbook2,NiederkruegerWendl}, 
holomorphic curve arguments were used to show that planar torsion is a
filling obstruction, but Theorems~\ref{thm:overtwisted}
and~\ref{thm:weak} make these proofs much easier by using essentially
``soft'' methods to reduce them to the well-known result of
Gromov-Eliashberg that overtwisted contact manifolds are not weakly
fillable.  This does not of course make everything elementary, as
the Gromov-Eliashberg theorem still requires some technology---the original
proof used a ``Bishop family'' of holomorphic disks with totally real
boundary, and these days one can instead use punctured holomorphic curves,
Seiberg-Witten theory or Heegaard Floer homology if preferred.
While this technological overhead is probably not
removable, we can use a decoupling cobordism to simplify the level
of technology a tiny bit: namely we can reduce it to the following
standard fact whose proof requires only \emph{closed} holomorphic spheres,
e.g.~the methods used in \cite{McDuff:rationalRuled}.

\begin{lemma}
\label{lemma:spheres}
If $(W,\omega)$ is a connected weak filling of a nonempty contact manifold
$(M,\xi)$, then it contains no embedded symplectic sphere with
vanishing self-intersection.
\end{lemma}

This lemma follows essentially from McDuff's results 
\cite{McDuff:rationalRuled}, but by today's standards it is also easy to 
prove on its own: if one chooses a compatible
almost complex structure to make the boundary $J$-convex and 
the embedded symplectic sphere $J$-holomorphic,
then vanishing self-intersection implies that the latter lives in a smooth
$2$-dimensional moduli space of holomorphic spheres that foliate~$W$
(except at finitely many nodal singularities).  This forces some leaf of
the foliation to hit the boundary tangentially, thus contradicting
$J$-convexity.

\begin{coru}[\cites{Gromov,Eliashberg:diskFilling}]
Every weakly fillable contact manifold is tight.
\end{coru}
\begin{proof}
A schematic diagram of the proof is shown in Figure~\ref{fig:overtwisted}.
Suppose $(W,\omega)$ is a weak filling of $(M,\xi)$ and the latter is
overtwisted.  Then $(M,\xi)$ contains a planar $0$-torsion 
domain\footnote{The fact that overtwistedness implies planar $0$-torsion
relies on Eliashberg's classification of overtwisted contact structures 
\cite{Eliashberg:overtwisted}, quite a large result in itself.  The
original ``Bishop disk'' argument of Gromov and Eliashberg 
had the advantage of not requiring
this.} $M_0$, whose planar piece $M_0^P$ is a solid torus with disk-like
pages, attached along an interface torus $T = \p M_0^P$ to another
subdomain whose pages are not disks.  Since $[T] = 0 \in H_2(M)$,
$\int_T \omega = 0$ and we can attach a $\DD$-decoupling cobordism
along~$T$, producing a larger symplectic manifold $(W',\omega)$ whose
boundary has two connected components
$$
\p W' = M'_{\text{flat}} \sqcup M'_{\text{convex}},
$$
of which the latter carries a contact structure~$\xi'$ dominated
by~$\omega$.  The component $M'_{\text{flat}}$ has closed sphere-like pages,
and is thus the trivial symplectic fibration $S^1 \times S^2 \to S^1$.
After capping $M'_{\text{flat}}$ by a symplectic fibration
$\DD \times S^2 \to \DD$, we then obtain a weak filling of
$(M'_{\text{convex}},\xi')$ containing a symplectic sphere with vanishing
self-intersection, contradicting Lemma~\ref{lemma:spheres}.
\end{proof}

\begin{figure}
\begin{center}
\includegraphics{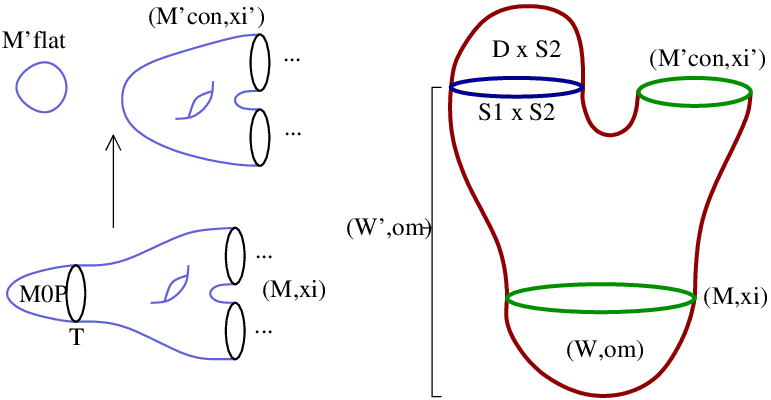}
\caption{\label{fig:overtwisted}
A proof of the Gromov-Eliashberg theorem using a $\DD$-decoupling cobordism.
The left shows the effect on the pages when a round handle $\DD \times \AA$
is attached to a planar $0$-torsion domain.  The right shows the resulting
cobordism and consequent weak filling which furnishes a contradiction
to Lemma~\ref{lemma:spheres}.}
\end{center}
\end{figure}

\begin{remark}
A related argument appears in \cite{Gay:GirouxTorsion}, using the fact
that overtwisted contact manifolds always have Giroux torsion;
see also \S\ref{sec:Gay} below.
\end{remark}

\subsection{Eliashberg's cobordisms from $T^3$ to $S^3 \sqcup \ldots \sqcup S^3$}
\label{sec:EliashbergT3}

Let $T^3 = S^1 \times S^1 \times S^1$ with coordinates $(\eta,\phi,\theta)$
and define for $n \in \NN$ the contact structure
\begin{equation}
\label{eqn:xin}
\xi_n = \ker\left[ \cos(2\pi n \eta) \, d\theta + \sin(2\pi n \eta)\, d\phi \right].
\end{equation}
These contact structures are all tight, but Eliashberg showed in
\cite{Eliashberg:fillableTorus} that they are not strongly fillable for 
$n \ge 2$, which follows from the fact that disjoint unions of multiple
copies of $(S^3,\xi_0)$ are not fillable, together with the following:

\begin{thmu}[\cite{Eliashberg:fillableTorus}]
For any $n \in \NN$, $(T^3,\xi_n)$ is symplectically cobordant to the
disjoint union of~$n$ copies of the tight $3$-sphere.
\end{thmu}
\begin{proof}
The torus $(T^3,\xi_n)$ admits a supporting summed open book 
decomposition with $2n$ irreducible subdomains $M_j$
each having cylindrical
pages and trivial monodromy, attached to each other along~$2n$ interface
tori $\iI = \bigcup_j T_j$ such that
$T_j = M_j \cap M_{j+1}$ for $j = \ZZ_{2n}$.  Attaching round handles
$\DD \times\AA$ along every second interface torus 
$T_0,T_2,\ldots,T_{2n-2}$ yields a weak symplectic cobordism to the 
disjoint union of~$n$ copies of the tight $S^1 \times S^2$
(Figure~\ref{fig:T3}).  The latter
is also supported by an open book with cylindrical pages and trivial monodromy,
so we can attach a $2$-handle $\DD \times \DD$ along one binding
component to create a weak cobordism to the tight~$S^3$.  The resulting
weak cobordism from $T^3$ to $S^3 \sqcup \ldots \sqcup S^3$ can be
deformed to a strong cobordism since the symplectic form is necessarily
exact near $S^3 \sqcup \ldots \sqcup S^3$.
\end{proof}
\begin{remark}
\label{remark:needOmegaSeparating}
Note that $(T^3,\xi_n)$ is always \emph{weakly} fillable
\cite{Giroux:plusOuMoins}, and
indeed, the above cobordism cannot be attached to any weak filling
$(W,\omega)$ of $(T^3,\xi_n)$ for which $\int_{T_j} \omega \ne 0$.
This shows that the homological condition in Theorem~\ref{thm:roundHandles}
cannot be removed.
\end{remark}

\begin{figure}
\begin{center}
\includegraphics{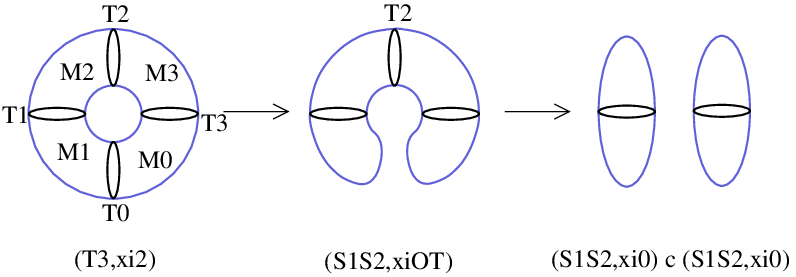}
\caption{\label{fig:T3}
The torus $(T^3,\xi_2)$ can be constructed out of four irreducible
subdomains containing cylindrical pages with trivial monodromy.
Attaching one $\DD$-decoupling cobordism yields an overtwisted
$S^1 \times S^2$, and one can then attach a second one to obtain a
disjoint union of two copies of the tight $S^1 \times S^2$.}
\end{center}
\end{figure}

\subsection{Gay's cobordisms for Giroux torsion}
\label{sec:Gay}

Recall that a contact manifold $(M,\xi)$ is said to have
\defin{Giroux torsion} $\GT(M,\xi) = n \in \NN$ 
if $n$ is the largest integer for which $(M,\xi)$
admits a contact embedding of
$([0,1] \times T^2, \xi_n)$, where $\xi_n$ is given by \eqref{eqn:xin};
we write $\GT(M,\xi) = 0$ if there are no such embeddings and
$\GT(M,\xi) = \infty$ if they exist for arbitrarily large~$n$.
Every contact manifold with positive Giroux torsion also has
planar $1$-torsion (see \cite{Wendl:openbook2}), thus
as a special case of Theorem~\ref{thm:overtwisted}, every
$(M,\xi)$ with $\GT(M,\xi) \ge 1$ is symplectically cobordant to
something overtwisted; this was proved by David Gay in \cite{Gay:GirouxTorsion}
for $\GT(M,\xi) \ge 2$.  A concrete picture of this 
cobordism\footnote{Both the cobordism in Figure~\ref{fig:Giroux} and the
one that is constructed explicitly in \cite{Gay:GirouxTorsion} for the
case $\GT(M,\xi) \ge 2$ are \emph{weak} cobordisms, not strong in general.
As David Gay has pointed out to me, these can always be turned into
strong cobordisms by attaching additional $2$-handles to make the 
positive boundary an overtwisted rational homology sphere (see the proof of
Theorem~\ref{thm:overtwisted} in \S\ref{sec:proofs}).} 
is shown
in Figure~\ref{fig:Giroux}: namely, if $\GT(M,\xi) \ge 1$, then~$M$ contains a
domain $[0,1] \times T^2 \cong M_0 \subset M$ on which~$\xi$ is supported 
by a blown up summed open book with three irreducible subdomains
$$
M_0 = M_- \cup_{T_-} Z \cup_{T_+} M_+
$$
attached to each other in a chain along two
interface tori $T_\pm = Z \cap M_\pm$.  As is explained in
\cite{Wendl:openbook2}, $M_0$ is literally the closure of some open
neighborhood of the standard Giroux torsion domain 
$([0,1] \times T^2,\xi_1)$ in~$M$, and the middle segment $Z$ can be identified
with $[1/4,3/4] \times T^2$ in $([0,1] \times T^2,\xi_1)$: in particular
it has cylindrical pages with trivial monodromy.  Likewise
$M_+$ and $M_-$ have cylindrical pages but nontrivial monodromy in
general---this detail will play no role in the following.
Attaching a round handle
$\DD \times\AA$ along~$T_-$ produces a weak symplectic cobordism to
a new contact manifold~$M'$, containing the disconnected domain
$M'_0$ shown in Figure~\ref{fig:Giroux}: in particular $M_-$ and $Z$ are each
transformed into subdomains $M_-'$ and $Z'$ with disk-like pages.
Now $Z' \cup M_+ \subset M'$ contains an overtwisted disk; indeed, it is
a planar $0$-torsion domain.  Observe that this construction can also
be used to show that $(M,\xi)$ is not \emph{weakly} fillable if
the torsion domain separates~$M$, as then $\int_{T_-} \omega = 0$ for
any symplectic form~$\omega$ arising from a weak filling.

\begin{figure}
\begin{center}
\includegraphics{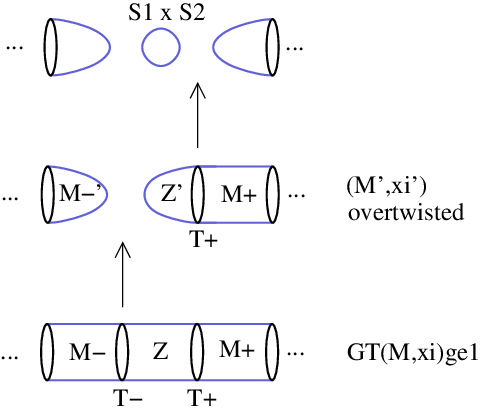}
\caption{\label{fig:Giroux}
Attaching a $\DD$-decoupling cobordism to $(M,\xi)$ with Giroux torsion
at least~$1$ yields an overtwisted contact manifold $(M',\xi')$.
Attaching one more yields the disjoint union of a contact manifold with
a trivial symplectic $S^2$-fibration over~$S^1$.}
\end{center}
\end{figure}

Gay's proof in \cite{Gay:GirouxTorsion} that Giroux torsion obstructs 
strong filling did not directly use the above cobordism, but proved
instead that $(M,\xi)$ with $\GT(M,\xi) \ge 1$ admits a symplectic cobordism
to some non-empty contact manifold such that the cobordism itself contains
a symplectic sphere with vanishing self-intersection---Gay's argument
then used gauge theory to derive a contradiction if $(M,\xi)$ has a
filling, but one can just as well use Lemma~\ref{lemma:spheres} above.
A close relative of Gay's cobordism construction is easily obtained
from the above picture: attaching round handles $\DD\times\AA$
along both $T_-$ and $T_+$, the top of the cobordism contains a connected
component with closed sphere-like pages (the top picture in
Figure~\ref{fig:Giroux}), which can be capped by
$\DD \times S^2$ to produce a cobordism that contains symplectic
spheres of self-intersection number~$0$.

\subsection{Some new examples with $M_- \cob M_+$ but $M_- \necob M_+$}
\label{sec:nonexact}

Gromov's theorem \cite{Gromov} on the non-existence of exact Lagrangians
in $\RR^{2n}$ provides perhaps the original example of a pair of
contact manifolds that are strongly but not exactly cobordant: indeed,
viewing $(T^3,\xi_1)$ as the boundary of a Weinstein neighborhood of
any Lagrangian torus in the
standard strong filling of the tight $3$-sphere $(S^3,\xi_0)$, we obtain
$$
(T^3,\xi_1) \cob (S^3,\xi_0)
\quad\text{ but }\quad
(T^3,\xi_1) \necob (S^3,\xi_0).
$$
The nonexistence of the exact cobordism here can also be proved by the
argument of Hofer \cite{Hofer:weinstein} mentioned in the introduction,
which implies that if $(M,\xi)$ admits a Reeb vector field with no
contractible periodic orbit, then $(M,\xi) \necob (M',\xi')$ whenever
either $(M',\xi')$ is overtwisted or $M' \cong S^3$.
Together with Theorem~\ref{thm:overtwisted}, this implies that for any
$(M_\OT,\xi_\OT)$ overtwisted and $n \ge 2$,
$$
(T^3,\xi_n) \cob (M_\OT,\xi_\OT)
\quad\text{ but }\quad
(T^3,\xi_n) \necob (M_\OT,\xi_\OT).
$$
A subtler obstruction to exact cobordisms is defined in joint work of the
author with Janko Latschev \cite{LatschevWendl} via
Symplectic Field Theory, leading to the
following example.  For any integer $k \ge 1$, suppose $\Sigma$ is a closed,
connected and oriented surface of genus $g \ge k$, and
$\Gamma \subset \Sigma$ is a multicurve consisting of $k$ disjoint embedded
loops which divide~$\Sigma$ into exactly two connected components
$$
\Sigma = \Sigma_+ \cup_\Gamma \Sigma_-,
$$
such that $\Sigma_+$ has genus~$0$ and $\Sigma_-$ has genus $g - k + 1 > 0$.
By a construction due to Lutz \cite{Lutz:77}, the product 
$$
M_{k,g} := S^1 \times \Sigma
$$
then admits a unique (up to isotopy) $S^1$-invariant contact structure
$\xi_{k,g}$ such that the convex surfaces $\{*\} \times \Sigma$ have
dividing set~$\Gamma$.  The contact manifold
$(M_{k,g},\xi_{k,g})$ then has planar $(k-1)$-torsion, as the two subsets
$S^1 \times \Sigma_\pm$ can be regarded as the irreducible subdomains of a
supporting summed open book with pages $\{*\} \times \Sigma_\pm$,
so we view $S^1 \times \Sigma_+$ as the planar piece and
$S^1 \times \Sigma_-$ as the padding (see Definition~\ref{defn:planarTorsion}).
In particular, $(M_{k,g},\xi_{k,g})$ is overtwisted if and only if
$k=1$, and for $k \ge 2$ it has a Reeb vector field with no contractible
periodic orbits.  It turns out in fact that each increment of~$k$ contains
an obstruction to exact fillings that is invisible in the non-exact case.

\begin{thm}
\label{thm:Janko}
If $k > \ell \ge 1$ then for any $g \ge k$ and $g' \ge \ell$,
$$
(M_{k,g},\xi_{k,g}) \cob (M_{\ell,g'},\xi_{\ell,g'})
\quad\text{ but }\quad
(M_{k,g},\xi_{k,g}) \necob (M_{\ell,g'},\xi_{\ell,g'})
$$
\end{thm}
\begin{proof}
The nonexistence of the exact cobordism is a result of \cite{LatschevWendl}.
The existence of the non-exact cobordism follows immediately from
Corollary~\ref{cor:equivalence}, but in certain cases one can construct
it much more explicitly as in Figure~\ref{fig:Mk}.  In particular,
$(M_{k,g},\xi_{k,g})$ is supported by a summed open book
consisting of the two irreducible subdomains $S^1 \times \Sigma_\pm$
with pages $\{*\} \times \Sigma_\pm$ attached along~$k$ interface tori.
Attaching $\DD \times \AA$ along one
of the interface tori gives a weak $\DD$-decoupling cobordism to
$(M_{k-1,g-1},\xi_{k-1,g-1})$.
Theorem~\ref{thm:roundHandleCohomology} then implies that this can be
deformed to a strong cobordism, as the restriction of the symplectic
form to $M_{k-1,g-1}$ is Poincar\'e dual to a multiple of the boundary of
the co-core; the latter consists of two loops of the form $S^1 \times \{*\}$
with opposite orientations and is thus nullhomologous.
\end{proof}

\begin{figure}
\begin{center}
\includegraphics{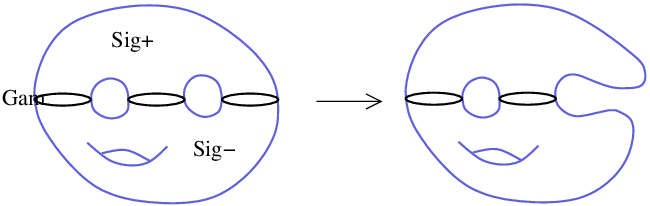}
\caption{\label{fig:Mk}
An explicit cobordism from $(M_{3,3},\xi_{3,3})$ to $(M_{2,2},\xi_{2,2})$
as in the proof of Theorem~\ref{thm:Janko}
can be realized as a $\DD$-decoupling cobordism.}
\end{center}
\end{figure}

\subsection{Open books with reducible monodromy}
\label{sec:reducible}

Any compact, connected and oriented surface $\Sigma$ with boundary,
together with a
diffeomorphism $\varphi : \Sigma \to \Sigma$ fixing the boundary,
determines a contact $3$-manifold $(M_{\varphi},
\xi_{\varphi})$, namely the one supported by the open book
decomposition with page~$\Sigma$ and monodromy~$\varphi$.  
Recall that the mapping class of the monodromy map~$\varphi$ is said to be 
\defin{reducible} if it has a representative that preserves some
multicurve $\Gamma \subset \Sigma$ such that no component of
$\Sigma\setminus\Gamma$ is a disk or an annulus.  Consider the simple
case in which $\varphi$ preserves each individual connected component 
$\gamma \subset \Gamma$ and also preserves its orientation (note that this
is always true for some iterate of~$\varphi$).  In this case we 
may assume after a suitable isotopy that $\varphi$ is the identity on a
neighborhood of $\p\Sigma \cup \Gamma$, so that for some open
annular neighborhood $\gamma \subset \nN(\gamma) \subset \Sigma$ of each
curve $\gamma \subset \Gamma$, $M_\varphi$ contains a thickened torus
region
$$
S^1 \times \nN(\gamma) \subset M_\varphi
$$
on which the open book decomposition is the projection to the first factor.
Let $\nN(\Gamma) \subset \Sigma$ denote the union of all the neighborhoods
$\nN(\gamma)$ and define the possibly disconnected compact surface
$$
\Sigma_\Gamma = \Sigma \setminus \nN(\Gamma)
$$
with boundary; then $\varphi$ restricts to this surface as an orientation
preserving diffeomorphism $\varphi_\Gamma : \Sigma_\Gamma \to \Sigma_\Gamma$
that preserves each connected component and equals the identity
near~$\p\Sigma_\Gamma$.  Denote the connected components
of $\Sigma_\Gamma$ by
$$
\Sigma_\Gamma = \Sigma_\Gamma^1 \sqcup \ldots \sqcup \Sigma_\Gamma^N
$$
and the corresponding restrictions of~$\varphi_\Gamma$ by
$$
\varphi_\Gamma^j : \Sigma_\Gamma^j \to \Sigma_\Gamma^j
$$
for $j = 1,\ldots,N$.  Since each $\Sigma_\Gamma^j$ necessarily has nonempty
boundary, each gives rise to a connected contact manifold
$(M_{\varphi_\Gamma^j},\xi_{\varphi_\Gamma^j})$.

\begin{thm}
\label{thm:reducible}
Given a reducible monodromy map $\varphi : \Sigma \to \Sigma$ as
described above, there exists a weak symplectic cobordism $(W,\omega)$ from
$$
(M_{\varphi_\Gamma^1},\xi_{\varphi_\Gamma^1}) \sqcup \ldots \sqcup
(M_{\varphi_\Gamma^N},\xi_{\varphi_\Gamma^N})
\quad\text{ to }\quad
(M_\varphi,\xi_\varphi),
$$
which is (strongly) concave at the negative boundary and such that the
restriction of~$\omega$
to the positive boundary is Poincar\'e dual to a positive multiple of
$$
\sum_{\gamma \subset \Gamma} [S^1 \times \{p_\gamma\}] \in H_1(M_\varphi;\RR),
$$
where the summation is over the connected components of~$\Gamma$ and
$p_\gamma \subset \nN(\gamma)$ denotes an arbitrarily chosen point.

Moreover, given any closed $2$-form $\Omega$ on
$M_{\varphi_\Gamma^1} \sqcup \ldots \sqcup M_{\varphi_\Gamma^N}$ that
dominates the respective contact structures, one can also construct a
weak cobordism between the contact manifolds above such that
$\omega$ matches $\Omega$ at the negative boundary.
\end{thm}
\begin{proof}
The cobordism is a stack of $\AA$-capping cobordisms, constructed
by attaching handles of the form
$[-1,1] \times S^1 \times \DD$ via Theorem~\ref{thm:2handles} along
all pairs of binding circles in $M_{\varphi_\Gamma^1} \sqcup \ldots
\sqcup M_{\varphi_\Gamma^N}$ that correspond to the same curve in~$\Gamma$.
The co-core of each of these handles is a disk with boundary of the
form $S^1 \times \{*\} \subset S^1 \times \nN(\gamma) \subset M_\varphi$,
thus the cohomology class of~$\omega$ at the positive boundary follows
immediately from Theorem~\ref{thm:2handleCohomology}.
\end{proof}
\begin{cor}
If the contact manifolds $(M_{\varphi_\Gamma^j},\xi_{\varphi_\Gamma^j})$
for $j=1,\ldots,N$ are all weakly fillable, then so is
$(M_\varphi,\xi_\varphi)$.
\end{cor}

\begin{remark}
\label{remark:JohnJeremy}
John Baldwin \cite{Baldwin:cap} has observed that topologically,
the cobordism of Theorem~\ref{thm:reducible} can also be obtained by
performing boundary connected sums on the pages and then using
$\DD$-capping cobordisms to remove extra boundary components; in
\cite{Baldwin:cap} this is used to deduce a relation between the
Ozsv\'ath-Szab\'o contact invariants of ($M_{\varphi},\xi_{\varphi}$)
and the pieces $(M_{\varphi^1_\Gamma},\xi_{\varphi^1_\Gamma}),\ldots,
(M_{\varphi^N_\Gamma},\xi_{\varphi^N_\Gamma})$.  Additionally,
Jeremy Van Horn-Morris and John Etnyre have pointed out to me that if one
also assumes every component of $\Sigma\setminus\Gamma$ to intersect
$\p \Sigma$, then one can replace the weak cobordism of
Theorem~\ref{thm:reducible} with a Stein cobordism.  This does not appear
to be possible if any component of $\Sigma\setminus\Gamma$ has its
full boundary in~$\Gamma$.
\end{remark}

\subsection{Etnyre's planarity obstruction}
\label{sec:Etnyre}

Let us say that a connected contact 
$3$-manifold $(M,\xi)$ is \defin{maximally cobordant to $S^3$} if there
exists a compact connected $4$-manifold $W$ with $\p W = S^3 \sqcup (-M)$
such that for every closed $2$-form $\Omega$ on~$M$ with $\Omega|_\xi > 0$,
there is a symplectic form~$\omega$ on~$W$ with $\omega|_{TM} = \Omega$
defining a weak symplectic cobordism from $(M,\xi)$ to $(S^3,\xi_0)$.
Theorem~\ref{thm:planar} says that
every planar contact manifold is maximally
cobordant to~$S^3$.  It turns out that this suffices to give an alternative
proof of the planarity obstruction in \cite{Etnyre:planar}*{Theorem~4.1}.

\begin{thm}
\label{thm:Etnyre}
Suppose $(M,\xi)$ is maximally cobordant to~$S^3$.  Then every 
connected weak
semi-filling of $(M,\xi)$ has connected boundary and negative-definite
intersection form.
\end{thm}
\begin{proof}
Let~$W_1$ be the compact $4$-manifold with $\p W_1 = S^3 \sqcup (-M)$
guaranteed by the assumption, and suppose $(W_0,\omega)$ is a weak filling
of $(M,\xi) \sqcup (M',\xi')$, where $(M',\xi')$ is some other contact
manifold, possibly empty.  If $W = W_0 \cup_M W_1$ is defined by gluing
these two along~$M$, then by assumption $\omega$ can be extended over~$W_1$
so that $(W,\omega)$ becomes a weak filling of $(S^3,\xi_0) \sqcup (M',\xi')$,
implying that $M'$ must be empty since $(S^3,\xi_0)$ is not weakly
co-fillable.  Now~$\omega$ is exact near $\p W = S^3$, so without loss
of generality we may assume $(W,\omega)$ is a strong filling of $(S^3,\xi_0)$.

We claim that the map induced on homology $H_2(W_0;\QQ) \to H_2(W;\QQ)$ by
the inclusion $\iota : W_0 \hookrightarrow W$ is injective.  Indeed, if
$A \in H_2(W_0;\QQ)$ satisfies $\int_A \omega \ne 0$, then obviously
$\int_{\iota_*A} \omega$ is also nonzero and thus $\iota_*A \ne 0 \in
H_2(W ; \QQ)$.  If $\int_A \omega = 0$ but $A \in H_2(W_0;\QQ)$ is nontrivial,
we can pick any closed $2$-form $\sigma$ on~$W_0$ with 
$\int_A \sigma \ne 0$ and replace~$\omega$ by $\omega + \epsilon\sigma$
for any $\epsilon > 0$ sufficiently small so that $(W_0,\omega+\epsilon\sigma)$
remains a weak filling of $(M,\xi)$.  Then $\omega + \epsilon\sigma$ also
extends over~$W_1$, so that the above argument goes through
again to prove that $\iota_*A$ is nontrivial.

Finally, we use the fact that the strong fillings of $(S^3,\xi_0)$ have been
classified: by a result of Gromov \cite{Gromov}
and Eliashberg \cite{Eliashberg:diskFilling}, $W$ is necessarily diffeomorphic
to a symplectic blow-up of the $4$-ball, i.e.
$$
W \cong B^4 \# \overline{\CC P^2} \# \ldots \# \overline{\CC P^2}.
$$
Since the latter has a negative-definite intersection form and
$\iota_* : H_2(W_0;\QQ) \to H_2(W;\QQ)$ is injective, the result follows.
\end{proof}

Our proof of Theorem~\ref{thm:planar} combined with 
Conjecture~\ref{conj:c} would
also reprove the algebraic planarity obstruction established in
\cite{Wendl:openbook2}, which is the twisted ECH version
of a Heegaard Floer theoretic result by
due to Oszv\'ath, Stipsicz and Szab\'o 
\cite{OzsvathSzaboStipsicz:planar}.  Note that the condition of
being maximally cobordant to~$S^3$ does not require $(M,\xi)$ to be fillable.
It is also not clear whether there can exist non-planar contact manifolds
that also satisfy this condition; the author is unaware of any known 
invariants that would be able to detect this distinction.

\begin{question}
\label{question:planar}
Is there a non-planar contact $3$-manifold which is maximally
cobordant to~$S^3$?
\end{question}

Note that if the assumption of Theorem~\ref{thm:Etnyre} is relaxed to
$(M,\xi) \cob (S^3,\xi_0)$, then the result becomes false: a counterexample
is furnished by the standard $3$-torus $(T^3,\xi_1)$, which admits a
cobordism to $(S^3,\xi_0)$ by Theorem~\ref{thm:toS3} but also is strongly
filled by
$T^*T^2$, whose intersection form is indefinite.  Assuming a strong filling
$(W_0,\omega)$ of $(M,\xi)$, the proof above
fails precisely at the point where the inclusion $W_0 
\hookrightarrow W$ is required to induce an \emph{injective} map
$H_2(W_0; \QQ) \to H_2(W;\QQ)$.  However, it still follows by the same
argument that $H_2(W_0;\QQ)$ cannot contain any class with strictly
positive square, hence we obtain the following weaker result with more
general assumptions---it applies in particular to all the contact manifolds
covered by Theorem~\ref{thm:toS3}.

\begin{thm}
\label{thm:nonpositive}
Suppose $(M,\xi)$ is a closed connected contact $3$-manifold with
$(M,\xi) \cob (S^3,\xi_0)$.
Then every strong semi-filling $(W,\omega)$ of $(M,\xi)$ has connected 
boundary and $b_2^+(W) = 0$.
\end{thm}

\subsection{Some remarks on planar torsion}
\label{sec:afterword}

The filling obstruction known as planar torsion was introduced in
\cite{Wendl:openbook2} with mainly holomorphic curves as motivation,
as it provides the most general setting known so far in which the existence
and uniqueness of certain embedded holomorphic curves leads to a vanishing
result for the ECH contact invariant.  In light of our cobordism
construction, however, one can now provide an alternative motivation for
the definition in purely symplectic topological terms.
The first step is to
understand what kinds of blown up summed open books automatically support
overtwisted contact structures: using Eliashberg's classification theorem
\cite{Eliashberg:overtwisted} and Giroux's criterion (cf.~\cite{Geiges:book}),
this naturally leads to the notion of planar $0$-torsion.  Then a more
general blown up summed open book defines a planar $k$-torsion domain
for some $k \ge 1$ if and only if it can be transformed into a planar
$0$-torsion domain by a sequence of $\DD$-capping and $\DD$-decoupling
surgeries; this is the essence of Proposition~\ref{prop:kTok-1} proved below.  
From this perspective, the definition of planar torsion and
the crucial role played by blown up summed open books seem completely
natural.  

More generally, the partially planar domains are precisely the
blown up summed open books for which a sequence of $\DD$-capping and
$\DD$-decoupling cobordisms can be used to construct a symplectic cap
that contains a symplectic sphere with square~$0$.  
As far as the author is aware,
almost all existing uniqueness or classification results for symplectic
fillings 
(e.g.~\cites{Wendl:fillable,Lisca:fillingsLens,OhtaOno:simpleSingularities})
apply to contact manifolds that admit caps of this type.  However, it does
not always suffice to construct an appropriate cap and then apply
McDuff's results \cite{McDuff:rationalRuled}: e.g.~the classification of
strong fillings of planar contact manifolds in terms of Lefschetz
fibrations \cites{Wendl:fillable,LisiVanhornWendl} truly relies on
\emph{punctured} holomorphic curves, as there is no obvious way to
produce a Lefschetz fibration with bounded fibers out of a family of
holomorphic spheres in a cap.

Finally, we remark that while Theorems~\ref{thm:overtwisted} and~\ref{thm:weak} 
substantially simplify the proof that planar torsion is a filling obstruction,
they do not reproduce all of the results in \cite{Wendl:openbook2}:
in particular the technology of Embedded Contact Homology is not yet far
enough along to deduce the vanishing of the contact invariant from a
non-exact cobordism.  Moreover, a proof using capping and decoupling
cobordisms simplifies the technology needed but does not remove it, as
a simplified version of the very same technology is required to 
prove the Gromov-Eliashberg
theorem (cf.~\S\ref{sec:GromovEliashberg}).  From the author's own
perspective, the idea for constructing symplectic cobordisms out of these
types of handles would never have emerged without a holomorphic curve
picture in the background (cf.~Figure~\ref{fig:Jhol}), 
and as we will discuss in \S\ref{sec:holomorphic},
after one has constructed the symplectic structure, it is practically no
extra effort to add a foliation by embedded $J$-holomorphic curves which
reproduces the $J$-holomorphic blown up open books of \cite{Wendl:openbook2}
on both boundary components.  The moral is that whether one prefers to
prove non-fillability results by direct holomorphic curve arguments or by
constructing cobordisms to reduce them to previously known results, it is
essentially the same thing: neither proof would be possible without the other.

\section{The details}
\label{sec:details}

The plan for proving the main results is as follows.  
We begin in \S\ref{sec:BUSOBD}
by reviewing the fundamental definitions involving blown up summed open
books and planar torsion, culminating with the (more or less obvious)
observation that one can
always use capping or decoupling surgery to decrease the order of a planar 
torsion domain.  In \S\ref{sec:model}, we introduce a useful concrete model
for a blown up summed open book and its supported contact structure.
This is applied in \S\ref{sec:modelCobordism} to write down a model of a weak
$\Sigma$-decoupling cobordism, and minor modifications explained in
\S\ref{sec:notRound} yield a similar model for the $\Sigma$-capping
cobordism.  This completes the cobordism construction for the case
where the negative boundary is strongly concave (or more generally when the
given symplectic form $\omega$ in Theorem~\ref{thm:2handles}
or~\ref{thm:roundHandles} is exact), see Remark~\ref{remark:concave}.
For the general case, we need
to show additionally that any given symplectic form
on $[0,1] \times M$ satisfying the necessary cohomological condition 
can be deformed so as to attach smoothly to the
model cobordisms we've constructed; 
this  is shown in \S\ref{sec:deformation}, thus completing the
proofs of Theorems~\ref{thm:2handles} and~\ref{thm:roundHandles}.  We
prove Theorems~\ref{thm:2handleCohomology} 
and~\ref{thm:roundHandleCohomology} in \S\ref{sec:cohomology},
answering the essentially cohomological question
of when the weak cobordism can be made strong, and when its symplectic
form is exact.  With these ingredients all in place, the proofs 
of the main results from \S\ref{sec:results} 
are completed in~\S\ref{sec:proofs}.  Finally,
\S\ref{sec:holomorphic} gives a brief discussion of the existence
and uniqueness of holomorphic curves in the cobordisms we've constructed.

\subsection{Review of summed open books and planar torsion}
\label{sec:BUSOBD}

The following notions were introduced in \cite{Wendl:openbook2}, and we
refer to that paper for more precise definitions and further discussion.

Assume $M$ is a compact oriented $3$-manifold, possibly with boundary, the
latter consisting of a union of $2$-tori.
A \defin{blown up summed open book} $\boldsymbol{\pi}$ on~$M$ can be 
described via the following data.
\begin{enumerate}
\item An oriented link $B \subset M \setminus \p M$, called the \defin{binding}.
\item A disjoint union of $2$-tori $\iI \subset M \setminus \p M$, called the
\defin{interface}.
\item For each interface torus $T \subset \iI$ a distinguished basis 
$(\lambda,\mu)$ of $H_1(T)$, where $\mu$ is defined only up to sign.
\item For each boundary torus $T \subset \p M$ a distinguished basis 
$(\lambda,\mu)$ of $H_1(T)$.
\item A fibration
$$
\pi : M \setminus (B \cup \iI) \to S^1
$$
whose restriction to $\p M$ is a submersion.
\end{enumerate}
The distinguished homology classes $\lambda, \mu \in H_1(T)$ associated
to each torus $T \subset \iI \cup \p M$ are called \defin{longitudes}
and \defin{meridians} respectively, and the oriented connected components
of the fibers $\pi^{-1}(\text{const})$ are called \defin{pages}.
We assume moreover that the fibration~$\pi$ can be expressed in the following
normal forms near the components of $B \cup \iI \cup \p M$.
As in an ordinary open book decomposition, each binding circle 
$\gamma \subset B$ has a neighborhood
admitting coordinates $(\theta,\rho,\phi) \in S^1 \times \DD$, where
$(\rho,\phi)$ are polar coordinates on the disk (normalized so that
$\phi \in S^1 = \RR / \ZZ$), such that $\gamma = \{ \rho = 0 \}$ and
\begin{equation}
\label{eqn:normalBinding}
\pi(\theta,\rho,\phi) = \phi.
\end{equation}
Near an interface torus $T \subset \iI$, we can find a
neighborhood with coordinates $(\theta,\rho,\phi) \in S^1 \times
[-1,1] \times S^1$ such that $T = \{\rho = 0 \} = S^1 \times \{0\} \times S^1$
with $(\lambda,\mu)$ matching the natural basis of 
$H_1(S^1 \times \{0\} \times S^1)$,  and
\begin{equation}
\label{eqn:normalInterface}
\pi(\theta,\rho,\phi) = \begin{cases}
\phi & \text{ for $\rho > 0$,}\\
-\phi & \text{ for $\rho < 0$.}
\end{cases}
\end{equation}
A neighborhood of a boundary torus $T \subset \p M$ similarly admits
coordinates $(\theta,\rho,\phi) \in S^1 \times [0,1] \times S^1$ with
$T = S^1 \times \{0\} \times S^1$ and
\begin{equation}
\label{eqn:normalBoundary}
\pi(\theta,\rho,\phi) = \phi.
\end{equation}
Observe that unlike the normal form \eqref{eqn:normalBinding}, the
map \eqref{eqn:normalBoundary} is well defined at $\rho=0$, since
there are no polar coordinates and hence no coordinate singularity.
The above conditions imply that the closure of each page is a smoothly
\emph{immersed} surface, whose boundary components are each embedded submanifolds of
$B$, $\iI$ or $\p M$, and in the last two cases homologous to the 
distinguished longitude~$\lambda$.  The ``generic'' page has an embedded
closure, but in isolated cases there may be pairs of boundary components
that are identical as oriented $1$-dimensional submanifolds in~$\iI$.

In general, any or all of $B$, $\iI$ and $\p M$ may be empty, and~$M$
may also be disconnected.
If $B \cup \iI \cup \p M = \emptyset$ we have simply a fibration 
$\pi : M \to S^1$ whose
fibers are closed oriented surfaces.  If $\iI \cup \p M = \emptyset$
but $B \ne \emptyset$ and $M$ is connected, we have an ordinary open book.  

We say that $\boldsymbol{\pi}$ is
\defin{irreducible} if the fibers $\pi^{-1}(\text{const})$ are connected,
i.e.~there is only one $S^1$-parametrized family of pages.  More generally,
any blown up summed open book can be presented uniquely as a union of
\defin{irreducible subdomains}
$$
M = M_1 \cup \ldots \cup M_N,
$$
which each inherit irreducible blown up summed open books and are attached
together along boundary tori (which become interface tori in~$M$).

The notion of a contact structure supported by an open book generalizes
in a natural way: we say that a contact structure~$\xi$ on~$M$ is 
\defin{supported} by $\boldsymbol{\pi}$ if it is the kernel of a \defin{Giroux form},
a contact form whose Reeb vector field is everywhere positively transverse
to the pages and positively tangent to their boundaries, and which
induces a characteristic foliation on $\iI \cup \p M$ with closed leaves 
parallel to the distinguished meridians.  A Giroux form exists and is
unique up to homotopy through Giroux forms on any
connected manifold with a blown up summed open book, except in the case
where the pages are
closed, i.e.~$B \cup \iI \cup \p M = \emptyset$.  The binding
is then a positively transverse link, and the interface and boundary are 
disjoint unions of pre-Lagrangian tori.

\begin{defn}
\label{defn:planar}
An irreducible blown up summed open book is called \defin{planar} if its
pages have genus zero.  An arbitrary blown up summed open book is then
called \defin{partially planar} if its interior contains a 
planar irreducible subdomain, which we call a \defin{planar piece}.
A \defin{partially planar domain} is a contact $3$-manifold $(M,\xi)$,
possibly with boundary, together with a supporting blown up summed open
book that is partially planar.  For a given closed $2$-form~$\Omega$
on~$M$, and a partially planar domain $(M,\xi)$ with planar piece
$M^P \subset M$, we say that $(M,\xi)$ is \defin{$\Omega$-separating} 
if $\int_T \Omega = 0$ for all interface tori~$T$ of~$M$ that lie in $M^P$,
and \defin{fully separating} if this is true for all~$\Omega$.
\end{defn}

\begin{defn}
\label{defn:symmetric}
A blown up summed open book is called \defin{symmetric} if it has empty
boundary, all its pages
are diffeomorphic and it contains exactly two irreducible subdomains
$$
M = M_+ \cup M_-,
$$
each of which has empty binding and interface.
\end{defn}

The simplest example of a symmetric summed open book is the one whose pages
are disks: this supports the tight contact structure on $S^1 \times S^2$
(cf.~Figure~\ref{fig:T3}, right).

\begin{defn}
\label{defn:planarTorsion}
For any integer $k \ge 0$, an $\Omega$-separating 
partially planar domain $(M,\xi)$ with planar piece $M^P \subset M$ 
is called an \defin{$\Omega$-separating planar $k$-torsion domain} 
if it satisfies the following conditions:
\begin{itemize}
\setlength{\itemsep}{0cm}
\item $(M,\xi)$ is not symmetric.
\item $\p M^P \ne \emptyset$.
\item The pages in $M^P$ have $k+1$ boundary components.
\end{itemize}
The (necessarily nonempty) subdomain $\overline{M \setminus M^P}$ 
is then called the \defin{padding}.

We say that a contact manifold $(M,\xi)$ with closed $2$-form $\Omega$
has \defin{$\Omega$-separating planar $k$-torsion}
if it contains an $\Omega$-separating planar $k$-torsion domain.
If this is true for all closed $2$-forms $\Omega$ on~$M$, then we say
$(M,\xi)$ has \defin{fully separating} planar $k$-torsion.
\end{defn}

It was shown in \cite{Wendl:openbook2} that a contact manifold is overtwisted
if and only if it has planar $0$-torsion, which is always fully separating
since the interface then intersects the planar piece only at its boundary,
a single nullhomologous torus.
The proofs of Theorems~\ref{thm:overtwisted} and~\ref{thm:weak} thus rest
on the following easy consequence of the preceeding definitions.

\begin{prop}
\label{prop:kTok-1}
If $M$ is a planar $k$-torsion domain for some $k \ge 1$, then it
contains a binding circle~$\gamma$ or interface torus~$T$ in its planar
piece such that the
following is true.  Let $M'$ denote the manifold with corresponding
blown up summed open book obtained from~$M$ by $\DD$-capping surgery
along~$\gamma$ or $\DD$-decoupling surgery along~$T$ respetively.
Then some connected component of~$M'$ is a planar $\ell$-torsion domain
for some $\ell \in \{k-2,k-1\}$.
\end{prop}
\begin{proof}
By assumption, $M$ contains a planar piece $M^P$ with nonempty boundary, 
and if $T_0 \subset \p M^P$ denotes a boundary component, then the
pages in~$M^P$ have exactly one boundary component adjacent to~$T_0$.
The pages in $M^P$ have $k+1$ boundary components, and without loss
of generality we may assume no other irreducible subdomain in the interior
of~$M$ has planar pages with fewer boundary components than this.  
Since $k \ge 1$, these pages have at least one boundary component
adjacent to some binding circle~$\gamma$ or interface torus~$T$ distinct
from~$T_0$.  Performing $\DD$-capping sugery to remove~$\gamma$ 
or $\DD$-decoupling surgery to remove~$T$ produces a new
manifold $M_1$ containing a planar irreducible subdomain $M^P_1$ whose pages
have~$\ell$ boundary components where~$\ell$ is either~$k$ or $k-1$; the
latter can only result from a decoupling surgery along $T \subset M^P
\setminus \p M^P$.  Thus $M_1$ is a planar
$(\ell-1)$-torsion domain unless it is symmetric.
The latter would mean $\p M_1 = \emptyset$, hence also $\p M = \emptyset$,
and $\overline{M_1 \setminus M^P_1}$ is also irreducible and has planar pages
with~$\ell$ boundary components.  This cannot arise from capping
surgery along a binding circle or decoupling
surgery along a torus in the interior of $M^P$, as we assumed all
planar pages in the interior of~$M$ outside of $M^P$ to have at least
$k+1 \ge \ell + 1$ boundary components.  The only remaining possibility would be
decoupling surgery along $T \subset \p M^P$, but then symmetry of~$M_1$
would imply that $M$ must also have been symmetric, hence a contradiction.
\end{proof}

\subsection{A model for a blown up summed open book}
\label{sec:model}

Assume~$(M_0,\xi)$ is a compact contact $3$-manifold, possibly 
with boundary, supported by a blown up summed open book $\boldsymbol{\pi}$
with binding~$B$, interface~$\iI$ and fibration
$$
\pi : M_0 \setminus (B \cup \iI) \to S^1.
$$
We assume that each connected component of~$M_0$ contains at least one
component of $B \cup \iI \cup \p M_0$, so that $\boldsymbol{\pi}$ will
support a contact structure everywhere.
It will be useful to identify this with the following
generalization of the notion of an \emph{abstract} open
book (cf.~\cite{Etnyre:lectures}).  The closure of a fiber
$\overline{\pi^{-1}(\text{const})} \subset M_0$ is the image of some 
compact oriented surface $S$ with boundary 
under an immersion
$$
\iota : S \looparrowright M_0,
$$
which is an embedding on the interior.  
The monodromy of the fibration then determines (up to isotopy) 
a diffeomorphism $\psi : S \to S$
which preserves connected components and is the identity 
in a neighborhood of the boundary, and we define the mapping torus
$$
S_\psi = (S \times \RR) / \sim
$$
with $(z,t + 1) \sim (\psi(z),t)$ for all $t \in \RR$, $z \in S$.
Denote by
$$
\phi : S_\psi \to \RR / \ZZ = S^1
$$
the natural fibration.  

Let us label the connected components of $\p S$ by
$$
\p S = \p_1 S \cup \ldots \cup \p_n S,
$$
and for each $i=1,\ldots,n$ choose an open collar neighborhood 
$\uU^i \subset  S$ of
$\p_i S$ on which $\psi$ is the identity.  Denote the union of
all these neighborhoods by~$\uU \subset  S$.  Now for each $i=1,\ldots,n$, 
choose positively oriented coordinates
$$
(\theta,\rho) : \uU^i \to S^1 \times [r,1)
$$
for some $r \in (0,1)$.  These neighborhoods give rise to 
corresponding collar neighborhoods of $\p  S_\psi$,
$$
\uU^i_\psi = \uU^i \times S^1 \subset  S_\psi,
$$
which can be identified with $S^1 \times [r,1) \times S^1$
via the coordinates $(\theta,\rho,\phi)$.
The index set $I := \{1,\ldots,n\}$ comes with an obvious partition
$$
I = I_B \cup I_\iI \cup I_\p,
$$
where
\begin{equation*}
\begin{split}
I_B &= \{ i \in I \ |\ \iota(\p_i S) \subset B  \}, \\
I_\iI &= \{ i \in I \ |\ \iota(\p_i S) \subset \iI  \}, \\
I_\p &= \{ i \in I \ |\ \iota(\p_i S) \subset \p M_0  \}.
\end{split}
\end{equation*}
There is also a free $\ZZ_2$-action on $I_\iI$ defined via an involution
$$
\sigma : I_\iI \to I_\iI
$$
such that $j = \sigma(i)$ if and only if $\iota(\p_i  S)$ and
$\iota(\p_j S)$ lie in the same connected component of~$\iI$.
Now define for each $i \in I$ the domain
$$
\nN_i = 
\begin{cases}
S^1 \times \DD & \text{ if $i \in I_B$,}\\
S^1 \times [-1,1] \times S^1 & \text{ if $i \in I_\iI$,}\\
S^1 \times [0,1] \times S^1 & \text{ if $i \in I_\p$,}
\end{cases}
$$
and denote by $(\theta,\rho,\phi)$ the natural coordinates on~$\nN_i$,
where for $i \in I_B$ we view $(\rho,\phi)$ as polar
coordinates on the disk with the angle normalized to take values
in $S^1 = \RR / \ZZ$.  Denote the subsets $\{ \rho=0 \}$ by
$$
B^\Abstract = \bigsqcup_{i \in I_B} S^1 \times \{0\} \subset
\bigsqcup_{i \in I_B} \nN_i,
\qquad
\iI^\Abstract = \bigsqcup_{i \in I_\iI} S^1 \times \{0\} \times S^1 \subset
\bigsqcup_{i \in I_\iI} \nN_i.
$$
The chosen coordinates on the neighborhoods $\uU^i_\psi$ then 
determine a gluing map
$$
\Phi : \bigcup_{i \in I} \uU^i_\psi \to \bigsqcup_{i\in I} \nN_i
$$
which takes $\uU^i_\psi$ to $\nN_i$, 
and we use this to define a new compact and oriented manifold, 
possibly with boundary,
$$
M_0^\Abstract = 
 S_\psi \cup_\Phi \left( \bigsqcup_{i\in I} \nN_i \right) \Bigg/ \sim,
$$
where the equivalence relation identifies $(\theta,\rho,\phi) \in
\nN_i$ for $i \in I_\iI$ with $(\theta,-\rho,-\phi) \in \nN_{\sigma(i)}$.
This naturally contains $B^\Abstract$ and $\iI^\Abstract$ as submanifolds,
and the fibration $\phi :  S_\psi \to S^1$ can be extended
over $M_0^\Abstract \setminus (B^\Abstract \cup \iI^\Abstract)$ so that
it matches the canonical $\phi$-coordinate on $\nN_i$ wherever $\rho > 0$.
Now $M_0$ can be identified with $M_0^\Abstract$ via a diffeomorphism that 
maps $B$ to $B^\Abstract$ and $\iI$ to $\iI^\Abstract$, and transforms
the fibration $\pi : M_0 \setminus (B \cup \iI) \to S^1$ to~$\phi$.

A supported contact structure on~$M_0^\Abstract$ can be defined as follows.
First, define a smooth $1$-form of the form
$$
\lambda_0 = \begin{cases}
d\phi & \text{ on $ S_\psi$,} \\
f_i(\rho)\, d\theta + g_i(\rho)\, d\phi & \text{ on $\nN_i$, 
$i \in I$,}
\end{cases}
$$
where $f_i , g_i : [0,1] \to \RR$ are smooth functions 
chosen to have the following properties:
\begin{enumerate}
\item As $\rho$ moves from~$0$ to~$1$, $\rho \mapsto 
(f_i(\rho),g_i(\rho)) \in \RR^2 \setminus \{0\}$ defines a path
through the first quadrant from $(1,0)$ to~$(0,1)$.
\item $\lambda_0$ is contact on $\{ 0 \le \rho < r \} \subset \nN_i$.
\item $f_i(\rho) = 0$ for $\rho \in [r,1]$.
\item $g_i(\rho) = 1$ for $\rho \in [r',1]$, for some positive number
$r' < r$.
\item $g_i'(\rho) > 0$ for $\rho \in (0,r')$.
\end{enumerate}
\begin{remark}
\label{remark:aremark}
The contact condition is satisfied if and only if
$f_i g_i' - f_i'g_i \ne 0$, except at $B^\Abstract$, where the coordinate singularity
changes the condition to $g_i''(0) \ne 0$.  One consequence is that
$f_i'(\rho) < 0$ for $\rho \in [r',r)$, hence
$f_i(r') > 0$.
The assumption that $\lambda_0$ is a \emph{smooth} $1$-form imposes some
additional conditions, namely for $i \in I_B$, $(\rho,\phi) \mapsto f_i(\rho)$
and $(\rho,\phi) \mapsto g_i(\rho) / \rho^2$ must define smooth functions
at the origin in~$\RR^2$ (in polar coordinates),
and for $i \in I_\iI$, $f_i$ and $g_i$ can be extended 
smoothly over $[-1,1]$ such that
$$
f_i(\rho) = f_{\sigma(i)}(-\rho),\qquad
g_i(\rho) = - g_{\sigma(i)}(-\rho).
$$
In particular this implies $(f_i(\rho),g_i(\rho)) = (0,-1)$ for
$\rho \in [-1,-r]$.  We will assume these conditions are always satisfied 
without further comment.
\end{remark}
The co-oriented distribution
$$
\xi_0 := \ker\lambda_0
$$
is a confoliation on $M_0^\Abstract$, which is integrable on the 
mapping torus $S_\psi$ and outside of this is a positive contact
structure.  To perturb it to a global contact structure, choose a
$1$-form $\alpha$ on $S$ which satisfies $d\alpha > 0$ and takes the form
\begin{equation}
\label{eqn:alphaBoundary}
\alpha = (2 - \rho)\, d\theta
\end{equation}
on $\uU^i$.  By a simple
interpolation trick (cf.~\cite{Etnyre:lectures}), $\alpha$ can be
used to construct a $1$-form $\alpha_\psi$ on $ S_\psi$ that
satisfies
$$
d\alpha_\psi|_{\xi_0} > 0 \quad \text{ and }\quad
\alpha_\psi = (2 - \rho)\, d\theta \text{ on $\uU^i_\psi$}.
$$
Choosing $\epsilon > 0$ sufficiently small,
we can bring $\ker(d\phi + \epsilon\alpha_\psi)$ sufficiently
$C^0$-close to $\xi_0$ on $ S_\psi$ so that 
$d\alpha_\psi|_{\ker(d\phi + \epsilon\alpha_\psi)} > 0$.
Then a contact form that equals $\lambda_0$ near $\p M_0^\Abstract$
can be defined by
\begin{equation}
\label{eqn:lambdaEpsilon}
\lambda_\epsilon = 
\begin{cases}
d\phi + \epsilon\, \alpha_\psi & \text{ on $S_\psi$,} \\
f_{i,\epsilon}(\rho)\, d\theta + d\phi & \text{ on $\{ \rho \in [r',
r] \} \subset \nN_i$,} \\
\lambda_0 & \text{ on $\{ \rho \le r'\}
\subset \nN_i$,}
\end{cases}
\end{equation}
where the fact that $f_i(r') > 0$ allows us for $\epsilon > 0$
sufficiently small to choose smooth
functions $f_{i,\epsilon} : [0,1] \to \RR$ satisfying
\begin{itemize}
\item $f_{i,\epsilon}(\rho) = f_i(\rho)$ for $\rho \in [0,r']$,
\item $f_{i,\epsilon}' < 0$ for $\rho \in [r',r]$,
and
\item $f_{i,\epsilon}(\rho) = \epsilon (2 - \rho)$ for $\rho \in
[r,1]$.
\end{itemize}
Note that for $i \in I_\iI$, $f_{i,\epsilon}$ also extends naturally over
$[-1,1]$ with $f_{i,\epsilon}(\rho) = f_{\sigma(i),\epsilon}(-\rho)$.
All contact forms that one can construct in this way are homotopic to each
other through families of contact forms, so the resulting contact structure
$$
\xi_\epsilon := \ker\lambda_\epsilon
$$
is uniquely determined up to isotopy.  Moreover, it is easy to check that
the Reeb vector field determined by $\lambda_\epsilon$ is everywhere 
positively transverse to the pages: in particular, $\lambda_\epsilon$ is a 
Giroux form for the blown up summed open book we've constructed on 
$M_0^\Abstract$, thus $(M_0^\Abstract,\xi_\epsilon)$ is
contactomorphic to $(M_0,\xi)$.

\subsection{A model decoupling cobordism}
\label{sec:modelCobordism}

Assume now that the manifold $M_0$ from the previous section is 
embedded into a closed contact $3$-manifold $(M,\xi)$
such that~$\xi$ is an extension of the contact structure that was
given on~$M_0$.  Without loss of generality, we can identify~$(M_0,\xi)$ with
the abstract model $(M_0^\Abstract,\xi_\epsilon)$, and assume
in particular that $\lambda_0$ and $\lambda_\epsilon$ are $1$-forms on~$M$
which restrict on~$M_0$ to the models constructed above, and on a
neighborhood of $\overline{M \setminus M_0}$ define matching contact forms 
whose kernel is~$\xi$.

Our goal in this section is to construct a weak symplectic cobordism that 
realizes a $\Sigma$-decoupling surgery along some set of
oriented interface tori
$$
\iI_0 = T_1 \cup \ldots \cup T_N \subset \iI.
$$
The chosen orientation of each~$T_j$ splits a tubular neighborhood
$\nN(T_j) \subset M$ of~$T_j$ naturally into positive and negative parts
$$
\nN(T_j) = \nN_-(T_j) \cup \nN_+(T_j)
$$
whose intersection is~$T_j$.
To simplify notation in the following, let us assume these neighborhoods are
chosen and the page boundary 
components $\p S = \p_1 S \cup \ldots \cup \p_n S$ are ordered so that
for each $j =1,\ldots,N$,
$$
\nN(T_j) = \nN_j = S^1 \times [-1,1] \times S^1
\quad \text{ and } \quad
\nN_+(T_j) = S^1 \times [0,1] \times S^1.
$$
We will fix on $\nN(T_j)$ the standard coordinates $(\theta,\rho,\phi)$ 
of~$\nN_j$, and assume all the functions chosen to define $\lambda_0$
and $\lambda_\epsilon$ are the same for all of these neighborhoods, so
we can write
$$
f = f_j, \ g = g_j, \ f_\epsilon = f_{j,\epsilon}
$$
for $j=1,\ldots,N$.

For Theorems~\ref{thm:2handles} and~\ref{thm:roundHandles},
the cobordism we construct will need to be attached to a trivial
cobordism of the form $([0,1] \times M, \omega)$, which will be 
impossible if our model symplectic form does not match $\omega$ at
least cohomologically at $\{1\} \times M$.  In order to realize the
right cohomology class in the model, we choose a closed $2$-form
$\Omega_0$ on~$M$ representing an arbitrary cohomology class for which the
condition \eqref{eqn:homological} is satisfied.  Since we only care about
$\Omega_0$ up to cohomology, we are free to add an exact $2$-form and thus
assume $\Omega_0$ satisfies
$$
\Omega_0 = c_j\, d\phi \wedge d\theta \quad \text{ on $\nN(T_j)$}
$$
for each $j=1,\ldots,N$, where $c_j \in \RR$ are constants satisfying
\begin{equation}
\label{eqn:homologicalcj}
\sum_{j=1}^N c_j = 0.
\end{equation}
Since $\lambda_0 \wedge d\lambda_\epsilon > 0$ everywhere on~$M$,
we can define an exact symplectic form on the trivial cobordism
$[0,1] \times M$ as follows: fix any smooth, strictly increasing function
$\varphi : [0,1] \to \RR$ with $\varphi(0)=0$ and
$|\varphi(t)|$ uniformly small, and set
\begin{equation}
\label{eqn:omegaSHS}
\omega_0 = d\left( \varphi(t) \lambda_0 + \lambda_\epsilon\right).
\end{equation}
If $\| \varphi \|_{L^\infty}$ is sufficiently small then $\omega_0$ is
symplectic and restricts positively to both~$\xi$ and the
pages of $\boldsymbol{\pi}$, everywhere on $[0,1] \times M$.
Now if $C > 0$ is a sufficiently large constant, then the $2$-form
\begin{equation}
\label{eqn:omegaC}
\omega_C := C\omega_0 + \Omega_0
\end{equation}
also has these properties.  In the following we shall always assume
$C$ is arbitrarily large whenever convenient.
Note that for the case of Theorems~\ref{thm:2handles}
and~\ref{thm:roundHandles} where the given $\omega$ on 
$[0,1] \times M$ is exact, we may assume without loss of generality
that $\Omega_0 \equiv 0$, see Remark~\ref{remark:concave}.

To construct a cobordism corresponding to the round handle attachment,
we shall first ``dig a hole'' in the trivial cobordism $[0,1]\times M$ near 
each of the tori $\{1\} \times T_j$.  In order to find nice coordinates near
the boundary of the hole, it will be useful to consider the vector field
$X_\theta$ on $[0,1] \times \nN(T_j)$ defined by the condition
$$
\omega_0(X_\theta,\cdot) = -d\theta.
$$

\begin{lemma}
\label{lemma:Xtheta}
The vector field $X_\theta$ is locally Hamiltonian with respect
to~$\omega_C$ and takes the form
\begin{equation}
\label{eqn:AB}
X_\theta = A(t,\rho)\, \p_t + B(t,\rho)\, \p_\rho
\end{equation}
for some smooth functions $A, B : [0,1] \times [1,1] \to \RR$ with the
following properties:
\begin{enumerate}
\item For $\pm \rho \in [r,1]$, $A(t,\rho)=0$ and 
$B(t,\rho) = \pm \frac{1}{\epsilon}$.
\item For $\pm \rho \in [r',r]$, $A(t,\rho)=0$ and $\pm B(t,\rho) > 0$.
\item For $\rho \in (-r',r')$, $A(t,\rho) < 0$.
\end{enumerate}
\end{lemma}
\begin{proof}
By a direct computation, $X_\theta$ takes the form \eqref{eqn:AB} with
$A$ and~$B$ satisfying the linear system
$$
\begin{pmatrix}
-\varphi'(t) f(\rho)  & - \left[ \varphi(t) f'(\rho) +
f_\epsilon'(\rho) \right] \\
\varphi'(t) g(\rho) & \left[ \varphi(t) + 1 \right] g'(\rho)
\end{pmatrix}
\begin{pmatrix}
A(t,\rho) \\
B(t,\rho)
\end{pmatrix}
=
\begin{pmatrix}
1 \\
0
\end{pmatrix} .
$$
The determinant~$\Delta(t,\rho)$ 
of this matrix is always negative since the contact
condition requires $f(\rho) g'(\rho) - f'(\rho)g(\rho) > 0$ for
$|\rho| < r$, and for $\pm \rho \in [r,1]$ we have $g(\rho)= \pm 1$,
$\pm f'(\rho) \le 0$ and $\pm f'_\epsilon(\rho) < 0$.  The general
solution for $A$ and~$B$ can thus be written as
$$
\begin{pmatrix}
A(t,\rho) \\
B(t,\rho)
\end{pmatrix} = \frac{1}{\Delta(t,\rho)}
\begin{pmatrix}
\left[ \varphi(t) + 1 \right] g'(\rho) \\
- \varphi'(t) g(\rho) 
\end{pmatrix}.
$$
The stated conditions on $A(t,\rho)$ and $B(t,\rho)$ then follow 
immediately from the conditions we've placed on $f$, $g$, $f_\epsilon$
and~$\varphi$.

In light of \eqref{eqn:AB}, $X_\theta$ is in the kernel of
$d\phi \wedge d\theta$, and we conclude easily that it is locally
Hamiltonian since
$$
\Lie_{X_\theta}\omega_C = d \iota_{X_\theta}(C\omega_0 + 
c_j\,d\phi \wedge d\theta) = d\left( -C\,d\theta\right) = 0.
$$
\end{proof}

\begin{figure}
\begin{center}
\includegraphics{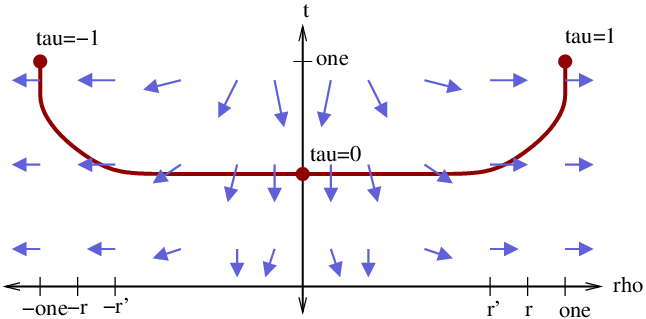}
\caption{\label{fig:hole}
The path $(t(\tau),\rho(\tau))$ transverse to the vector field
of \eqref{eqn:AB}.}
\end{center}
\end{figure}

Due to the lemma, we can choose a smoothly embedded curve
$$
[-1,1] \to [1/2,1] \times [-1,1] : \tau \mapsto (t(\tau),\rho(\tau))
$$
that is everywhere transverse to the vector field \eqref{eqn:AB} and
also satisfies $(t(0),\rho(0)) = (1/2,0)$ and
$$
(t(\tau),\rho(\tau)) = (\pm\tau,\pm 1)
$$
near $\tau = \pm 1$ (see Figure~\ref{fig:hole}).
Writing the annulus as $\AA = [-1,1] \times S^1$, use the curve
just chosen to define an embedding 
$$
\Psi : S^1 \times \AA \hookrightarrow [0,1] \times \nN(T_j) :
(\theta,\tau,\phi) \mapsto (t(\tau),\theta,\rho(\tau),\phi),
$$
which traces out a smooth hypersurface $H_{T_j} \subset [0,1] \times \nN(T_j)$
that meets $\{1\} \times M$ transversely at the pair of tori 
$\{1\} \times \p \nN(T_j)$.  Denote by
$$
\uU_{T_j} \subset [0,1] \times M
$$
the interior of the component of $([0,1] \times M) \setminus H_{T_j}$
that contains $\{1\} \times T_j$ (see Figure~\ref{fig:hole2}).  
Observe that by construction, $\uU_{T_j}$ 
lies entirely within $[1/2,1] \times \nN(T_j)$, and the locally
Hamiltonian vector field
$X_\theta$ points transversely outward at $\p\overline{\uU}_T = H_{T_j}$.
Thus for sufficiently small $\delta > 0$, we can use the 
flow $\varphi^t_{X_\theta}$ of $X_\theta$ to parametrize a neighborhood of 
$H_{T_j}$ in $\overline{\uU}_T$ by an embedding
\begin{equation*}
\begin{split}
\widetilde{\Psi} : (1 - \delta,1] \times S^1 \times \AA &\hookrightarrow
[0,1] \times M \\
(\sigma,\theta,\tau,\phi) &\mapsto \varphi_{X_\theta}^{\sigma-1}(\Psi(\theta,
\tau,\phi)).
\end{split}
\end{equation*}

\begin{lemma}
\label{lemma:pullback}
We have 
\begin{equation}
\label{eqn:symp}
\widetilde{\Psi}^*\omega_0 = - d\left( \sigma\, d\theta\right) + d\eta,
\end{equation}
where $\eta$ is an $S^1$-invariant $1$-form on $S^1 \times \AA$ that satisfies
$$
\eta = \pm \left[ \varphi(\pm\tau) + 1 \right] \, d\phi
\quad\text{ near $\{\tau = \pm 1\} = \p (S^1 \times \AA)$},
$$
and $d\theta \wedge d\eta > 0$ everywhere.
\end{lemma}
\begin{proof}
In $[0,1] \times \nN(T)$ we can write $\omega_0 = d\Lambda$, where
$$
\Lambda := \varphi(t)\, \lambda_0 + \lambda_\epsilon - \epsilon\, d\theta.
$$
Then defining $\eta := \Psi^*\Lambda$ on $S^1 \times \AA$, 
we have $d\eta = \Psi^*\omega_0$ and can write~$\eta$ explicitly
near $\tau = \pm 1$ by plugging in $t = \pm \tau$, $\rho= \pm 1$, 
$f(\rho) = 0$, $g(\rho) = \pm 1$
and $f_\epsilon(\rho) = \epsilon (2 - \rho) = \epsilon$, hence
\begin{equation*}
\begin{split}
\eta &= \Psi^*\left(\varphi(t)\, \lambda_0 + \lambda_\epsilon - 
\epsilon\, d\theta  \right) = 
\varphi(\pm\tau)( \pm d\phi) + \epsilon\, d\theta \pm d\phi - \epsilon\, d\theta \\
&= \pm \left[ \varphi(\pm\tau) + 1 \right] \, d\phi
\end{split}
\end{equation*}
as desired.  Since $\Lambda$ is invariant under the $S^1$-action by
translation of~$\theta$, $\eta$ is also $S^1$-invariant.
The claim $d\theta \wedge d\eta > 0$ is a consequence of the
fact that $H_{T_j}$ is transverse to the vector field $X_\theta$, which
is $\omega_0$-dual to $-d\theta$: indeed, ignoring 
combinatorial factors we find
\begin{equation*}
\begin{split}
d\theta \wedge d\eta(\p_\theta,\p_\tau,\p_\phi) &=
-\Psi^*(\iota_{X_\theta}\omega_0) \wedge \Psi^*\omega_0(\p_\theta,\p_\tau,\p_\phi)\\
&\propto -\omega_0 \wedge \omega_0(X_\theta,\Psi_*\p_\theta,\Psi_*\p_\tau,
\Psi_*\p_\phi) \ne 0.
\end{split}
\end{equation*}
It follows that $d\theta \wedge d\eta$ is positive since this is obviously
true near $\tau = \pm 1$.  The formula 
\eqref{eqn:symp} now follows from the fact that 
$-d\theta = \iota_{X_\theta}\omega_0$ and $X_\theta$ has a symplectic flow.
\end{proof}

\begin{remark}
The embedding $\widetilde{\Psi}$ reverses orientations.  This will be
convenient in the following since the handle $\roundhandle_\Sigma =
-\Sigma \times \AA$ also comes with a reversed orientation.
\end{remark}

Since $\widetilde{\Psi}$ acts trivially on the coordinates~$\phi$
and~$\theta$, the lemma also yields a formula for the pullback of $\omega_C$,
namely
\begin{equation}
\label{eqn:sympC}
\widetilde{\Psi}^*\omega_C = C\left[- d\left( \sigma\, d\theta\right) + d\eta
\right] + c_j\, d\phi \wedge d\theta.
\end{equation}
For each $j=1,\ldots,N$, denote by $\wW_j \subset [0,1] \times \nN(T_j)$
the image of the map $\widetilde{\Psi}$ as constructed above
(Figure~\ref{fig:hole2}): $\wW_j$ thus
inherits negatively oriented coordinates 
$(\sigma,\theta,\tau,\phi) \in (1-\delta,1] \times
S^1 \times [-1,1] \times S^1$ in which $\omega_C$ has the form given
in \eqref{eqn:sympC}.

\begin{figure}
\begin{center}
\includegraphics{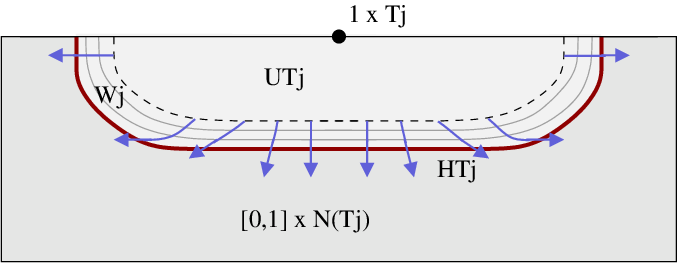}
\caption{\label{fig:hole2}
Digging a hole in $[0,1] \times \nN(T_j)$ near $\{1\} \times T_j$.}
\end{center}
\end{figure}

We are now ready to write down a smooth model of the round handle attachment.
As in \S\ref{sec:intro}, assume $\Sigma$ is a compact, connected and
oriented surface with~$N$ boundary components
$$
\p\Sigma = \p_1\Sigma \cup \ldots \cup \p_N\Sigma.
$$
Near each component $\p_j\Sigma$, identify a collar neighborhood 
$\vV_j \subset \Sigma$ with
$(1 - \delta,1] \times S^1$ and denote the resulting oriented coordinates by
$(\sigma,\theta)$.  Then denote the union of all the subsets
$\uU_{T_j}$ by $\uU_{\iI_0}$ and define the cobordism
$$
W = \left(([0,1] \times M) \setminus \uU_{\iI_0}\right) \cup
(-\Sigma \times \AA)
$$
by removing $\uU_{\iI_0}$ from $[0,1] \times M$ and replacing it by 
the handle $-\Sigma \times \AA$, gluing $\vV_j \times \AA$ to $\wW_j$
via the natural identification of the coordinates $(\sigma,\theta,\tau,\phi)$.
This yields a smooth $4$-manifold with two boundary components 
$$
\p W = M' \sqcup (-M),
$$
where we identify $M$ with $\{0\} \times M$ and write
$$
M' = ((\{1\} \times M) \setminus \nN(\iI_0)) \cup ( -\Sigma_+ 
\times S^1) \cup (-\Sigma_- \times S^1),
$$
using the identification $\Sigma \times \p\AA = (\Sigma_+\times S^1) \sqcup
(\Sigma_- \times S^1)$ defined in \eqref{eqn:Sigmapm}.
The oriented surfaces $-(\Sigma_+ \sqcup \Sigma_-) \times \{\phi\}$
now glue together smoothly
with the fibers $\pi^{-1}(\phi)$ in $M \setminus \nN(\iI_0)$ to form the
pages of the natural blown up summed open book $\boldsymbol{\pi}'$ on~$M'$
obtained from $\boldsymbol{\pi}$ by $\Sigma$-decoupling surgery along~$\iI_0$.
It remains to define a suitable symplectic form on $\Sigma \times \AA$ that
matches \eqref{eqn:sympC} near $\p\Sigma \times \AA$ and is positive on
these pages.

\begin{lemma}
\label{lemma:extendingOmega0}
There exists a symplectic form on $-\Sigma \times \AA$ that matches
\eqref{eqn:symp} near $\p\Sigma \times \AA$ and is positive on the
oriented surfaces $\{p\} \times \AA$ for any 
$p \in \Sigma \setminus (\vV_1 \cup \ldots \cup \vV_N)$
and $-\Sigma \times \{(\tau,\phi)\}$ for any $(\tau,\phi) \in \AA$, and makes
$T(\Sigma \times\{*\})$ and $T(\{*\} \times \AA)$ into symplectically
orthogonal symplectic subspaces everywhere along $\Sigma \times \p\AA$.
\end{lemma}
\begin{proof}
We will use a standard deformation trick to simplify \eqref{eqn:symp}
on each of the regions $\vV_j \times \AA$ so that it can be extended
as a split symplectic form.
Choose a $1$-form $\eta_0$ on~$\AA$ with $d\eta_0 > 0$
and lift it in the obvious way to $S^1 \times \AA$.  Since
$$
\int_{\{*\} \times \p \AA} \eta = 2\left[ \varphi(1) + 1 \right] > 0
$$
and $\eta$ has no $d\theta$-term near~$S^1 \times \p\AA$, we can also arrange
for $\eta_0$ to match $\eta$ on a neighborhood of $S^1 \times \p\AA$.
Next choose a smooth cutoff function $\tilde{\beta} : (1 - \delta,1] \to [0,1]$
that satisfies $\tilde{\beta}(\sigma)=0$ near $\sigma=1 - \delta$ and 
$\tilde{\beta}(\sigma)=1$ near $\sigma=1$, and use this to define a
smooth function $\beta : \Sigma \to [0,1]$ which satisfies
$$
\beta(\sigma,\theta) = \tilde{\beta}(\sigma) \text{ on $\vV_j$,}
\qquad
\beta \equiv 0 \text{ on $\Sigma \setminus (\vV_1 \cup \ldots \cup \vV_N)$}.
$$
We observe that the expression
$$
\beta \eta + (1 - \beta)\, \eta_0
$$
now gives a well-defined $1$-form on $\Sigma \times \AA$ by lifting
$\eta_0$ from~$\AA$ to $\Sigma \times \AA$ 
and $\eta$ from $S^1 \times \AA$ to $(1-\delta,1] \times S^1 \times \AA
= \vV_j \times \AA$ in the obvious ways.

Choose also a smooth function $\psi : (1 - \delta,1] \to [1,\infty)$
satisfying $\psi' > 0$ and $\psi'(\sigma)=1$ near $\sigma = 1$, and a
$1$-form $\mu$ on~$\Sigma$ such that
$$
\mu = \psi(\sigma)\, d\theta \text{ in $\vV_j$,} \qquad
d\mu > 0 \text{ everywhere.}
$$
A suitable symplectic form on $\Sigma \times \AA$ can then be defined by
\begin{equation}
\label{eqn:Omega'}
\omega_0' = - d\mu + d \big( \beta \eta + (1 - \beta)\, \eta_0 \big).
\end{equation}
By construction, $\omega_0'$ matches \eqref{eqn:symp} near
$\p \Sigma \times \AA$, while near $\Sigma \times \p\AA$ and outside of
the regions $\vV_j \times \AA$ it takes the split form
$$
- d \mu + d\eta_0,
$$
which is symplectic and makes each of
$T(\Sigma \times \{*\})$ and $T(\{*\} \times \AA)$
into symplectic subspaces
which are symplectically orthogonal to each other.
To test whether $\omega_0'$ is symplectic on $\vV_j \times \AA$, we compute
\begin{equation*}
\begin{split}
\frac{1}{2}\omega_0' \wedge \omega_0' &= \psi'\, d\theta \wedge d\sigma \wedge
\left[ \beta\, d\eta + (1 - \beta)\, d\eta_0 \right] \\
&+ \beta \beta' \, d\sigma \wedge (\eta - \eta_0) \wedge \left[
\beta\, d\eta + (1 - \beta)\, d\eta_0 \right].
\end{split}
\end{equation*}
The first term is always nonzero since $d\theta \wedge d\eta$ and
$d\theta \wedge d\eta_0$ are both positive.
The whole expression is therefore nonzero whenever either $\beta'(\sigma)=0$
or $\psi'(\sigma)$ is sufficiently large, and we are free to choose~$\psi$
so that it increases fast in the region where $\beta$ is not constant.
This choice also ensures $\omega_0'(\p_\theta,\p_\sigma) > 0$ everywhere
on $\vV_j \times \AA$.  
\end{proof}

To find a symplectic extension of \eqref{eqn:sympC} over
$\Sigma \times \AA$, choose now a closed $1$-form $\kappa$ on~$\Sigma$
which takes the form
$$
\kappa = c_j\, d\theta
$$
near each boundary component~$\p_j\Sigma$; this is possible due to the
homological condition \eqref{eqn:homologicalcj}.  Then if
$\omega_0'$ denotes the extension of $\widetilde{\Psi}^*\omega_0$ given by
Lemma~\ref{lemma:extendingOmega0}, we extend \eqref{eqn:sympC} as
$$
\omega'_C := C\omega_0' + d\phi \wedge \kappa.
$$
Whenever $C$ is sufficiently large, Lemma~\ref{lemma:extendingOmega0}
implies that this form is also symplectic and restricts positively to the
surfaces $-\Sigma \times \{(\tau,\phi)\}$ and $\{p\} \times \AA$ if
$p \in \Sigma$ lies outside a neighborhood of the boundary.
This implies that it is positive on the pages of~$\boldsymbol{\pi}'$,
as well as on the core 
$\roundcore_\Sigma = ([0,1/2] \times \widehat{B}_0) \cup
(- \Sigma \times \{(0,0)\}) \subset W$
and co-core
$\roundcocore_\Sigma = \{p\} \times \AA \subset W$
(cf.~\S\ref{sec:handles}).

To summarize: we have constructed a smooth cobordism~$W$ with symplectic
form $\omega_C'$ that matches $\omega_C$ near $M = \{0\} \times M$ 
and is positive on the core and co-core and on the pages of the induced 
blown up summed open book
at the other boundary component~$M'$.  An appropriate confoliation $1$-form
$\lambda_0'$ can now be defined on~$M'$ by
\begin{equation}
\label{eqn:lambda0'}
\lambda_0' =
\begin{cases}
\lambda_0 & \text{ on $M \setminus \nN(\iI_0)$},\\
d\phi     & \text{ on $\Sigma_\pm \times S^1$},
\end{cases}
\end{equation}
where we use~$\phi$ to denote the natural $S^1$-coordinate on
$\Sigma_\pm \times S^1$.
The distribution $\xi_0' := \ker\lambda_0'$ is then tangent to the pages on the
glued in region, hence $\omega_C'|_{\xi_0'} > 0$.  It follows that
on any connected component of~$M'$ that does not contain closed pages,
$\xi_0'$ has a perturbation to a contact structure~$\xi'$ that is
supported by $\boldsymbol{\pi}'$ and dominated by~$\omega_C'$.

\begin{remark}
\label{remark:coreIsLarge}
It will be useful later to observe that $\int_{\roundcore_\Sigma} \omega_C'$
is not only positive but can be assumed to be arbitrarily large.  In fact
it \emph{must} in general be large due to the deformation trick used
in the proof of Lemma~\ref{lemma:extendingOmega0}.
\end{remark}

\begin{remark}
\label{remark:noKappa}
If the constants~$c_j$ all vanish, i.e.~$\Omega_0 = 0$ on $\nN(\iI_0)$,
then one can choose the $1$-form $\kappa$ in the above construction
to be identically zero.  This has the useful consequence that for any
$\tau \in [-1,1]$ and any
closed embedded loop $\ell \subset \Sigma$ outside a
neighborhood of~$\p\Sigma$, the torus $\ell \times \{\tau\} \times S^1
\subset \Sigma \times \AA$ is Lagrangian.  More generally, if
$\ell \subset \Sigma$ is any properly embedded compact $1$-dimensional
submanifold transverse to~$\p\Sigma$, then
$$
\int_{\ell \times \{\tau\} \times S^1} \omega'_C = 0.
$$
Indeed, with $\kappa=0$ it is equivalent to show that the 
integral of~$\omega'_0$ vanishes, and using \eqref{eqn:Omega'} we find
$$
\int_{\ell \times \{\tau\} \times S^1} \omega'_0 = 
\int_{\p\ell \times \{\tau\} \times S^1} \eta
$$
since~$\mu$ vanishes on the $S^1$-factor in
$\p\ell \times \{\tau\} \times S^1$.  Since
$\eta$ is $S^1$-invariant on $S^1 \times \AA$, this integral doesn't
depend on the position of any point in $\p\ell \subset \p\Sigma$ but
only on the algebraic count of these points, which is zero, thus
$$
\int_{\p\ell \times \{\tau\} \times S^1} \eta =
\#(\p\ell) \int_{\{(\theta,\tau)\} \times S^1} \eta = 0.
$$
\end{remark}

\begin{remark}
\label{remark:concave}
The reader who is only interested in \emph{strong} cobordisms, or
more generally the case where the negative boundary of the cobordism
is (strongly) concave, may assume throughout this section 
that $\Omega_0 \equiv 0$.
In this case, the symplectic form we have defined on~$W$ is exact near 
$M \subset \p W$ and has a primitive there
which restricts to a constant multiple of the contact form
$\lambda_\epsilon$, so this boundary component
is concave.  The contents of this and the next section
therefore suffice to complete the proofs of Theorems~\ref{thm:roundHandles} 
and~\ref{thm:2handles} respectively if the given $\omega$ 
on $[0,1] \times M$ is exact: indeed, by
\cite{Eliashberg:contactProperties}*{Proposition~3.1},
$\omega$ can then be deformed to make it (strongly)
convex at the positive boundary, so after a further deformation to match the
contact forms, the Liouville flow can be used to attach it smoothly to our model
as long as the constant $C > 0$ is chosen sufficiently large.
The case where $\omega$ is not exact requires the additional deformation
argument of \S\ref{sec:deformation} below.
\end{remark}

\subsection{Modifications for the capping cobordism}
\label{sec:notRound}

The above construction works essentially the same way for the handle
$\Sigma \times \DD$, so we will be content to briefly summarize the differences.
Here we pick binding components
$$
B_0 = \gamma_1 \cup \ldots \cup \gamma_N \subset B
$$
and denote the corresponding solid torus neighborhoods by
$\nN(\gamma_j) = S^1 \times \DD$ with coordinates $(\theta,\rho,\phi)$,
viewing $(\rho,\phi)$ as polar coordinates, and denote the union of these
neighborhoods by~$\nN(B_0)$.  The model symplectic form 
$\omega_C$ on the trivial cobordism $[0,1] \times M$
is again defined via \eqref{eqn:omegaSHS} and \eqref{eqn:omegaC}, with the
difference that since every closed $2$-form on $\nN(\gamma_j)$ is exact,
we can assume (after adding an exact $2$-form) that $\Omega_0$ \emph{vanishes}
on all of these neighborhoods.  The role of $H_{T_j}$ is now played
by a hypersurface
$$
H_{\gamma_j} \subset [0,1] \times \nN(\gamma_j)
$$
parametrized by an embedding
$$
\Psi : S^1 \times \DD \to [0,1] \times M,
$$
thus defining a similar set of coordinates 
$(\theta,\tau,\phi) \in S^1 \times \DD$ on~$H_{\gamma_j}$, where $(\tau,\phi)$ 
are now \emph{polar} coordinates on~$\DD$.  We can again arrange
$H_{\gamma_j}$ to be transverse to the vector field~$X_\theta$, defined
exactly as before, and then
use its flow to parametrize a neighborhood of~$H_{\gamma_j}$ in the 
region~$\overline{\uU}_{\gamma_j}$ that it bounds via a map
\begin{equation*}
\begin{split}
\widetilde{\Psi} : (1 - \delta,1] \times S^1 &\times \DD \hookrightarrow 
[0,1] \times \nN(\gamma_j) : 
(\sigma,\theta,\tau,\phi) \mapsto \varphi^{\sigma-1}_{X_\theta}(\Psi(\theta,\phi,\tau))
\end{split}
\end{equation*}
for which $\widetilde{\Psi}^*\omega_0$ again takes the form
$-d(\sigma\, d\theta) + d\eta$ for some $1$-form~$\eta$ on $S^1 \times \DD$
that satisfies
$$
\eta = \left[\varphi(\tau) + 1 \right] \, d\phi
$$
near $S^1 \times \p \DD$ and $d\theta \wedge d\eta > 0$ everywhere.
Denote the image of $\widetilde{\Psi}$ corresponding to each~$\gamma_j$
by~$\wW_j$, with negatively oriented coordinates $(\sigma,\theta,\tau,\phi) \in
(1-\delta,1] \times S^1 \times \DD$.
Writing the union of the regions $\uU_{\gamma_j}$ as $\uU_{B_0}$, 
the smooth cobordism is then defined by
$$
W = ( ([0,1] \times M) \setminus \uU_{B_0} ) \cup (-\Sigma \times \DD),
$$
where $\Sigma \times \DD$ is glued in by identifying $\vV_j \times \DD$
with $\wW_j$ so that the coordinates match.
This has boundary $\p W = M' \sqcup (-M)$, where
$M = \{0\} \times M$ and
$$
M' = ((\{1\} \times M) \setminus \nN(B_0)) \cup 
(-\Sigma \times S^1),
$$
hence the glued in region $\Sigma \times S^1$ carries the coordinates
$(\sigma,\theta,\phi)$ near its boundary.
Choosing a $1$-form $\eta_0$ on~$\DD$ that matches $\eta$ near $\p\DD$
and satisfies $d\eta_0 > 0$, the interpolation trick \eqref{eqn:Omega'} 
can again be used to deform~$\omega_0$ in a collar neighborhood of 
$\p\Sigma \times \DD$ so that it admits a symplectic extension over the 
rest of $\Sigma \times \DD$ in the form
$\omega_0' = -d\mu + d\eta_0$.  The resulting form
$\omega_C' = C\omega_0' + \Omega_0$ is symplectic everywhere on~$W$
and is also positive on the pages of~$\boldsymbol{\pi}'$ at~$M'$ if~$C$
is sufficiently large, as well as on the core
\begin{equation}
\label{eqn:coreSmooth}
\core_\Sigma = ([0,1/2] \times B_0) \cup
(-\Sigma \times \{0\}) \subset W
\end{equation}
and the co-core
$$
\cocore_\Sigma = \{p\} \times \DD \subset W
$$
for an appropriate choice of $p \in \Sigma$.
The confoliation $1$-form extends smoothly over $\Sigma \times S^1$
as $\lambda_0' = d\phi$, so that $\omega_C'$ is also positive on
$\xi_0' := \ker\lambda_0'$ and 
thus dominates any contact form obtained as a small perturbation.

\subsection{Symplectic deformation in a collar neighborhood}
\label{sec:deformation}

To apply the constructions of the previous sections in proving
Theorems~\ref{thm:2handles} and~\ref{thm:roundHandles} when the
given symplectic form $\omega$ on $[0,1] \times M$ dominating $\xi$ is non-exact, 
we must show that $\omega$ can be deformed away from $\{0\} \times M$ to 
reproduce the model
\begin{equation}
\label{eqn:omegaCagain}
\omega_C = C\, d\left( \varphi(t) \lambda_0 + \lambda\right) +
\Omega_0,
\end{equation}
where $\lambda := \lambda_\epsilon$ 
and $\lambda_0$ are $1$-forms as described at the
beginning of \S\ref{sec:modelCobordism},
$\varphi : [0,1] \to \RR$ is a smooth function with $\varphi' > 0$,
$\varphi(0)=0$ and $\| \varphi \|_{L^\infty}$ small,
$\Omega_0$ is some closed $2$-form on~$M$ in the appropriate cohomology
class, and $C > 0$ is a constant that we can assume to be as large as
necessary.  The following application of a standard Moser deformation
argument (cf.~\cite{NiederkruegerWendl}*{Lemma~2.3}) will be useful.

\begin{lemma}
\label{lemma:Moser}
Suppose $(W,\omega)$ is a symplectic $4$-manifold, $M$ is a closed oriented
$3$-manifold with an embedding $\Phi : M \hookrightarrow W$ and
$\lambda$ is a $1$-form on~$M$ that satisfies 
$\lambda \wedge \Phi^*\omega > 0$.  Then for sufficiently small
$\epsilon > 0$, $\Phi$ extends to an embedding
$$
\widetilde{\Phi} : (-\epsilon,\epsilon) \times M \hookrightarrow W
$$
such that $\widetilde{\Phi}(0,\cdot) = \Phi$ and 
$\widetilde{\Phi}^*\omega = d(t\lambda) + \Phi^*\omega$.
\end{lemma}
Observe that if $M \subset W$ is an oriented hypersurface in a symplectic 
$4$-manifold $(W,\omega)$ with a positive contact structure~$\xi$, then
$\omega$ dominates~$\xi$ if and only if it satisfies
$$
\lambda \wedge \omega|_{TM} > 0
$$
for every contact form~$\lambda$ on $(M,\xi)$.  Using the
obvious variants of Lemma~\ref{lemma:Moser} when the hypersurface is a
positive or negative boundary component of~$W$, we obtain the following
useful consequence:

\begin{lemma}
\label{lemma:gluing}
Suppose $(M,\xi)$ is a closed contact $3$-manifold and
$((-1,0] \times M, \omega_-)$ and $([0,1) \times M, \omega_+)$
are two symplectic manifolds such that the restrictions of
$\omega_-$ and $\omega_+$ to $\{0\} \times M$ define the same
$2$-form~$\Omega$ on~$M$, with $\Omega|_\xi > 0$.  Then for
any small $\epsilon > 0$,
$(-1,1) \times M$ admits a symplectic form which matches
$\omega_+$ on $[\epsilon,1) \times M$ and $\omega_-$ on
$(-1,-\epsilon] \times M$.
\end{lemma}

\begin{prop}
\label{prop:deformation}
Suppose $(M,\xi)$ is any closed contact $3$-manifold with contact
form~$\lambda$, $\lambda_0$ is a $1$-form on~$M$ satisfying
$\lambda_0 \wedge d\lambda > 0$,
$\omega$ is a symplectic form on $[0,1] \times M$ with
$\omega|_{\xi} > 0$, and $\Omega_0$ is a closed $2$-form on~$M$
with $[\Omega_0] = [\omega|_{TM}] \in H^2_\dR(M)$.
Then for any $\delta \in (0,1)$ and sufficiently large $C > 0$, 
there exists a symplectic form~$\omega'$
on $[0,1] \times M$ that matches~$\omega$ on a neighborhood
of $\{0\} \times M$ and takes the form \eqref{eqn:omegaCagain} on
$[\delta,1] \times M$.
\end{prop}
\begin{proof}
By Lemma~\ref{lemma:Moser} we can assume without loss of generality
that~$\omega$ has the form
$$
\omega = d(t \lambda) + \Omega,
$$
near $\{0\} \times M$, where $\Omega$ is the closed $2$-form on~$M$ 
defined as the restriction of $\omega$ to $\{0\} \times M$.  

The proof now proceeds in two steps, of which the first is to put the 
symplectic structure $\omega_C$ of
\eqref{eqn:omegaCagain} into a slightly simpler form via a 
coordinate change near $\{0\} \times M$.  Define the $1$-form
$$
\Lambda_0 = \varphi(t)\, \lambda_0 + \lambda
$$
on $[0,1] \times M$
and write $\omega_0 := d\Lambda_0$, 
so $\omega_C = C \omega_0 + \Omega_0$.
Let $V$ denote the vector field that is
$\omega_C$-dual to~$C\Lambda_0$, i.e.~$\omega_C(V,\cdot) = C\Lambda_0$.
For $C$ sufficiently large, $V$ is then a small perturbation of the
vector field that is $\omega_0$-dual to~$\Lambda_0$, which is a Liouville
(with respect to~$\omega_0$) vector field positively transverse to
$\{0\} \times M$ since $\Lambda_0|_{\{0\} \times M} = \lambda$ is contact.
Hence we may assume~$V$ is also positively transverse to $\{0\} \times M$
and use its flow $\varphi_V^t$ to define an embedding
$$
\psi : [0,\epsilon) \times M \hookrightarrow [0,1] \times M :
(t,m) \mapsto \varphi_V^t(m)
$$
for $\epsilon > 0$ sufficiently small.  If $X_\lambda$ denotes the Reeb
vector field determined by~$\lambda$, along $\{0\} \times M$ we then have
$$
\iota_{\p_t}(\psi^*\omega_C) = C\lambda
$$
and
$$
\psi^*\omega_C|_{T(\{0\} \times M)} = C\,d\lambda + \Omega_0.
$$
Hence $\psi^*\omega_C$ matches the symplectic form
$d(t\, C\lambda) + C\, d\lambda + \Omega_0$
pointwise at $\{0\} \times M$, and another Moser deformation argument
thus allows us to isotop the embedding~$\psi$ so that
$\psi^*\omega_C$ takes this form on some neighborhood of $\{0\} \times M$.
Equivalently, this means $\omega_C$ admits a deformation to a new
symplectic form $\omega_C'$ which takes the form
\begin{equation}
\label{eqn:whatever}
\omega_C' = d(t C\lambda) + C\, d\lambda + \Omega_0
\end{equation}
on an arbitrarily small neighborhood of $\{0\} \times M$ and matches
the original $\omega_C$ outside a slightly larger neighborhood.

For step two, we show that the given~$\omega$ can be deformed outside a
small neighborhood of $\{0\} \times M$ to a new symplectic form~$\omega'$
that matches \eqref{eqn:whatever} outside a slightly larger neighborhood.
Indeed, choose a constant $C' > 0$ large enough so that
$$
(C'\,d\lambda + \Omega_0)|_{\xi} > 0,
$$
and since $\Omega$ and $\Omega_0$ are cohomologous by assumption,
choose a $1$-form $\eta$ on~$M$ such that 
$C'\,d\lambda + \Omega_0 - \Omega = d\eta$.
For some $\delta > 0$ small, choose a cutoff function
$\beta(t)$ that equals~$0$ near $t=0$ and~$1$ near $t=\delta$, 
and define
$$
\omega' = d \left( f(t)\,\lambda \right) + \Omega
+ d\left( \beta(t)\,\eta \right),
$$
with $f : [0,\delta] \to [0,\infty)$ a smooth function satisfying
\begin{itemize}
\item $f(t) = t$ near $t=0$,
\item $f' > 0$,
\item $f(\delta) + C' = C(\delta + 1)$.
\end{itemize}
If $f$ is chosen to increase sufficiently fast, then $\omega'$ is
symplectic, and this can always be arranged if~$C > 0$ is made
sufficiently large.  This depends in particular on the fact that
the $2$-forms $\Omega$ and $C\, d\lambda + \Omega_0$ are both positive
on~$\xi$.  The restrictions of~$\omega'$
and $\omega_C'$ to the hypersurface $\{\delta\} \times M$ now match,
thus the two can be glued together smoothly by Lemma~\ref{lemma:gluing}.
\end{proof}

Combining Proposition~\ref{prop:deformation} with the cobordism
constructions of \S\ref{sec:modelCobordism} and \S\ref{sec:notRound}
completes the proofs of Theorems~\ref{thm:2handles} and~\ref{thm:roundHandles}.

\subsection{Cohomology}
\label{sec:cohomology}

We now prove Theorems~\ref{thm:2handleCohomology} 
and~\ref{thm:roundHandleCohomology} by characterizing the situations
in which~$\omega$ can be made exact on~$W$ or on~$M'_{\text{convex}}$.

Assume first that $(W,\omega)$ is a
$\Sigma$-capping cobordism $([0,1] \times M) \cup \handle_\Sigma$,
with $\handle_\Sigma = -\Sigma \times \DD$
attached along a neighborhood $\nN(B_0)$
of $B_0 = \gamma_1 \cup \ldots \cup \gamma_N$.  Write $\p W = M' \sqcup
(-M)$ and
$$
\Omega := \omega|_{TM},\qquad
\Omega' := \omega|_{TM'}.
$$
Due to \S\ref{sec:deformation}, we may assume without loss of 
generality that~$\Omega$ has the form
\begin{equation}
\label{eqn:Gaudy}
\Omega = C\, d\lambda + \Omega_0
\end{equation}
where $C > 0$ is arbitrarily large, $\lambda$ is the usual contact
form on~$M$ and $\Omega_0$ vanishes on~$\nN(B_0)$.
By Remark~\ref{remark:coreIsLarge}, we can also assume in the following
that $\int_{\core_\Sigma} \omega$ is arbitrarily large.

The decomposition of~$W$ into $[0,1] \times M$ and $\handle_\Sigma$,
which intersect at $\nN(B_0) \subset \{1\} \times M$, gives rise to
the Mayer-Vietoris sequence,
$$
\ldots \to H_2(\nN(B_0)) \to 
H_2(M) \oplus H_2(\handle_\Sigma) \to H_2(W) \to H_1(\nN(B_0))
     \to H_1(M) \oplus H_1(\handle_\Sigma) \to \ldots
$$
in which $H_2(\nN(B_0)) =  H_2(\handle_\Sigma)=0$, 
$H_1(\handle_\Sigma) = H_1(\Sigma)$ and
$H_1(\nN(B_0)) = H_1(B_0) = \ZZ^N$.  Thus there is an isomorphism
\begin{equation}
\begin{split}
\label{eqn:isoH2W}
H_2(W) \cong \im \big(H_2(M) \to & H_2(W)\big) \\ &\oplus 
\ker \big( H_1(\nN(B_0)) \to H_1(M) \oplus H_1(\handle_\Sigma) \big),
\end{split}
\end{equation}
in which the first summand is an isomorphic copy of $H_2(M)$.
Denote by $\iota^M : \nN(B_0) \hookrightarrow M$ and 
$\iota^\Sigma : \nN(B_0) \hookrightarrow \handle_\Sigma$ the natural inclusions.
Then $\iota^\Sigma_*([\gamma_j]) = [\p_j\Sigma] \in H_1(\Sigma) = 
H_1(\handle_\Sigma)$, so since~$\Sigma$ is connected, $\ker\iota^\Sigma_*$
is isomorphic to~$\ZZ$ and is generated by
$[\gamma_1] + \ldots + [\gamma_N]$.
It follows that the second summand in \eqref{eqn:isoH2W} consists of
all integer multiples of $[\gamma_1] + \ldots + [\gamma_N]$ 
which are also in $\ker \iota^M_*$, i.e.~it is isomorphic 
to~$\ZZ$ if $[\gamma_1] + \ldots + [\gamma_N]$ is torsion in $H_1(M)$, and is
otherwise trivial.  In the former case, let $k_0 \in \NN$ be the smallest
number for which $k_0([\gamma_1] + \ldots + [\gamma_N]) = 0 \in H_1(M)$,
and construct a cycle $A_{k_0} \in H_2(W)$ in the form
\begin{equation}
\label{eqn:Ak0}
A_{k_0} = C_M + k_0 [\core_\Sigma],
\end{equation}
where $C_M$  is any $2$-chain in $\{0\} \times M$ with
$\p C_M = k_0([\gamma_1] + \ldots + [\gamma_N])$ and $\core_\Sigma \subset W$
is the core \eqref{eqn:coreSmooth}.
The isomorphism \eqref{eqn:isoH2W} implies
that everything in $H_2(W)$ is an element of $H_2(M)$ plus an integer
multiple of \eqref{eqn:Ak0}.

Let $h$ denote a real $1$-cycle in $M \setminus \nN(B_0)$ such that
$[h] = \PD([\Omega]) \in H_1(M;\RR)$; note that this is always possible
since $\Omega$ is necessarily exact on $\nN(B_0)$.  The product
$[0,1] \times h$ then represents a relative homology class
in $H_2(W,\p W ; \RR)$.

\begin{prop}
There is a number $c > 0$ such that $\PD([\omega]) = [0,1] \times [h]
+ c [\cocore_\Sigma] \in H_2(W,\p W ; \RR)$.
\end{prop}
\begin{proof}
It suffices to show that for every $A \in H_2(W)$, the evaluation of~$\omega$
on~$A$ matches the intersection product
\begin{equation}
\label{eqn:intersection}
\int_A \omega = A \bullet \left( [0,1] \times [h] + c [\cocore_\Sigma] \right).
\end{equation}
For any $A \in \im(H_2(M) \to H_2(W))$ this is immediately clear since
$$
\int_A \omega = \int_A \Omega = A \bullet [h],
$$
where the latter is the intersection product in~$M$, and~$A$ does not
intersect anything in the handle.  By \eqref{eqn:isoH2W}, either the image of
$H_2(M) \to H_2(W)$ is the entirety of $H_2(W)$ or there is one more generator
$A_{k_0} = C_M + k_0[\core_\Sigma]$.  For the latter we have
$$
\int_{A_{k_0}} \omega = \int_{C_M} \Omega + k_0 \int_{\core_\Sigma} \omega
$$
and
$$
A_{k_0} \bullet \left( [0,1] \times [h] + c [\cocore_\Sigma] \right) =
C_M \bullet [h] + k_0 c,
$$
so \eqref{eqn:intersection} is satisfied if and only if
$$
c = \int_{\core_\Sigma} \omega + \frac{1}{k_0} \left( \int_{C_M} \Omega - 
C_M \bullet [h] \right).
$$
This is positive without loss of generality since
$\int_{\core_\Sigma} \omega$ was assumed to be arbitrarily large.
\end{proof}

The above argument also shows that if $\{0\} \times M \subset (W,\omega)$ 
is concave, then~$\omega$ can never be exact if
$[\gamma_1] + \ldots + [\gamma_N] \in H_1(M)$ is torsion, even without
assuming $\int_{\core_\Sigma} \omega$ to be arbitrarily large.  Indeed, 
in this case we have $\Omega = d\lambda$ for a contact form~$\lambda$ on
$(M,\xi)$, and $[h] = 0$, hence
$$
\int_{A_{k_0}} \omega = \int_{C_M} d\lambda + k_0 \int_{\core_\Sigma} \omega
= k_0 \sum_{j=1}^N \int_{\gamma_j} \lambda + k_0 \int_{\core_\Sigma} \omega > 0,
$$
and
$$
c = \int_{\core_\Sigma} \omega + \sum_{j=1}^N \int_{\gamma_j} \lambda > 0.
$$
On the other hand if $[\gamma_1] + \ldots + [\gamma_N] \in H_1(M)$ is
not torsion, then $H_2(M)$ generates everything in $H_2(W)$, so
$\int_A \omega$ always vanishes since~$\Omega$ is exact.  This proves
the first half of Theorem~\ref{thm:2handleCohomology}.

We also conclude from the above that if $\{0\} \times M \subset (W,\omega)$ is
concave, then there is a constant $c > 0$ such that 
$$
\PD([\Omega']) = c [\p \cocore_\Sigma] \in H_1(M' ; \RR),
$$
so the second half of the theorem is proved by showing that
$[\p \cocore_\Sigma] = 0 \in H_1(M' ; \RR)$ if and only if the stated
homological condition on $\gamma_1,\ldots,\gamma_N$ is satisfied.
Writing $M' = (M \setminus \nN(B_0)) \cup (-\Sigma \times S^1)$, we obtain
the Mayer-Vietoris sequence
$$
\ldots \to H_2(M') \to H_1(\p \nN(B_0)) \to H_1(M \setminus B_0) \oplus
H_1(\Sigma \times S^1) \to \ldots,
$$
where $H_1(\p\nN(B_0)) \cong \ZZ^{2N}$, with each component $\p\nN(\gamma_j)$
carrying the two distinguished generators $\mu_j,\lambda_j$ defined
in \S\ref{sec:intro}.
Denote the inclusions $\iota^M : \p\nN(B_0) \to M \setminus B_0$ and
$\iota^\Sigma : \p\nN(B_0) \to \Sigma \times S^1$.  Then
$\iota^\Sigma_*\lambda_j = [\p_j\Sigma \times \{*\}] \in H_1(\Sigma \times S^1)$ 
and $\iota^\Sigma_*\mu_j = [\{*\} \times S^1] \in H_1(\Sigma \times S^1)$,
so $\ker \iota^\Sigma_*$ consists of all classes of the form
$$
k \sum_{j=1}^N \lambda_j + \sum_{j=1}^N m_j \mu_j
$$
with $k, m_1,\ldots,m_N \in \ZZ$ and $\sum_j m_j = 0$.  Now,
$[\p\cocore_\Sigma]$ is represented by the cycle
$\{*\} \times S^1 \subset \Sigma \times S^1 \subset M'$, and it vanishes
in $H_1(M';\RR)$ if and only if
$$
A \bullet [\{*\} \times S^1] = 0
$$
for every $A \in H_2(M')$.  This is true if and only if the image of the
map $H_2(M') \to H_1(\p\nN(B_0))$ in the above sequence contains only
cycles of the form $\sum_j m_j\mu_j$.  In light of the above description
of $\ker \iota^\Sigma_*$, this is true if and only if
$$
k (\lambda_1 + \ldots + \lambda_N) \not\in \ker \iota^M_*
$$
for all $k \ne 0$.  This completes the proof of
Theorem~\ref{thm:2handleCohomology}.

The proof of Theorem~\ref{thm:roundHandleCohomology} proceeds similarly:
Assume $W = ([0,1] \times M) \cup \roundhandle_\Sigma$ is a
$\Sigma$-decoupling cobordism, with $\roundhandle_\Sigma = -\Sigma \times \AA$
attached along a neighborhood $\nN(\iI_0)$ of
$\iI_0 = T_1 \cup \ldots \cup T_N$, write $\p W = M' \sqcup (-M)$,
$\Omega := \omega|_{TM}$ and $\Omega' := \omega|_{TM'}$.  We again assume
that $\int_{\roundcore_\Sigma} \omega$ is arbitrarily large, and that
$\Omega$ takes the form of \eqref{eqn:Gaudy}, and we also impose
the extra condition
$$
\int_{T_j} \Omega = 0 \quad
\text{ for every component $T_j \subset \iI_0$.}
$$
In this case we can find a real $1$-cycle~$h$ in $M \setminus\nN(\iI_0)$
that represents $\PD([\Omega]) \in H_1(M;\RR)$.  Without changing the
cohomology class or the symplectic properties of~$\omega$, we can then also
assume that~$\Omega_0$ is supported in a tubular neighborhood of the
cycle~$h$.

Recall from \S\ref{sec:intro} that each oriented torus 
$T_j \subset \iI_0$ comes with a distinguished homology basis
$\{\mu_j,\lambda_j\} \subset H_1(T_j)$, where $\lambda_j$ is a boundary
component of a page and $\mu_j$ is represented by a Legendrian loop in~$T_j$.
This also gives rise to bases
$\{\mu_j^\pm,\lambda_j^\pm\}$ of $H_1(\p_\pm\nN(T_j))$, where
the orientation of $\mu_j^-$ is reversed compared with~$\mu_j$.
For $W = ([0,1] \times M) \cup \roundhandle_\Sigma$ we have the Mayer-Vietoris
sequence
$$
\ldots \to H_2(M) \oplus H_2(\roundhandle_\Sigma) \to H_2(W) \to H_1(\nN(\iI_0))
     \to H_1(M) \oplus H_1(\roundhandle_\Sigma) \to \ldots
$$
and resulting isomorphism
\begin{equation}
\begin{split}
\label{eqn:isoH2W'}
H_2(W) \cong \im \big(H_2(M) \oplus H_2(\roundhandle_\Sigma) \to & H_2(W)\big) 
\\ & \oplus
\ker \big( H_1(\nN(\iI_0)) \to H_1(M) \oplus H_1(\roundhandle_\Sigma) \big).
\end{split}
\end{equation}
Denote the generator of $H_1(\AA) = \ZZ$ by $[S^1]$, which can also naturally
be regarded as a primitive class in 
$H_1(\roundhandle_\Sigma) = H_1(\Sigma) \oplus H_1(\AA)$.  
Then writing the inclusions as $\iota^M : \nN(\iI_0) \hookrightarrow M$
and $\iota^\Sigma : \nN(\iI_0) \hookrightarrow \roundhandle_\Sigma$,
we have $\iota^\Sigma_*(\lambda_j) = [\p_j\Sigma \times \{*\}]$
and $\iota^\Sigma_*(\mu_j) = [S^1]$, hence $\ker \iota^\Sigma_*$ consists
of all classes of the form
$$
k \sum_{j=1}^N \lambda_j + \sum_{j=1}^N m_j \mu_j
$$
for $k, m_1,\ldots,m_N \in \ZZ$ with $\sum_j m_j = 0$.  For any
$\Gamma \in H_1(\iI_0)$ of this form which is also nullhomologous in~$M$,
we can form a cycle $A_\Gamma \in H_2(W)$ as follows.
First choose a $2$-chain $C_M$ in $\{0\} \times M$ with $\p C_M = \Gamma$.
Choose also a $1$-chain $\ell$ in~$\Sigma$ with boundary in~$\p\Sigma$ 
such that the $2$-chain $\ell \times \{*\} \times S^1$ 
in $\roundhandle_\Sigma$ has boundary
$$
\p(\ell \times \{*\} \times S^1) = -\sum_{j=1}^N m_j \mu_j,
$$
which is always possible since $\sum_j m_j = 0$.
We can represent $\ell$ by a properly immersed submanifold in~$\Sigma$
so that by Remark~\ref{remark:noKappa}, 
$\int_{\ell\times \{*\} \times S^1} \omega = 0$.
Now extend $\ell \times \{*\} \times S^1$ to a $2$-chain in~$W$
with boundary in $\{0\} \times M$ by attaching trivial cylinders over the
appropriate covers of Legendrian representatives of~$\mu_j$.  Since these
cylinders are Lagrangian, this construction yields an immersed submanifold
$L_\ell \subset W$ which satisfies
\begin{equation}
\label{eqn:Lell}
\int_{L_\ell} \omega = 0
\end{equation}
and $\p L_\ell \subset \iI_0 \subset \{0\} \times M$, with
$[\p L_\ell] = - \sum_j m_j \mu_j \in H_1(\iI_0)$.
We define~$A_\Gamma \in H_2(W)$ by
\begin{equation}
\label{eqn:AGamma}
A_\Gamma = C_M + [L_\ell] + k[\roundcore_\Sigma].
\end{equation}

\begin{prop}
There is a number $c > 0$ such that $\PD([\omega]) = [0,1] \times [h]
+ c [\roundcocore_\Sigma] \in H_2(W,\p W ; \RR)$.
\end{prop}
\begin{proof}
The goal is again to prove
\begin{equation}
\label{eqn:intersectionRound}
\int_A \omega = A \bullet \left( [0,1] \times [h] + c [\roundcocore_\Sigma] \right)
\end{equation}
for every $A \in H_2(W)$, and it is again immediate if
$A \in \im(H_2(M) \to H_2(W))$.  It is also clear for 
$A \in \im(H_2(\roundhandle_\Sigma) \to H_2(W))$, as $H_2(\roundhandle_\Sigma)$
is generated by classes of the form
$\ell' \times [S^1]$ for $\ell' \in H_1(\Sigma)$,
hence both sides of \eqref{eqn:intersectionRound} vanish
(see Remark~\ref{remark:noKappa}).

The rest of $H_2(W)$ is generated by classes of the form
$A_\Gamma$ defined in \eqref{eqn:AGamma}, for which
$$
\int_{A_\Gamma} \omega = \int_{C_M} \Omega + k \int_{\roundcore_\Sigma} \omega
$$
in light of \eqref{eqn:Lell}.  Similarly,
$L_\ell$ does not intersect either $[0,1] \times h$
or $\roundcocore_\Sigma$, thus
$$
A_\gamma \bullet \left( [0,1] \times [h] + c [\roundcocore_\Sigma] \right)
= C_M \bullet [h] + kc,
$$
and \eqref{eqn:intersectionRound} is thus satisfied if and only if
$$
c = \int_{\roundcore_\Sigma}\omega + \frac{1}{k} \left(
\int_{C_M} \Omega - C_M \bullet [h] \right),
$$
which is positive if $\int_{\roundcore_\Sigma} \omega$ is made
sufficiently large.  To see that this formula for~$c$ doesn't depend on
any choices,
observe that if~$\Omega$ is exact, then $h=0$ and $\Omega = C\, d\lambda$, so
\begin{equation*}
\int_{C_M} \Omega - C_M \bullet [h] = C \int_{\p C_M} \lambda
\end{equation*}
is proportional to~$k$, as the integral of~$\lambda$ vanishes on all
the meridians~$\mu_j$.  When $\Omega$ is not exact but equals
$C\, d\lambda + \Omega_0$ with $\Omega_0$ supported in a
tubular neighborhood of~$h$, we can find a real homology class 
$B \in H_2(M;\RR)$ with $B \bullet [h] = C_M \bullet [h]$ and thus define 
a real $2$-chain
$$
C_M' := C_M - B
$$
with $\p C_M' = \p C_M$ and $C_M' \bullet [h] = 0$.  Then up to relative
homology, $C_M'$ can be represented by a real linear combination of immersed
surfaces that have no geometric intersection with~$h$, hence
$\int_{C_M'} \Omega_0 = 0$.  Now since $\int_B \Omega = B \bullet [h]$,
$$
\int_{C_M} \Omega - C_M \bullet [h] = \int_{C_M'} \Omega =
C \int_{C_M'} d\lambda = C \int_{\p C_M} \lambda,
$$
and this is again proportional to~$k$.
\end{proof}

If $\{0\} \times M \subset (W,\omega)$ is concave, then writing
$h=0$ and $\Omega = d\lambda$ gives
$$
\int_{A_\Gamma} \omega = \int_\Gamma \lambda + k \int_{\roundcore_\Sigma} \omega
$$
for any cycle $\Gamma = k(\lambda_1 + \ldots + \lambda_N) +
\sum_j m_j\mu_j \in H_1(\iI_0)$ with $\sum_j m_j = 0$ that is 
nullhomologous in~$M$.  Since $\int_\Gamma \lambda$ is also positively
proportional to~$k$, this proves that~$\omega$ is exact if and only if
there is no such nullhomologous cycle~$\Gamma$ with $k > 0$.
Moreover, $\PD([\Omega']) = c [\p\roundcocore_\Sigma] \in H_1(M';\RR)$
for some $c > 0$, so it remains to characterize the situations where this
homology class vanishes.  Write $M' = (M \setminus \nN(\iI_0)) 
\cup (-\Sigma_+ \times S^1) \cup (-\Sigma_- \times S^1)$ and consider
the resulting Mayer-Vietoris sequence
$$
\ldots \to H_2(M') \to H_1(\p \nN(\iI_0)) \to H_1(M \setminus \iI_0) \oplus
H_1((\Sigma_+ \sqcup \Sigma_-) \times S^1) \to \ldots,
$$
where $H_1(\p\nN(\iI_0))$ is freely generated by the~$4N$ cycles
$\mu_j^\pm, \lambda_j^\pm$.
Denote the inclusions $\iota^M : \p\nN(\iI_0) \to M \setminus \iI_0$ and
$\iota^\Sigma : \p\nN(\iI_0) \to (\Sigma_+ \sqcup \Sigma_-) \times S^1$,
where the latter maps $\p_\pm\nN(T_j)$ into $\Sigma_\pm \times S^1$.  Then
$\iota^\Sigma_*\lambda_j^\pm = [\p_j\Sigma \times \{*\}] \in 
H_1(\Sigma_\pm \times S^1)$ 
and $\iota^\Sigma_*\mu_j^\pm = [\{*\} \times S^1] \in 
H_1(\Sigma_\pm \times S^1)$,
thus $\ker \iota^\Sigma_*$ consists of all classes of the form
\begin{equation}
\label{eqn:theKernel}
k_+ \sum_{j=1}^N \lambda_j^+ + k_- \sum_{j=1}^N \lambda_j^- +
\sum_{j=1}^N m_j^+ \mu_j^+ + \sum_{j=1}^N m_j^- \mu_j^-,
\end{equation}
with $k_\pm, m_1^\pm,\ldots,m_N^\pm \in \ZZ$ satisfying 
$\sum_j m_j^+ = \sum_j m_j^- = 0$.  The co-core $\roundcocore_\Sigma$ has
two boundary components, one generating each of the cycles
$\{*\} \times S^1 \subset \Sigma_\pm \times S^1 \subset M'$, which we will
denote by $[S^1]_\pm \in H_1(M')$.  Thus $[\p\roundcocore_\Sigma]$ vanishes
in $H_1(M';\RR)$ if and only if
$$
A \bullet ([S^1]_+ + [S^1]_-) = 0
$$
for every $A \in H_2(M')$.  This is true if and only if the image of the
map $H_2(M') \to H_1(\p\nN(\iI_0))$ in the above sequence contains only
cycles of the form \eqref{eqn:theKernel} with $k_+ + k_- = 0$,
meaning that cycles of this form with $k_+ + k_- \ne 0$ are never
trivial in $H_1(M \setminus \iI_0)$.

We've now characterized the cases in which $\Omega'$ is globally 
exact on~$M'$; of course this never happens if $M'_{\text{flat}} \ne \emptyset$
since the latter then contains closed pages on which~$\Omega'$ is positive.
If both $M'_{\text{convex}}$ and~$M'_{\text{flat}}$ are nonempty, then
the interesting question is when $\Omega'$ will be exact on 
$M'_{\text{convex}}$, which is the case if and only if
$$
[\p\roundcocore_\Sigma \cap M'_{\text{convex}}] = 0 \in 
H_1(M'_{\text{convex}} ; \RR).
$$
Assuming the labels are chosen so that $\Sigma_+ \times S^1 \subset
M'_{\text{convex}}$ and $\Sigma_- \times S^1 \subset M'_{\text{flat}}$,
$[\p\roundcocore_\Sigma \cap M'_{\text{convex}}]$ is now represented
by $\{*\} \times S^1 \subset \Sigma_+ \times S^1$, and a repeat of the
above argument shows that this cycle vanishes if and only if
$H_1(\p \nN(\iI_0))$ contains no cycle of the form \eqref{eqn:theKernel}
which vanishes in $H_1(M \setminus \iI_0)$ and has $k_+ \ne 0$.
We observe however that in this situation, $\iI_0$ must separate~$M$
so that each $\p_+\nN(T_j)$ lies in a different connected component
of $M \setminus \iI_0$ from each $\p_-\nN(T_i)$, hence a cycle of the
form \eqref{eqn:theKernel} vanishes in $H_1(M \setminus \iI_0)$ if and
only if both the plus and minus parts vanish.  Our condition is thus
reduced to the nonexistence of a cycle
$$
k \sum_{j=1}^N \lambda_j^+ + \sum_{j=1}^N m_j \mu_j^+
$$
with $k \ne 0$ that vanishes in $H_1(M \setminus \iI_0)$.

\subsection{Proofs of the results from \S\ref{sec:results}}
\label{sec:proofs}

\begin{proof}[Proofs of Theorems~\ref{thm:overtwisted} and~\ref{thm:weak}]
To prove Theorem~\ref{thm:weak}, suppose $(M,\xi)$ contains an
$\Omega$-separating planar $k$-torsion domain $M_0$ for some closed $2$-form
$\Omega$ with $\Omega|_\xi > 0$ and an integer $k \ge 1$. 
Then $\int_T \Omega = 0$ for every interface torus~$T$
in~$M_0$ that lies in the planar piece, so we are free to remove any such
torus by attaching a $\DD$-decoupling cobordism whose symplectic
structure matches~$\Omega$ at~$M$.  By
Proposition~\ref{prop:kTok-1}, one can find a binding component~$\gamma$
or interface torus~$T$ such that if~$(W,\omega)$ with $\p W = M' \sqcup (-M)$ 
denotes the result of
attaching the corresponding $\DD$-capping or $\DD$-decoupling cobordism
respectively, then $M'$ contains a planar torsion domain of order either
$k-1$ or $k-2$.  Writing $\Omega' := \omega|_{TM'}$, the latter is also
$\Omega'$-separating since near each of the remaining interface tori,
which lie outside the region of surgery, 
$\omega$ is still cohomologous to the original~$\Omega$.  The process
can therefore be repeated until the manifold at the top has
planar $0$-torsion, meaning it is overtwisted.

Theorem~\ref{thm:overtwisted} is essentially the special case of
Theorem~\ref{thm:weak} for which we assume~$\Omega$ is exact to start with,
except that the above argument actually gives a \emph{weak} symplectic
cobordism $(W,\omega)$ from $(M,\xi)$ to some overtwisted
$(M_\OT,\xi_\OT)$, where we can assume $M \subset (W,\omega)$ is concave
and $M_\OT \subset (W,\omega)$ is not necessarily convex, but
$\omega|_{\xi_\OT} > 0$.  This can now be turned into a strong cobordism
by the following trick which was suggested to me by David Gay:
first, observe that if $M_\OT$ is a rational homology sphere, then
$\omega$ is exact near~$M_\OT$ and can thus be deformed to make~$M_\OT$ 
convex using the argument of Eliashberg 
in \cite{Eliashberg:contactProperties}*{Proposition~3.1}.  
Otherwise, take any knot $K \subset M_\OT$ that is nontrivial
in $H_1(M_\OT ; \QQ)$, and isotop it if necessary so that it is
disjoint from some overtwisted disk.  Then after a $C^0$-small perturbation
to make~$K$ Legendrian, one can attach a symplectic $2$-handle along~$K$
so that the new positive boundary becomes an overtwisted
contact manifold $(M'_\OT,\xi'_\OT)$ with
$\dim H_1(M'_\OT ; \QQ) = \dim H_1(M_\OT ; \QQ) - 1$,
see Lemma~\ref{lemma:DehnHomology} below.  
Repeating this process
enough times, the positive boundary eventually becomes an overtwisted 
rational homology sphere, so that the weak cobordism can be deformed to a 
strong one.
\end{proof}

The final step in the above proof is justified by the following
simple homological lemma:
\begin{lemma}
\label{lemma:DehnHomology}
Suppose $M$ is a closed oriented $3$-manifold, $K \subset M$ is a knot
with $[K] \ne 0 \in H_1(M;\QQ)$ and $M'$ is the result of performing
Dehn surgery along $K$ with any framing.  Then
$\dim H_1(M' ; \QQ) = \dim H_1(M ; \QQ) - 1$.
\end{lemma}
\begin{proof}
As preparation, suppose $K$ is any knot in a closed oriented
$3$-manifold $M$, denote a neighborhood of $K$ in~$M$ by $N_K$
and let $(\mu,\lambda)$ denote any basis of $H_1(\p N_K)$ such that
$\mu$ is a meridian.  If $\iota : \p N_K \hookrightarrow M\setminus K$
denotes the inclusion, we claim that
$$
\dim H_1(M ; \QQ) = \begin{cases}
\dim H_1(M \setminus K ; \QQ) & \text{ if $\iota_*\mu = 0 \in H_1(M \setminus K ; \QQ)$}, \\
\dim H_1(M \setminus K ; \QQ) - 1 & \text{ if $\iota_*\mu \ne 0 \in H_1(M\setminus K; \QQ)$}.
\end{cases}
$$
This follows from the Mayer-Vietoris sequence for $M = N_K \cup (M \setminus K)$:
$$
\ldots H_2(M ; \QQ) \to H_1(\p N_K ; \QQ) \stackrel{\Phi}{\to} H_1(N_K ; \QQ) \oplus
H_1( M\setminus K; \QQ) \to H_1(M ; \QQ) \to H_0(\p N_K; \QQ) \ldots
$$
Since any $1$-cycle in $M$ can be disjoined from $N_K$, the map $H_1(M ; \QQ) \to
H_0(\p N_K ; \QQ)$ in this sequence is trivial, thus exactness implies
$$
H_1(M ; \QQ) \cong \left( H_1(N_K;\QQ) \oplus H_1(M\setminus K; \QQ) \right) \Big/
\im \Phi.
$$
The map $\Phi : H_1(\p N_K ; \QQ) \to H_1(N_K ; \QQ) \oplus H_1(M\setminus K; \QQ)$ is
nontrivial since $\lambda$ maps to the generator of $H_1(N_K ; \QQ) = \QQ$.
Since $\mu$ maps to $0$ in $H_1(N_K ; \QQ)$, $\im \Phi$ is
thus $1$-dimensional if and only if $\iota_*\mu = 0 \in H_1(M\setminus K; \QQ)$,
and is otherwise $2$-dimensional.  This proves the claim.

Now assume $[K] \ne 0 \in H_1(M ; \QQ)$, and we are given a framing such
that $\lambda$ is the preferred longitude.  This implies immediately that
$\iota_*\lambda \ne 0 \in H_1(M\setminus K; \QQ)$.  Likewise, 
$\iota_*\mu = 0 \in H_1(M\setminus K; \QQ)$: to see this, note that by the
nondegeneracy of the intersection form, there exists a $2$-cycle $C$ in~$M$
such that $[C] \bullet [K] = 1$, hence the restriction of $C$ to
the complement of $N_K$ defines a chain whose boundary is~$\mu$; alternatively,
one can also derive this from the exact sequence above by considering the
image of $[C]$ under $H_2(M;\QQ) \to H_1(\p N_K; \QQ)$.  We therefore have
$\dim H_1(M ; \QQ) = \dim H_1(M\setminus K ; \QQ)$ by the claim above.
If $M'$ is now defined
by gluing another solid torus into $M\setminus N_K$ such that $\lambda$
becomes the meridian, then the claim is again applicable and
implies $\dim H_1(M' ; \QQ) = \dim H_1(M\setminus K ; \QQ) - 1$.
\end{proof}

\begin{proof}[Proof of Theorem~\ref{thm:caps}]
Suppose $(M,\xi)$ contains an $\Omega$-separating partially planar
domain $M_0 \subset M$ with planar piece $M_0^P \subset M_0$, where
$\Omega$ is a closed $2$-form on~$M$ satisfying $\Omega|_\xi > 0$.
Then for every binding circle or interface torus in~$M_0^P$, we can
attach $\DD$-capping or $\DD$-decoupling cobordisms to produce a
symplectic manifold $(W,\omega)$ with $\p W = M' \sqcup (-M)$,
$\omega|_{TM} = \Omega$ and
$$
M' = M'_{\text{flat}} \sqcup M'_{\text{convex}},
$$
where $M'_{\text{flat}}$ contains a component that 
is a symplectic $S^2$-fibration over $S^1$,
and $M'_{\text{convex}}$ carries a contact structure~$\xi'$ with
$\omega|_{\xi'} > 0$.  The desired cap is then obtained from $(W,\omega)$
after capping $M'_{\text{flat}} \sqcup M'_{\text{convex}}$ via
\cite{Eliashberg:cap} or \cite{Etnyre:fillings}.
\end{proof}

\begin{proof}[Proof of Theorem~\ref{thm:toS3}]
Note that since $H^2_\dR(S^3) = 0$, any weak cobordism from $(M,\xi)$ to
$(S^3,\xi_0)$ that is concave at~$M$ can be deformed to a strong
cobordism, so it suffices to prove that $(S^3,\xi_0)$ can be obtained
from $(M,\xi)$ by a finite sequence of (generally weak) capping and
decoupling cobordisms.

Suppose $M_0 \subset M$ is a partially planar domain.  If it is also
a planar torsion domain then the result already follows from
Corollary~\ref{cor:equivalence}, thus assume not.  If~$M_0$ has only
one irreducible subdomain with nonempty binding, we can remove binding
components by $\DD$-capping cobordisms and interface tori by
$\DD$-decoupling cobordisms until the planar piece has exactly one
binding component left and no interface or boundary, which means
it is the tight~$S^3$.  The desired cobordism can then be obtained by
capping any additional components that may remain at the end of this
process.

If~$M_0$ has more than one irreducible subdomain but does not have planar
torsion, then it must be symmetric (see Definition~\ref{defn:symmetric}).
This means that $M=M_0$, the binding and boundary are empty and the
interface tori divide~$M$ into exactly two irreducible subdomains that have
diffeomorphic planar pages.  Then we can remove interface tori by
$\DD$-decoupling cobordisms until exactly one remains, and the resulting
contact manifold is the tight $S^1\times S^2$.  The latter is cobordant
to~$S^3$ by a $\DD$-capping cobordism that removes one binding component
from the supporting open book with cylindrical pages and trivial monodromy.

Theorem~\ref{thm:planar} follows essentially 
by the same argument since every planar
open book is also a fully twisted partially planar domain.  We only need
to add that the topological construction of the cobordism by attaching
$2$-handles along binding components does not depend
on the choice of $2$-form $\Omega$ on~$M$, which after the deformation
carried out in \S\ref{sec:deformation}, always looks the same on a large
tubular neighborhood of the binding.
\end{proof}

\subsection{Holomorphic curves}
\label{sec:holomorphic}

For applications to Embedded Contact Homology and Symplectic Field Theory
among other things, it may be quite helpful to observe that
the cobordism $(W,\omega)$ generally admits not
only a symplectic structure but also a foliation by $J$-holomorphic curves.
We don't plan to pursue this here in full detail, but we shall give a sketch of
the general picture.  For simplicity, we consider only the 
$\Sigma$-decoupling cobordism along $\iI_0 = T_1 \cup \ldots \cup T_N$
in the case where
the negative boundary component $(M,\xi)$ is concave, so we can 
arrange~$\omega$ near $\{0\} \times M$ to have the form
\begin{equation}
\label{eqn:symplectization}
\omega = d\left( \varphi(t) \lambda_0 + \lambda\right)
\end{equation}
as in \S\ref{sec:modelCobordism}, where $\lambda_0$ and $\lambda$ are the
confoliation $1$-form and contact form respectively that were constructed
in \S\ref{sec:model}.  These have the following convenient properties:
\begin{itemize}
\item $\lambda_0 \wedge d\lambda > 0$
\item $\ker d\lambda \subset \ker d\lambda_0$
\end{itemize}
Together with the obvious fact that $d\lambda$ is closed, these properties
mean that the pair $(\lambda_0,d\lambda)$ is a \defin{stable Hamiltonian
structure}, to which we associate the co-oriented distribution
$\xi_0 = \ker\lambda_0$ and positively transverse
vector field~$X_0$ on~$M$ such that
$$
\lambda_0(X_0) \equiv 1, \qquad d\lambda(X_0,\cdot) \equiv 0.
$$
Similarly, writing $\Omega' := \omega|_{TM'}$ and recalling the
confoliation $1$-form $\lambda_0'$ defined on~$M'$ by \eqref{eqn:lambda0'},
the pair
$(\lambda_0',\Omega')$ forms a stable Hamiltonian structure on~$M'$, and
we define the corresponding distribution $\xi_0' = \ker\lambda_0'$ and
vector field $X_0'$ on~$M'$ such that
$$
\lambda_0'(X_0') \equiv 1, \qquad \Omega'(X_0',\cdot) \equiv 0.
$$
In fact, on components of $M'$ where the pages are not closed, one
can show with a little more effort that $\lambda_0'$ admits
a perturbation to a contact form $\lambda'$ on~$M'$ such that
$\xi' := \ker\lambda'$ is dominated by $\Omega'$, where
$$
\Omega' = F\, d\lambda'
$$
for some smooth function $F : M' \to (0,\infty)$ that satisfies
$dF \wedge d\lambda' \equiv 0$ and is constant outside a neighborhood
of the boundary of the co-core, $\p\roundcocore_\Sigma \subset M'$, hence
$X_0'$ is colinear with the Reeb vector field determined by~$\lambda'$.
There is now a collar neighborhood $(-\epsilon,0] \times M' \subset W$ of~$M'$ 
on which $\omega$ takes the form
$d(t \lambda_0') + \Omega'$, thus we can attach positive and negative
cylindrical ends to define the \defin{completion} of~$(W,\omega)$,
$$
W^\infty := W \cup \left( (-\infty,0] \times M \right) \cup
\left( [0,\infty) \times M' \right)
$$
and extend~$\omega$ symplectically so that it takes the form
$d(\varphi(t)\, \lambda_0) + d\lambda$ on $(-\infty,0] \times M$ and
$d(\psi(t)\, \lambda_0') + \Omega'$ on $[0,\infty) \times M'$ for
suitable choices of functions $\varphi$ and~$\psi$.  Equivalently,
$(W^\infty,\omega)$ can be constructed directly from the symplectization
of $(M,\xi)$ as follows: extend the function $\varphi(t)$ to~$\RR$ so that
\eqref{eqn:symplectization} defines a symplectic form on $\RR\times M$,
and extend the ``hole'' $\uU_{\iI_0}$ defined in
\S\ref{sec:modelCobordism} to a hole in
$\RR\times M$ by including the interior of the 
region $(1,\infty) \times \nN(\iI_0)$; let
$$
\uU_{\iI_0}^\infty \subset \RR\times M
$$
denote the extended hole.  Then $(W^\infty,\omega)$ can be obtained
by removing $\uU_{\iI_0}^\infty$ from $(\RR\times M,d(\varphi(t)\lambda_0 +
\lambda))$ and replacing it with the completion of the round handle,
$$
\roundhandle_\Sigma^\infty := \Sigma \times (-\infty,\infty) \times S^1,
$$
extending the symplectic form over $\roundhandle_\Sigma^\infty$ by an
adaptation of the argument in \S\ref{sec:modelCobordism}.

An almost complex structure $J$ on $(W^\infty,\omega)$ is now
\defin{admissible} if it is $\omega$-compatible on~$W$ and is compatible
with the stable Hamiltonian structures on both cylindrical
ends, meaning it is $\RR$-invariant, restricts to a complex structure 
on the respective distributions $\xi_0$ and $\xi_0'$ defining the 
correct orientations, and
maps the unit vector $\p_t$ in the $\RR$-direction to the vector field
$X_0$ or $X_0'$.

It was shown in \cite{Wendl:openbook2} that for a
suitable choice of almost complex structure $J_0$ on
$\RR \times M$ compatible with $(\lambda_0,d\lambda)$, the pages of the
blown up summed open book~$\boldsymbol{\pi}$ in~$M_0$ admit lifts to embedded
$J_0$-holomorphic curves in $\RR \times M$ which match the fibers of the
mapping torus $S_\psi$ outside of the neighborhoods $\nN_i$ of
$B \cup \iI \cup \p M_0$, have positive cylindrical
ends approaching closed orbits of~$X_0$ in $B \cup \iI \cup \p M_0$ and
satisfy a suitable finite energy condition.  We can now define an
admissible almost complex structure on $(W^\infty,\omega)$ which
matches $J_0$ outside of $\uU_{\iI_0}^\infty$ and is $\omega$-compatible
on $\roundhandle_\Sigma^\infty$.  The $J_0$-holomorphic curves in
$(\RR\times M) \setminus \uU_{\iI_0}^\infty$ can be extended into
$\roundhandle_\Sigma^\infty$ as symplectic surfaces that are
diffeomorphic to~$\Sigma$ and foliate $\roundhandle_\Sigma^\infty$, thus
we can extend~$J_0$ into the handle so that it is $\omega$-compatible
and these surfaces become $J_0$-holomorphic.  In doing this,
we can also make the natural completion of the core 
$$
\roundcore_{\Sigma,\infty} := \roundcore_\Sigma \cup_{\widehat{B}_0} 
((-\infty,0] \times \widehat{B}_0)
$$
$J_0$-holomorphic, as well as its translations under the
$S^1$-action by translating the local $\phi$-coordinates, and the completion
of the co-core
$$
\roundcocore_{\Sigma,\infty} := \roundcocore_\Sigma
\cup_{\p\roundcocore_\Sigma} ( [1,\infty) \times \p\roundcocore_\Sigma).
$$
The result is a foliation of $W^\infty$ (or at least the
region outside of $\RR\times (M \setminus M_0)$) by
finite energy $J_0$-holomorphic curves.  We summarize this construction
as follows (see Figure~\ref{fig:Jhol}).

\begin{prop}
\label{prop:Jhol}
One can choose an admissible almost complex structure $J_0$ on the 
completion $(W^\infty,\omega)$ of a $\Sigma$-decoupling cobordism
$(W,\omega)$, such that there exists a foliation $\fF$ by embedded
$J_0$-holomorphic curves with the following properties:
\begin{enumerate}
\item In each cylindrical end, the leaves of~$\fF$ match the holomorphic
lifts of the pages of $\boldsymbol{\pi}$ and $\boldsymbol{\pi}'$
constructed in \cite{Wendl:openbook2}.
\item The completed core $\roundcore_{\Sigma,\infty}$ and all its
$S^1$-translations are leaves of~$\fF$.
\item The trivial cylinders over periodic orbits in $B$, $\p M_0$ and
$\iI \setminus \iI_0$ are all leaves of~$\fF$.
\item All other leaves of~$\fF$ have only positive cylindrical ends
asymptotic to orbits in $B \cup (\iI \setminus \iI_0) \cup \p M$, and they
are homotopic in the moduli space to the holomorphic lifts of the pages
of~$\boldsymbol{\pi}'$ in $[1,\infty) \times M'$.
\item The completed co-core $\roundcocore_{\Sigma,\infty}$ is also
$J_0$-holomorphic and intersects the leaves of~$\fF$ transversely.
\end{enumerate}
\end{prop}

\begin{figure}
\begin{center}
\includegraphics[width=5in]{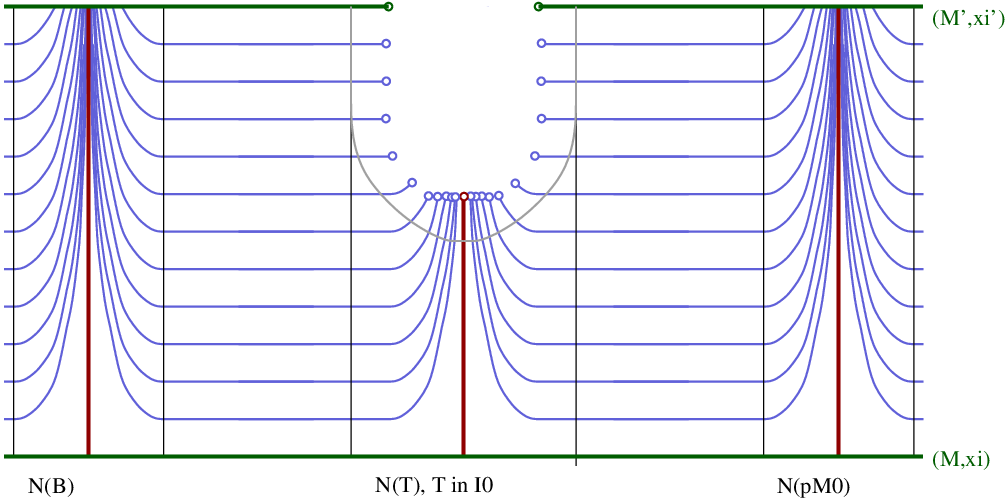}
\caption{\label{fig:Jhol}
The $J_0$-holomorphic foliation described by Proposition~\ref{prop:Jhol}
for a $\Sigma$-decoupling cobordism, not including the cylindrical ends.  
The picture shows the~$t$ and
$\rho$~coordinates near various binding, interface and boundary components, 
including an interface torus~$T$ where a handle $\Sigma \times \AA$ 
has been attached.
The circles at the ends of leaves in this region represent capping by~$\Sigma$.
The core $\roundcore_{\Sigma,\infty}$ is shown as the vertical leaf directly
in the center, which emerges from $-\infty$ and is capped off in the handle.}
\end{center}
\end{figure}

In considering the behavior of holomorphic curve invariants under symplectic
cobordisms, a special role is typically played by curves that have no
positive ends---such curves can only exist in non-exact cobordisms.
One useful application of the foliation constructed above is that we
can now characterize all such curves precisely:

\begin{prop}
\label{prop:noPositive}
Suppose $u : \dot{\Sigma} \to W^\infty$ is a finite energy
$J_0$-holomorphic curve that is connected, somewhere injective and has
no positive ends.  Then~$u$ is a leaf of~$\fF$, specifically it is an 
$S^1$-translation of the core $\roundcore_{\Sigma,\infty}$.
\end{prop}
\begin{proof}
There are no curves without positive ends outside the region of surgery
since here the symplectic form is exact,
thus we may assume~$u$ intersects both the handle and its complement.
If~$u$ is a leaf of~$\fF$ then it must be an $S^1$-translation of the
core, as all other leaves have positive ends.  If it is not a leaf
of~$\fF$ then it has a positive intersection with some leaf~$v$, and
without loss of generality we may suppose that~$v$ has only positive
ends.  Then~$v$ is homotopic in the moduli space to a holomorphic lift
of a page of~$\boldsymbol{\pi}'$, which we may assume lives in the
region $[c,\infty) \times M'$ for an arbitrarily large number $c > 0$
and thus cannot intersect~$u$.  This is a contradiction, due to
positivity of intersections.
\end{proof}

To apply these constructions to ECH, or to Symplectic Field Theory
for that matter, one must perturb $\xi_0$ and $\xi_0'$ to contact
structures and perturb~$J_0$ with them.  The
$J_0$-holomorphic leaves of~$\fF$ will not generally behave well under
this perturbation: a leaf with only positive ends for instance,
if it has genus~$g$, will have Fredholm index $2 - 2g$ and thus disappears
under any generic perturbation of the data unless $g=0$.
Proposition~\ref{prop:noPositive}, however, should still hold under a
sufficiently small perturbation, because for any sequence $J_k$ of
perturbed almost complex structures converging to~$J_0$, a sequence
of $J_k$-holomorphic curves should converge in the sense of
\cite{SFTcompactness} to a $J_0$-holomorphic building, and
Proposition~\ref{prop:noPositive} determines what this building can
look like.  This is a variation on the uniqueness argument used in
\cite{Wendl:openbook2} to prove vanishing of the ECH contact invariant:
higher genus holomorphic curves do not generically exist, but they remain
useful for proving that no other curves can exist either.

\end{document}